\documentclass[12pt,reqno]{amsart}
\usepackage{amssymb,dsfont}
\usepackage{eucal}
\usepackage{amsmath}
\usepackage{amscd}
\usepackage[dvips]{color}
\usepackage{multicol}
\usepackage[all]{xy}           
\usepackage{graphicx}
\usepackage{color}
\usepackage{colordvi}
\usepackage{xspace}
\usepackage{txfonts}
\usepackage{lscape}
\usepackage{tikz} 

\def\supp{\mathrm{supp}}

\numberwithin{equation}{section}
\usepackage[shortlabels]{enumitem}
\usepackage{ifpdf}
\ifpdf
\usepackage[colorlinks,final,hyperindex]{hyperref}
\else
\fi
\usepackage{tikz}
\usepackage[active]{srcltx}

\makeatletter

\newtheorem{theo}{Theorem}[section]
\newtheorem{pro}[theo]{Proposition}
\newtheorem{lem}[theo]{Lemma}
\newtheorem{cor}[theo]{Corollary}
\newtheorem{defi}[theo]{Definition}
\newtheorem{exa}[theo]{Example}

\theoremstyle{definition}
\newtheorem{rem}[theo]{Remark}

\def\normOrd#1{\mathop{:}\nolimits\!#1\!\mathop{:}\nolimits}
\def\Vir{\mathfrak{Vir}}

\def\de{\delta}

\def\c{{\bf c}}

\newcommand{\M}{\mathbb{M}}
\newcommand{\R}{{\mathcal R}}
\newcommand{\N}{{\mathbb N}}

\newcommand{\Z}{{\mathcal Z}}

\newcommand{\mi}{\mathbf{i}}
\newcommand{\mk}{\mathbf{k}}

\def\bN{{\mathbb N}}
\def\bZ{{\mathbb Z}}

\def\bR{{\mathbb R}}

\def\bC{{\mathbb C}}

\def\z{\mathbb Z}
\def\cc{\mathbb C}

\def\mh{\mathfrak{H}}

\def\mg{\mathfrak{L}}

\def\beq{\begin{eqnarray}}
\def\eeq{\end{eqnarray}}
\def\beqs{\begin{eqnarray*}}
\def\eeqs{\end{eqnarray*}}

\def\bfz{{\mathbf{0}}}

\def\bfi{{\mathbf{i}}}
\def\bfj{{\mathbf{j}}}
\def\bfk{{\mathbf{k}}}
\def\bfl{{\mathbf{l}}}

\def\h{\mathfrak{h}}
\def\GG{\mathcal{G}}

\def\UU{\mathcal{U}}
\def\DD{\mathcal{L}}
\def\mh{\mathcal{H}}
\def\DD{\mathfrak{D}}
\def\RR{\mathcal{R}}

\def\ad{\mbox{\rm ad}}

\def\Ind{\text{\rm Ind}}

\newcommand{\nc}{\newcommand}
\nc{\km}[1]{\textcolor{red}{Kaiming: #1}}
\nc{\yf}[1]{\textcolor{blue}{Yufeng: #1}}
\nc{\dong}[1]{\textcolor{blue}{Dong: #1}}
\nc{\lm}[1]{\textcolor{blue}{Limeng: #1}}
\nc{\da}[1]{\textcolor{teal}{Drazen: #1}}

\begin{document}
	\title[Heisenberg-Virasoro vertex operator algebras]
	{Simple restricted modules over the   Heisenberg-Virasoro algebra as   VOA modules}
	


	\author{Haijun Tan}

\address{Tan: School of Mathematics and Statistics, Northeast Normal University, Changchun, Jilin, 130024, P. R. China.}
\email{tanhj9999@163.com}

\author{Yufeng Yao$^\dagger$}

\address{Yao: Department of Mathematics, Shanghai Maritime University,
 Shanghai, 201306, China.}\email{yfyao@shmtu.edu.cn}

\author{Kaiming Zhao}

\address{Zhao: Department of Mathematics, Wilfrid
		Laurier University, Waterloo, ON, Canada N2L 3C5,  and School of
		Mathematical Science, Hebei Normal (Teachers) University, Shijiazhuang, Hebei, 050024 P. R. China.}
	\email{kzhao@wlu.ca}
	
	\thanks{$^\dagger$the corresponding author}
	
	
	\begin{abstract}   In this paper,  we determine all simple restricted modules over  the mirror Heisenberg-Virasoro algebra ${\mathfrak{D}}$, and  the twisted Heisenberg-Virasoro algebra $\bar \DD$  with nonzero level.  As applications, we
characterize simple Whittaker modules and simple highest weight modules over  ${\mathfrak{D}}$. A vertex-algebraic interpretation of our result is the classification of  simple weak  twisted and untwisted modules over the   Heisenberg-Virasoro vertex operator algebras $\mathcal V^{c} \cong V_{Vir} ^{c} \otimes M(1)$.
We also present a few examples of simple restricted ${\mathfrak{D}}$-modules and $\bar \DD$-modules  induced from simple modules over  finite dimensional solvable Lie algebras, that are not    tensor product modules of Virasoro modules and Heisenberg modules.  This is very different from the case of   simple highest weight modules over $\DD$ and $\bar \DD$ which are  always tensor products of  simple  Virasoro modules and simple Heisenberg modules.

	\end{abstract}
	\subjclass[2010]{17B10, 17B65, 17B66, 17B68, 17B69, 17B81}
	
	\keywords{The Virasoro algebra,  the mirror Heisenberg-Virasoro algebra,    restricted module, Heisenberg-Virasoro vertex operator algebra, tensor product module}
	\maketitle
	

\setcounter{tocdepth}{1}\tableofcontents
\begin{center}
\end{center}

\section{Introduction}
		
Throughout the paper we denote by $\bZ, \bZ^*,\bN,\bZ_+, \bZ_{\leq 0}, \bR, \bC,$ and  $\bC^*$ the sets of  integers, nonzero integers, non-negative integers,  positive integers, non-positive integers, real numbers, complex numbers,   and nonzero complex numbers, respectively. All vector spaces and Lie algebras are assumed to be over $\bC$.
For a Lie algebra $\mathcal{G}$, the universal algebra of $\mathcal{G}$ is denoted by $\UU(\mathcal{G})$.

The Virasoro algebra $ \Vir$  and the Heisenberg algebra $\mh$ are infinite-dimensional Lie algebras with bases $\{{\bf c}, d_n: n\in\bZ\}$ and $\{{\bf l}, h_n: n\in\bZ\}$, respectively.  Their Lie brackets are given by
$$[\Vir,{\bf c}]=0,\,[d_m,d_n]=(m-n)d_{m+n}+\frac{m^3-m}{12}\delta_{m+n,0}{\bf c},\,m,n\in\bZ,$$
and
$$[\mh,{\bf l}]=0,\,\,\,[h_m,h_n]=m\delta_{m+n,0}{\bf l},\,\,\,m,n\in\bZ,$$
respectively. The twisted Heisenberg-Virasoro algebra  $ \bar \DD$ is the universal central extension of the Lie algebra
$$\Big\{ f(t)\frac{d}{dt}+g(t) : f,g\in\bC[t,t^{-1}]\Big\}$$ of differential operators of order at most one on the Laurent polynomial algebra $\bC [t,t^{-t}]$. Since the Lie algebra $\bar \DD$ contains the Virasoro algebra $\Vir$ and the Heisenberg algebra $\mh$ as subalgebras (but not the semi-direct product of the two subalgebras), many properties of $\bar \DD$ are closely related to the algebras $\Vir$ and $\mh$.

The Virasoro algebra $\Vir$, the Heisenberg algebra $\mh$ and the  twisted Heisenberg-Virasoro algebra  $ \bar \DD$ are very important infinite-dimensional Lie algebras in mathematics and in mathematical physics  because of its beautiful representation theory (see \cite{IK, KRR}),  and its widespread applications to vertex operator algebras (see \cite{DMZ,FZ}), quantum physics (see \cite{GO}), conformal field theory (see \cite{DMS}), and so on. Many other interesting and important algebras contain  the Virasoro algebra as a subalgebra, such as the  Schr{\"o}dinger-Virasoro algebra (see \cite{H,H1}), the mirror  Heisenberg-Virasoro algebra $\DD$ (see \cite{B,GZ,LPXZ}) and so on. These Lie algebras have nice structures and perfect theory on simple Harish-Chandra modules.  The mirror  Heisenberg-Virasoro algebra $\DD$   is the even part of the mirror  $N=2$ superconformal algebra (see \cite{B}), and
 is the semi-direct product of the Virasoro algebra and the twisted Heisenberg algebra (see Definition \ref{def.1}).
\subsection{Connection with representation theory of Lie algebras} Representation theory of Lie algebras has attracted a lot of attentions of mathematicians and physicists. For a Lie algebra $\GG$ with a  triangular decomposition $\GG=\GG_+\oplus \h\oplus \GG_-$ in the sense of \cite{MP}, one can study its weight and non-weight representation theory. For weight representation approach,
to some extent, Harish-Chandra modules are well understood for many infinite-dimensional Lie algebras, for example,
the affine Kac-Moody algebras in \cite{CP, MP}, the Virasoro algebra in \cite{FF,KRR,M},
the twisted Heisenberg-Virasoro algebra  in \cite{ACKP, LZ1},
the Schr{\"o}dinger-Virasoro algebra (partial results) in \cite{H, H1, LS},
and the mirror  Heisenberg-Virasoro algebra in \cite{LPXZ}.
There are also some researches about weight modules with infinite-dimensional weight spaces
(see \cite{BBFK,CGZ,LZ3}).

Recently, non-weight module theory over   Lie algebras $\GG$ attracts more attentions from mathematicians. In particular,  $\UU(\h)$-free $\GG$-modules, Whittaker modules, and restricted modules have been widely studied for many Lie algebras.
The notation of $\UU(\h)$-free modules was first introduced by  Nilsson \cite{N} for the simple Lie algebra $\mathfrak{sl}_{n+1}$. At the same time
these modules were introduced in a very different approach in the paper \cite{TZ}.
Later, $\UU(\h)$-free modules for many infinite-dimensional Lie algebras are determined, for example,
the Kac-Moody algebras in \cite{C,CTZ, GZ1}, the Virasoro algebra in \cite{LGZ,LZ2,MZ2},
the Witt algebra in \cite{TZ}, the twisted Heisenberg-Virasoro algebra and $W(2,2)$ algebra in \cite{CHSY,CG, LZ3}, and so on.

Whittaker modules  for $\mathfrak{sl}_2(\bC)$ were first constructed by Arnal and Pinzcon (see \cite{AP}).
Whittaker modules  for arbitrary finite-dimensional complex semisimple Lie algebra $\mg$ were  introduced and systematically studied by Kostant in \cite{Ko},
where he proved that these modules
with a fixed regular Whittaker function (Lie homomorphism) on a nilpotent
radical are (up to isomorphism) in bijective correspondence with
central characters of $\UU(\mg)$.
In recent years,   Whittaker modules for many other Lie algebras have been investigated
(see \cite{AHPY,ALZ,BM,BO,C,MD1,MD2}).

\subsection{Restricted modules}
The restricted modules for a $\bZ$-graded Lie algebra are the modules in which any vector can be annihilated by sufficiently large positive part of the Lie algebra. Whittaker modules and highest weight modules are restricted modules, and, in some sense, restricted modules can be seen as generalization of Whittaker modules and highest weight modules. Understanding restricted modules for an infinite-dimensional Lie algebra with a $\bZ$-gradation is one of core topics in Lie theory, for
this class of modules are closely connected with the modules for corresponding vertex operator algebras. The first step of studying restricted modules is to classify all restricted modules for a Lie algebra. But this is a difficult challenge.
Up to now
all  simple restricted modules for the Virasoro algebra are classified in \cite{MZ2}. There are some partial results of  simple  restricted modules for other Lie algebras. Some   simple  restricted modules for twisted Heisenberg-Virasoro algebra and  mirror Heisenberg-Virasoro algebra with level $0$  were constructed in \cite{CG1, G, LPXZ}.  Rudakov investigated  a class of simple  modules  over Lie algebras of Cartan type $W, S, H$ in
\cite{Ru1,Ru2}, and these modules are restricted modules over the Cartan type Lie algebras of the formal power series.

\subsection{Vertex algebraic approach}
For many infinite-dimensional $\mathbb{Z}$-graded Lie algebras and superalgebras $\mathcal G$, one can construct the associated (unversal) vertex algebra $\mathcal V_{\mathcal G}$  with the property:
\begin{itemize}
\item Any restricted $\mathcal G$-module is a  weak $\mathcal V_{\mathcal G}$-module;
\item Any weak module for the vertex algebra $\mathcal V_{\mathcal G}$ has the structure of a restricted $\mathcal G$-module.
\end{itemize}
This approach is very prominent for the following cases:
\begin{itemize}
\item Affine Kac-Moody algebra of type $X_n ^{(1)}$, when the associated vertex algebra is the universal affine vertex algebra $V^k(\mathfrak{g})$ for  certain simple Lie algebra $\mathfrak{g}$. This approach was used in \cite{ALZ} for studying Whittaker modules.
\item Virasoro Lie algebra, when the associated vertex algebra is the universal Virasoro vertex algebra $V_{Vir} ^c$ (cf. \cite{LL})
\item Heisenberg vertex algebra,  when the associated vertex algebra is $M(1)$ (cf. \cite{LL}).
\item Heisenberg-Virasoro algebra; super conformal algebras, etc.
\end{itemize}

From the vertex-algebraic point of view, the twisted Heisenberg-Virasoro algebra and its untwisted modules were investigated in \cite{AR,  GW}.

The restricted representations of nonzero level for the twisted Heisenberg-Virasoro  algebra corresponds to representations of the   Heisenberg-Virasoro vertex operator algebra $\mathcal V^{c}=V_{Vir} ^c \otimes M(1)$ ,
 where $V_{Vir} ^c$ is the universal Virasoro vertex operator algebra of central charge   $c= \ell_1-1$, and
 $M(1)$
 is the Heisenberg vertex operator algebra of level $1$. (Since  $M(\ell_2) \cong M(1)$ (cf. \cite{LL}), we usually assume that the level $\ell_2=1$.)

Moreover, the  restricted representations of the  mirror Heisenberg-Virasoro algebra $\DD$ can be treated as twisted modules for the Heisenberg-Virasoro vertex operator algebra $\mathcal V^{c}=V_{Vir} ^c \otimes M(1)$.

  We summarize:

  \begin{itemize}
  \item The category of restricted $\bar{\DD}$-modules of level $1$  is equivalent to the category of weak (untwisted) modules for the vertex operator algebra $\mathcal V^{c}$;

   \item The category of restricted $ {\DD}$-modules of level $1$  is equivalent to the category of weak twisted  modules for the vertex operator algebra $\mathcal V^{c}$.
  \end{itemize}

\subsection{Main results} In this paper, our main  goal is to classify all    simple restricted modules  for mirror Heisenberg-Virasoro algebra  $\DD$, and  classify      simple restricted modules with non-zero level for the twisted Heisenberg-Virasoro algebra $\bar \DD$.  As applications, we   describe the  simple untwisted and  twisted modules for    Heisenberg-Virasoro vertex operator algebras $\mathcal V^{c}$. The main results are  the following theorems:

\noindent {\bf Main theorem A}  (Theorem \ref{mainthm})
{\it Let $S$ be a simple restricted module over the mirror Heisenberg-Virasoro algebra $\DD$ with level $\ell \ne 0$. Then

{\rm(i)} $S\cong H^{\DD}$ where $H$ is a simple  restricted module over the Heisenberg algebra $\mh$, or

{\rm(ii)} $S$ is an induced $\DD$-module from a  simple restricted $\DD^{(0,-n)}$-module, or

{\rm(iii)} $S\cong U^{\DD}\otimes H^{\DD}$ where $U$ is a  simple  restricted $\Vir$-module, and $H$ is a  simple restricted module over the Heisenberg algebra $\mh$.
}

\noindent {\bf Main theorem B } (Theorem \ref{mainthm'})
{\it Let $M$ be a simple restricted module over the twisted Heisenberg-Virasoro algebra $\bar \DD$ with level $\ell \ne 0$. Then

{\rm(i)} $M\cong K(z)^{\bar \DD}$ where $K$ is a simple  restricted $\bar \mh$-module and $z\in\bC$, or

{\rm(ii)} $M$ is an induced $\bar \DD$-module from a  simple restricted $\bar \DD^{(0,-n)}$-module for some $n\in\bZ_+$, or

{\rm(iii)} $M\cong K(z)^{\bar \DD}\otimes U^{\bar \DD}$ where  $z\in\bC$,  $K$ is a  simple restricted $\bar \mh$-module  and $U$ is a  simple  restricted $\Vir$-module.
}

These simple restricted modules over the (mirror) Heisenberg-Virasoro algebra  are actually all simple weak  ($\theta$-twisted) modules over Heisenberg-Virasoro vertex operator algebras $\mathcal V^{c}$ (where the involution $\theta$ is defined in Section \ref{voasection}, see Theorem \ref{5.3-untw}, Theorem \ref{5.3}). As a consequence, we obtain the classification of twisted and untwisted simple modules for   the Heisenberg-Virasoro vertex operator algebra $\mathcal V^{c}$, i.e., we obtain all weak simple  $\mathcal V^{c}$-modules and all weak simple $\theta$-twisted $\mathcal V^{c}$-modules.

 It is important to notice that certain weak modules induced from  simple restricted $\DD^{(0,-n)}$, as a (twisted) modules for  $V_{Vir} ^c \otimes M(\ell_2)$, do not have the form $M_1 \otimes M_2$ (see Section \ref{examples}). This is interesting, since in the category of ordinary  (twisted) modules for the vertex operator alegbras, such modules don't exist (see  \cite[Theorem 4.7.4]{FHL} and its twisted analogs).

\subsection{Organization of the paper} The present paper is organized as follows. In Section 2, we recall  notations related to the algebras $\DD$ and $\bar\DD$,
collect some known results, and establish a general result  for a simple module to be a tensor product module over a class of Lie algebras (Theorem \ref{generalize}). In Section 3, we construct a class of induced simple $\DD$-modules (Theorem \ref{thmmain1.1}). In Section 4, we completely determine all  simple restricted modules over the mirror Heisenberg-Virasoro algebra $\DD$  (Theorems \ref{mainthm}, \ref{MT}). In Section 5, we use a similar method as in Section 4 to classify the simple restricted modules of level nonzero over the twisted Heisenberg-Virasoro algebra $\bar \DD$ (see Theorem \ref{mainthm'}). In Section 6, we apply  Theorem \ref{mainthm} to generalize the result in \cite{MZ1} to the algebra  ${\mathfrak{D}}$, i.e., we give a new characterization  of  simple  highest weight modules  over ${\mathfrak{D}}$ (Theorem \ref{thmmain}). We also characterize simple Whittaker modules over  ${\mathfrak{D}}$ (Theorem \ref{thmmain17}).
 In Section 7,  we  present a few examples of simple restricted ${\mathfrak{D}}$-modules  and $\bar \DD$-modules  induced from simple modules over  finite dimensional solvable Lie algebras, that are not    tensor product modules of Virasoro modules and Heisenberg modules.  This is very different from the case of   simple highest weight modules over $\DD$  and $\bar \DD$ which are  always tensor products of  simple  Virasoro modules and simple Heisenberg modules.  In Appendix A,   we  apply  Theorems \ref{mainthm} and \ref{mainthm'} to  classify
 simple weak  twisted and untwisted modules over the   Heisenberg-Virasoro vertex operator algebras $\mathcal V^{c} \cong V_{Vir} ^{c} \otimes M(1)$  (Theorems \ref{5.3-untw},  \ref{5.3}).   The main method in this paper   is  the (twisted) weak module structure of the Heisenberg vertex operator algebra $M(1)$ on  simple restricted $\DD$-modules.

\section{Notations and preliminaries}

		In this section we recall some  notations and  known results related to  the    algebras  $\DD$  and   $\bar \DD$.

	\begin{defi}\label{def.1}
		The  {\bf twisted Heisenberg-Virasoro algebra} ${\bar {\mathfrak{D}}}$ is a Lie algebra with a basis
		$$\left\{d_{m},h_{r},\bar{\bf c}_1,\bar{\mathbf{c}}_2, \bar{{\bf c}}_3: m, r\in\bZ\right\}$$ and subject to the commutation relations
		\begin{equation}\label{thva}\begin{aligned}
			&[d_m, d_n]=(m-n)d_{m+n}+\de_{m+n,0}\frac{ m^3-m}{12}\bar{\c}_1,\\
			&[d_m,h_{r}]=-rh_{m+r}+\delta_{m+r,0}(m^2+m)\bar{\c}_2,\\
			&[h_{r},h_{s}]=r\de_{r+s,0}\bar{{\bf c}}_3,\\
			&[\bar{\c}_1,{\bar {\mathfrak{D}}}]=[\bar{\c}_2,{\bar {\mathfrak{D}}}]=[\bar{\c}_3,{\bar {\mathfrak{D}}}]=0,\end{aligned}
		\end{equation}
		for $m,n,r,s\in\bZ$.
	\end{defi}

 It is clear that ${\bar {\mathfrak{D}}}$ contains a copy of the Virasoro subalgebra $\Vir=\text{span}\{\bar{\bf c}_1, d_i: i\in\bZ\}$ and the   Heisenberg algebra $\bar {\mh}=\bigoplus_{r\in\bZ}\bC h_r\oplus\bC \bar{{\bf c}}_3$. So $\bar {\DD}$ has a quotient algebra that is isomorphic to a copy of {\bf Heisenberg-Virasoro algebra} $$\widetilde{\DD}=\text{span}_{\bC}\left\{d_{m},h_{r},\bar{\bf c}_1,\bar{{\bf c}}_3: m, r\in\bZ\right\}$$ whose relations are defined by (\ref{thva}) (but the second and fourth equalities are replaced by $[d_m,h_{r}]=-rh_{m+r}$ and $[\bar{\c}_1,\widetilde{\DD}]=[\bar{\c}_3,\widetilde{\DD}]=0$).

	Note that ${\bar {\mathfrak{D}}}$ is $\bZ$-graded and equipped with  a triangular decomposition:
	$
	{{\bar {\mathfrak{D}}}}={\bar {\mathfrak{D}}}^{+}\oplus {\mathfrak{h}}\oplus {\bar {\mathfrak{D}}}^{-},
	$
	where
	\begin{eqnarray*}
		&&{\bar {\mathfrak{D}}}^{\pm}=\bigoplus_{n, r\in\bZ_+}(\bC d_{\pm n}\oplus \bC h_{\pm r}),\quad  {\mathfrak{h}}=\bC d_0\oplus\bC h_{0}\oplus\bC \bar{\c}_1+\bC \bar{\c}_2+ \bC \bar{\c}_3.
	\end{eqnarray*}
Moreover, ${\bar {\mathfrak{D}}}=\oplus_{i\in\bZ}{\bar {\mathfrak{D}}}_i$ is $\bZ$-graded with ${\bar {\mathfrak{D}}}_i=\bC d_i\oplus\bC h_{i}$ for $i\in\bZ^*$, $\bar {\mathfrak{D}}_0={\mathfrak{h}}$.
	
	\begin{defi}\label{def.2}
		The  {\bf mirror Heisenberg-Virasoro algebra} ${\mathfrak{D}}$ is a Lie algebra with a basis
		$$\left\{d_{m},h_{r},{\bf c}_1,\mathbf{c}_2\mid m\in\bZ,r\in \frac{1}{2}+\bZ\right\}$$ and subject to the commutation relations
		\begin{eqnarray*}
			&&[d_m, d_n]=(m-n)d_{m+n}+\de_{m+n,0}\frac{ m^3-m}{12}\c_1,\\
			&&[d_m,h_{r}]=-rh_{m+r},\\
			&&[h_{r},h_{s}]=r\de_{r+s,0}{\bf c}_2,\\
			&&[\c_1,{\mathfrak{D}}]=[\mathbf{c}_2,{\mathfrak{D}}]=0,
		\end{eqnarray*}
		for $m,n\in\bZ, r,s\in \frac{1}{2}+\bZ$.
	\end{defi}
It is clear that ${\mathfrak{D}}$ is the semi-direct product of the Virasoro subalgebra $\Vir=\text{span}\{{\bf c}_1, d_i\mid i\in\bZ\}$ and the twisted Heisenberg algebra $\mh=\bigoplus_{r\in\frac{1}{2}+\bZ}\bC h_r\oplus\bC {\bf c}_2$. Note that ${\mathfrak{D}}$ is $\frac{1}{2}\bZ$-graded and equipped with  triangular decomposition:
	$
	{{\mathfrak{D}}}={\mathfrak{D}}^{+}\oplus {\mathfrak{D}}^{0}\oplus {\mathfrak{D}}^{-},
	$
	where
	\begin{eqnarray*}
		&&{\mathfrak{D}}^{\pm}=\bigoplus_{n\in\bZ_+}\bC d_{\pm n}\oplus \bigoplus_{r\in\frac{1}{2}+\bN}\bC h_{\pm r},\quad  {\mathfrak{D}}^{0}=\bC d_0\oplus\bC \c_1\oplus\bC \mathbf{c}_2.
	\end{eqnarray*}
Moreover, ${\mathfrak{D}}=\oplus_{i\in\bZ}{\mathfrak{D}}_i$ is $\bZ$-graded with ${\mathfrak{D}}_i=\bC d_i\oplus\bC h_{i+\frac{1}{2}}$ for $i\in\bZ^*\setminus\{-1\}$, ${\mathfrak{D}}_0=\bC d_0\oplus\bC h_{\frac{1}{2}}\oplus\bC \c_1$ and ${\mathfrak{D}}_{-1}=\bC d_{-1}\oplus\bC h_{-\frac{1}{2}}\oplus\bC \mathbf{c}_2$.

\begin{defi}
Let $\GG=\oplus_{i\in\bZ}{\GG}_{i}$ be a $\bZ$-graded Lie algebra.
A $\GG$-module $V$ is called  the $\bf{restricted}$ module if for any $v\in V$ there exists $n\in\bN$ such that $\GG_iv=0$,
for $i>n$. The  category of restricted modules over $\GG$ will be denoted as $\RR_{\GG}$.
\end{defi}

\begin{defi}
Let $\mathfrak{a}$ be a subalgebra of a Lie algebra $\GG$,  and $V$ be a $\GG$-module.  We denote
$${\rm Ann}_V(\mathfrak{a})=\{v\in V:\mathfrak{a}v=0\}.
$$
\end{defi}

\begin{defi}
Let $\GG$ be a  Lie algebra and $V$ a $\GG$-module and $x\in \GG$.
\begin{itemize}
\item[\rm (1)] If for any $v\in V$ there exists $n\in\z_+$ such that $x^nv=0$,  then we say that the action of  $x$  on $V$ is {\bf locally nilpotent}.
\item[\rm (2)] If for any $v\in V$ we have $\mathrm{dim}(\sum_{n\in\bN}\cc x^nv)<+\infty$, then   the action of $x$ on $V$ is  said to be {\bf locally finite}.
\item[\rm (3)] The action of $\GG$   on $V$  is said to be {\bf locally nilpotent} if for any $v\in V$ there exists an $n\in\z_+$ (depending on $v$) such that $x_1x_2\cdots x_nv=0$ for any $x_1,x_2,\cdots, x_n\in L$.
\item[\rm (4)] The action of $\GG$  on $V$  is said to be {\bf locally finite} 	if for any $v\in V$ there is a finite-dimensional $L$-submodule of $V$ containing $v$.
\end{itemize}
\end{defi}

	\begin{defi}
		If $W$ is a $\mathfrak D$-module ({\it resp.} $\bar {\DD}$-module) on which  ${\bf c}_1$ ({\it resp.} $\bar{\c}_1$) acts as complex
		scalar $c$,  we say that $W$ is of {\bf central charge} $c$. If $W$ is a $\mathfrak D$-module ({\it resp.} $\bar {\DD}$-module) on which  ${\bf c}_2$ ({\it resp.} $\bar{\c}_3$) acts as complex
		scalar $\ell$, we say that $W$ is of {\bf level} $\ell$.
	\end{defi}

Note that if $V$ is a $\Vir$-module,
then $V$ can be easily viewed as a $\DD$-module ({ resp.} $\bar {\DD}$-module) by defining $\mh V=0$ ({ resp.} $(\bar {\mh}+\bC \bar{\c}_2)V=0$),
the resulting  module is denoted by $V^{\DD}$({ resp.} $V^{\bar {\DD}}$).

Thanks to \cite{LPXZ}, for any $H\in\RR_{\mh}$ with the action of ${\bf c}_2$ as a nonzero scalar $\ell$, we can give $H$ a $\DD$-module structure
denoted by $H^{{\DD}}$ via the following map
\begin{align}
d_n&\mapsto L_n= \frac{1}{2\ell}\sum_{k\in\bZ+\frac{1}{2}}h_{n-k}h_k,\quad\forall n\in\bZ, n\neq0 ,\label{rep1}\\
d_0&\mapsto  L_0=\frac{1}{2\ell}\sum_{k\in\bZ+\frac{1}{2}}h_{-|k|}h_{|k|}+\frac{1}{16},\label{rep2}\\
h_r&\mapsto h_r,\quad\forall r\in\frac{1}{2}+\bZ,
\quad {\bf{c}_1} \mapsto 1,
\quad {\bf{c}_2}\mapsto \ell.\label{rep3}
\end{align}

\begin{rem}
The vertex operator algebra interpretation of the formula (\ref{rep1})-(\ref{rep2}) will be given in
Section \ref{voa-interp}. The operators $L_n$, $n \in {\bZ}$, will be represented as  components  of the field $L_{tw} ^{Heis} (z)$ in    (\ref{heis-tw-vir}).
\end{rem}

According to (9.4.13)  and (9.4.15) in  \cite{FLM},  we know that  for  all $m,n\in\bZ,   r\in\frac{1}{2}+\bZ$, we  have
\begin{equation}\label{rep4}
\aligned {}
[{L}_n,h_r]&=-rh_{n+r},\\ [{L}_m,{L}_n]&=(m-n){L}_{m+n}+\frac{m^3-m}{12}\delta_{m+n,0} .\endaligned
\end{equation}
Morover, since
$$[d_m,h_{n-k}h_k]=[d_m,h_{n-k}]h_k+h_{n-k}[d_m,h_k]=[{L}_m,h_{n-k}]h_k+h_{n-k}[{L}_m,h_k]=[{L}_m,h_{n-k}h_k],$$
we see  that
\begin{equation}\label{rep4'}[d_m,{L}_n]=[{L}_m,{L}_n]\end{equation}

By \cite{LZ3}, for  any $z\in\bC$ and  $H\in\RR_{\bar {\mh}}$ with the action of ${\bar{\c}}_3$ as a nonzero scalar $\ell$, we can give $H$ a $\bar {\DD}$-module structure
(denoted by $H(z)^{\bar {\DD}}$) via the following map
\begin{align}
d_n&\mapsto \bar{L}_n= \frac{1}{2\ell}\sum_{k\in\bZ}:h_{n-k}h_k:+\frac{(n+1)z}{\ell}h_n,\quad\forall n\in\bZ,\label{rep1-untw}\\
h_r&\mapsto h_r,\quad\forall r\in\bZ,
\quad {\bar{\c}_1} \mapsto 1-\frac{12z^2}{\ell}, \,\quad {\bar{\c}_2} \mapsto z,
\quad {\bar{\c}_3}\mapsto \ell,\label{rep2-untw}
\end{align}
where the normal order is define as
	$$\normOrd{h_{r}h_s}=\normOrd{h_{s}h_r}=h_{r}h_s,\text{ if } r\le s.$$

According to (8.7.9), (8.7.13) in  \cite{FLM} and by some simple computation,  we deduce  that  for  all $m,n,r\in\bZ$,
\begin{equation}\label{rep3-untw}
\aligned {}
[\bar{L}_m,h_r]&=-rh_{m+r}+\delta_{m+r,0}(m^2+m)z,   \\ [\bar{L}_m,\bar{L}_n]&=(m-n)\bar L_{m+n}+\frac{m^3-m}{12}\delta_{m+n,0}(1-\frac{12z^2}{\ell})      .\endaligned
\end{equation}
\begin{rem}  The vertex operator algebra interpretation of the operators (\ref{rep1-untw}) will be  given later  in (\ref{rep1-voa}). Then relations (\ref{rep3-untw})  can be obtained using commutator formula, similarly  as in \cite{AR}.
\end{rem}
Moreover, since
$$[d_m,h_{n-k}h_k]=[d_m,h_{n-k}]h_k+h_{n-k}[d_m,h_k]=[\bar{L}_m,h_{n-k}]h_k+h_{n-k}[\bar{L}_m,h_k]=[\bar{L}_m,h_{n-k}h_k],$$
we see  that
\begin{equation}\label{rep3'}[d_m,\bar{L}_n]=[\bar{L}_m,\bar{L}_n].\end{equation}


For convenience, we define the following subalgebras of $\DD$. For any $m\in\bN, n\in\bZ$, set
\begin{equation}\label{Natations}
\begin{split}
&\DD^{(m, n)}=\sum_{i\in\bN}\bC d_{m+i}\oplus \bC h_{n+i+\frac12}\oplus\bC{\bf c}_1\oplus \bC{\bf c}_2, \\
&\DD^{(m,-\infty)}=\sum_{i\in\bN}\bC d_{m+i}\oplus \sum_{i\in\bZ}\bC h_{i+\frac12}+\bC{\bf c}_1+\bC{\bf c}_2,\\
&\Vir^{(m)}= \sum_{i\in\bN}\bC d_{m+i}\oplus\bC{\bf c}_1,\\
&\Vir_{\geq m}= \sum_{i\in\bN}\bC d_{m+i},\\
&\Vir_{\leq 0}= \sum_{i\in\bN}\bC d_{-i},\\
 &\Vir_+=\text{span}\{{\bf c}_1, d_i: i\ge 0\},\\
 &\mh^{(n)}=\sum_{i\in\bN}\bC h_{n+i+\frac12}\oplus\bC{\bf c}_2,\\
 &\mh_{\ge n}=\sum_{i\in\bN}\bC h_{n+i+\frac12}.
\end{split}
\end{equation}

Similarly, we define the subalgebras of $\bar {\DD}$ as following: for $m\in\bN, n\in\bZ$, set
\begin{equation}\label{Natations'}
\begin{split}
&\bar {\DD}^{(m, n)}=\sum_{i\in\bN}\bC d_{m+i}\oplus \bC h_{n+i}\oplus\bC{\bar{\c}}_1\oplus \bC{\bar{\c}}_2+\bC{\bar{\c}}_3, \\
&\bar {\DD}^{(m,-\infty)}=\sum_{i\in\bN}\bC d_{m+i}\oplus \sum_{i\in\bZ}\bC h_{i}+\bC{\bar{\c}}_1\oplus \bC{\bar{\c}}_2+\bC{\bar{\c}}_3,\\
&\Vir^{(m)}= \sum_{i\in\bN}\bC d_{m+i}\oplus\bC{\bar{\c}}_1,\\
&\Vir_{\geq m}= \sum_{i\in\bN}\bC d_{m+i},\\
&\Vir_{\leq 0}= \sum_{i\in\bN}\bC d_{-i},\\
 &\Vir_+=\text{span}\{{\bar{\c}}_1, d_i: i\ge 0\},\\
 &\bar {\mh}^{(n)}=\sum_{i\in\bN}\bC h_{n+i}\oplus\bC{\bar{\c}}_3,\\
 &\bar {\mh}_{\ge n}=\sum_{i\in\bN}\bC h_{n+i}.
\end{split}
\end{equation}
Note that we use the same notations $\Vir^{(m)}, \Vir_{\geq m}, \Vir_{\leq 0}, \Vir_+$ to denote the subalgebras of $\DD$ and of  $\bar {\DD}$ since there will be no ambiguities in later contexts.
	

Denote by  $\mathbb{M}$ the set of all infinite vectors of the form $\mi:=(\ldots, i_2, i_1)$ with entries in $\mathbb N$,
satisfying the condition that the number of nonzero entries is finite.
We can make
 $(\mathbb{M}, +)$ a monoid by
 $$(\ldots, i_2, i_1)+(\ldots, j_2, j_1)=(\ldots, i_2+j_2, i_1+j_1).$$
 Let $\mathbf{0}$ denote the element $(\ldots, 0, 0)\in\mathbb{M}$ and
for
$i\in\z_+$ let
$\epsilon_i=(\ldots,0,1,0,\ldots,0)\in\mathbb{M}$,
where $1$ is
in the $i$'th  position from right.
For any $\mi\in\M$ we define
\begin{equation} {\rm w}(\mi)=\sum_{n\in\z_+}i_n\cdot n,\,\,\,
\label{length}\end{equation}


Let $\prec$ be the {\bf reverse lexicographic} total order on $\M$,
that is, for any $\bfi,\bfj\in \M$,
\begin{equation*}
\bfj \prec \bfi\Longleftrightarrow {\rm\ there\ exists\ } r\in\N {\rm\ such\ that\ } j_r<i_r {\rm\ and\ } j_s=i_s,\
\forall\ 1\leq s<r.
\end{equation*}

We extend the above total order on $\M \times \M$, that is,
for all $\bfi,\bfj,\bfk,\bfl\in \M$,
$$(\bfi,\bfj)  \prec (\bfk,\bfl)\iff \aligned
(\bfj,{\rm w}(\bfj), {\rm w}(\bfi)+{\rm w}(\bfj))&\prec(\bfl, {\rm w}(\bfl), {\rm w}(\bfk)+{\rm w}(\bfl)), \text{ or }\\
(\bfj,{\rm w}(\bfj), {\rm w}(\bfi)+{\rm w}(\bfj))&\prec(\bfl, {\rm w}(\bfl), {\rm w}(\bfk)+{\rm w}(\bfl)), \text{ and }
\bfi\prec\bfk
.\endaligned$$

Now we define another total order     $\prec'$  on $\M \times \M$:
for all $\bfi,\bfj,\bfk,\bfl\in \M$,
$$(\bfi,\bfj) \prec' (\bfk,\bfl)  \iff  (\bfj,\bfi) \prec (\bfl,\bfk) $$
The symbols $\preceq$ and $\preceq'$ have the obvious meanings.

It is not hard to verify  that
 $$(a, b) \preceq (c,d)\,\,\, \&\,\,\, (c',d')\prec(a',b') \implies (a-a', b-b') \prec (c-c',d-d'),$$
provided $(a, b) , (c,d)\,\,\, (c',d'),(a',b'), (a-a', b-b') , (c-c',d-d')\in \M \times \M,$ where the difference  is the corresponding entry difference.


For $n\in \bZ$, let $V$ be an  simple   ${\mathfrak{D}}^{(0,-n)}$-module.
According to the $\mathrm{PBW}$ Theorem, every nonzero element $v\in \mathrm{Ind}_{{\mathfrak{D}}^{(0,-n)}}^ {\mathfrak{D}}(V)$ can be uniquely written in the
following form
\begin{equation}\label{def2.1}
v=\sum_{\mi, \mk\in\mathbb{M}} h^{\mi}d^{\mk} v_{\mi, \mk},
\end{equation}
where
$$ h^{\mi}d^{\mk}=\ldots h_{-2-n+\frac{1}{2}}^{i_2} h_{-1-n+\frac{1}{2}}^{i_1}\ldots d_{-2}^{k_2} d_{-1}^{k_1}\in U(\mathfrak D^{-}),v_{\mi, \mk}\in V,$$
and only finitely many $v_{\mi, \mk}$ are nonzero. For any nonzero $v\in\mathrm{Ind}(V)$ as in  \eqref{def2.1}, we will use the following notations for   later use:

(1) Denote by $\mathrm{supp}(v)$ the set of all $(\mi,\mk)\in \mathbb{M}\times \mathbb{M}$  such that $v_{\mi, \mk}\neq0$.

(2) Denote by $$\mathrm{w}(v) =\mathrm{max}\{{\rm w}(\mi)+{\rm w}({\bf k}): (\mi, {\bf k})\in \mathrm{supp}(v)\}, $$ called the {\bf  length} of $v$.

(3) Denote by
$\deg(v)$ to be the largest element in $\mathrm{supp}(v)$ with respect to the total order $\prec$.

(4) Denote by
$\deg'(v)$ to be the largest element in $\mathrm{supp}(v)$ with respect to the total order $\prec'$.

We first recall from \cite{MZ2} the classification for  simple restricted $\Vir$-modules.

\begin{theo} \label{Vir-modules} Any   simple restricted $\Vir$-module is a highest weight module, or isomorphic to $ {\rm Ind}_{\Vir_+}^{\Vir}V$ for an  simple  $\Vir_+$-module $V$ such that for some $k\in \z_+$,
\begin{enumerate}[$($a$)$]
\item $d_k$ acts injectively on $V$;
\item$d_iV=0$ for all $i>k$.
\end{enumerate}
\end{theo}


 Simple restricted $\DD$-modules with level $0$ are classified in \cite{LPXZ} by the following two theorems.

 \begin{theo}\label{simple-theo}
	Let $V$ be a simple  ${\mathfrak{D}}^{(0,-n)}$-module for some $n\in \mathbb{Z}_+$ and $c\in \bC$ such that ${\bf c}_1v=cv, {\bf c}_2v=0$ for any $v\in V$. Assume that there exists an  integer $k\ge-n$ satisfying the following two conditions:
	\begin{itemize}
		\item[(a)] the action of  $h_{k+\frac12}$   on  $V$ is bijective;
		\item[(b)] $h_{m+\frac{1}{2}}V=0=d_{m+n}V$ for all $m> k$.
	\end{itemize}
	Then	 the induced    ${\mathfrak{D}}$-module $\mathrm{Ind}_{{\mathfrak{D}}^{(0,-n)}}^ {\mathfrak{D}} (V)$ is   simple.
\end{theo}

\begin{theo}\label{MT}
	Every simple restricted ${\mathfrak{D}}$-module of  level 0 is isomorphic to    a restricted $\Vir$-module with $\mh S=0$, or $S\cong\mathrm{Ind}_{{\mathfrak{D}}^{(0,-n)}}^ {\mathfrak{D}} (V)$ for   some $n\in \mathbb{N}$ and  a simple  ${\mathfrak{D}}^{(0,-n)}$-module $V$  as in Theorem \ref{simple-theo}.
\end{theo}
	
 Actually the simple  ${\mathfrak{D}}^{(0,-n)}$-module $V$  can be considered as a simple   module over a finite dimensional solvable  Lie algebra
	${\mathfrak{D}}^{(0,-n)}/{\mathfrak{D}}^{(t+n+1, t-n)}$   for some $ t\in\bZ_+$ and injective action of $h_{t+\frac12}$ on $V$.

For simple  restricted $\bar \DD$-modules with level $0$, we know the following results from \cite{CG1}.

\begin{theo} Let $n\in\bN$ and $V$ be a  simple module  over $\bar \DD^{(0,-n)}$  or over $\bar \DD^{(0,-\infty)}$  with  $\ell =0$, $h_0=\mu, \bar\c_2=z$. If there exists $k\in\bN$ such that

(a)\begin{equation*}
\begin{cases}h_k \,\text{acts\, injectively \,on\, V},& \text{if}\, k\ne 0,\\
\mu+(1-r)z\neq 0,\forall r\in\bZ\setminus\{0\}, &\text{if}\,k=0;
\end{cases}
\end{equation*}

(b) $h_iV=d_jV=0$ for all $i>k$ and $j>k+n$.

\noindent then

(1) $\Ind(V)$ is a simple $\bar \DD$-module;

(2) $h_i, d_j$ act locally nilpotently on $\Ind(V)$ for all $i>k$ and $j>k+n$.

\end{theo}

%
%

Now we generalize Theorem 12 in  \cite{LZ3} as follows.

Let ${\mathfrak g}={\mathfrak a}\ltimes{\mathfrak b}$ be a Lie algebra where ${\mathfrak a}$ is a Lie subalgebra  of ${\mathfrak g}$ and ${\mathfrak b}$ is an ideal of ${\mathfrak g}$. Let $M$ be a  ${\mathfrak g}$-module with a   ${\mathfrak b}$-submodule $H$ so that the ${\mathfrak b}$-submodule structure on $H$ can be extended to a ${\mathfrak g}$-module structure on $H$. We denote this ${\mathfrak g}$-module by $H^{\mathfrak g}$. For any ${\mathfrak a}$-module   $U$, we can make it into a ${\mathfrak g}$-module by ${\mathfrak b}U=0$. We denote this ${\mathfrak g}$-module by $U^{\mathfrak g}$.

\begin{theo}\label{generalize} Let ${\mathfrak g}={\mathfrak a}\ltimes{\mathfrak b}$ be a countable dimensional Lie algebra where ${\mathfrak a}$ is a Lie subalgebra  of ${\mathfrak g}$ and ${\mathfrak b}$ is an ideal of ${\mathfrak g}$.
Let $M$ be a  simple ${\mathfrak g}$-module
with a simple  ${\mathfrak b}$-submodule $H$ so that an $H^{\mathfrak g}$ exists.
Then $M\cong H^{\mathfrak{g}}\otimes U^{\mathfrak{g}}$ as $\mathfrak{g}$-modules for some  simple  ${\mathfrak a}$-module $U$.
\end{theo}

\begin{proof}
Define the one-dimensional ${\mathfrak b}$-module $\bC v_0$ by ${\mathfrak b} v_0=0$. Then $H\cong H\otimes \bC v_0$ as ${\mathfrak b}$-modules.
Now from Lemma 8 in \cite{LZ3}, we have
$$\Ind_{{\mathfrak b}}^{\mathfrak{g}}H\cong \Ind_{{\mathfrak b}}^{\mathfrak{g}}\left(H\otimes \bC v_0\right)\cong H^{\mathfrak{g}}\otimes \Ind_{{\mathfrak b}}^{\mathfrak{g}}\bC v_0 .$$
Note that $\Ind_{{\mathfrak b}}^{\mathfrak{g}}\bC v_0 \cong W^{\mathfrak{g}}$ for the universal ${\mathfrak a}$-module $W$.
Since $M$ is a simple quotient of $\Ind_{{\mathfrak b}}^{\mathfrak{g}}H$, from Theorem 7 in \cite{LZ3} we know that there is a simple quotient ${\mathfrak a}$-module $U$ of $W$ such that $M\cong H^{\mathfrak{g}}\otimes U^{\mathfrak{g}}$ as $\mathfrak{g}$-modules.
Now the theorem follows.\end{proof}

{\bf Remark.} This theorem has particular meaning for ${\mathfrak g}={\mathfrak a}\oplus {\mathfrak b}$ since $H^{\mathfrak g}$ automatically exists (see for example \cite{Li}).
Also, Theorem \ref{generalize} holds for associative algebras.

Applying the above theorem to our  mirror Heisenberg-Virasoro algebra $\mathfrak{D}=\Vir\ltimes \mh$ and  twisted Heisenberg-Virasoro algebra $\bar {\DD}=\Vir\ltimes (\bar {\mh}+\bC \bar{\c}_2)$, we have the following results.

\begin{cor}\label{tensor} Let $V$ be a simple ${\mathfrak{D}}$-module with nonzero action of ${\bf c}_2$. Then $V\cong H^{\mathfrak{D}}\otimes U^{\mathfrak{D}}$ as a $\mathfrak{D}$-module for some simple module $H\in\R_{\mh}$ and some  simple $\Vir$-module $U$  if and only if $ V$ contains
 a simple $\mh$-submodule $H\in\R_{\mh}$.
 \end{cor}
\begin{proof} The sufficiency follows from Theorem \ref{generalize}; and the necessity follows from that $H\otimes u$ is a simple $\mh$-submodule of $H^{\mathfrak{D}}\otimes U^{\mathfrak{D}}$ for any nonzero $u\in U$.
\end{proof}


\section{Induced modules over the mirror  Heisenberg-Virasoro algebra $\DD$}

In this section, we construct some simple restricted    $\DD$-modules induced from some simple ones over some subalgebras $\DD^{(0,-n)}$ for $n\in\bZ_+$.
For that, we need the following formulas in $U(\mathfrak{D})$ which can be shown by induction on $t$: let $i,j_s\in\bZ,1\le s\le t$ with $j_1\le j_2\le\cdots\le j_t$,
\begin{equation}\label{bracket3.1}
[h_{i-\frac{1}{2}},h_{j_1+\frac{1}{2}}h_{j_2+\frac{1}{2}}\cdots h_{j_t+\frac{1}{2}}]=\sum_{1\le s\le t}\delta_{i+j_s,0}(i-\frac{1}{2}) {\bf{c}}_2 h_{j_1+\frac{1}{2}}\cdots \hat{h}_{j_{s}+\frac{1}{2}}\cdots h_{j_t+\frac{1}{2}},
\end{equation}
\begin{equation}\label{bracket3.2}
\begin{aligned}
&[d_{i},h_{j_1+\frac{1}{2}}h_{j_2+\frac{1}{2}}\cdots h_{j_t+\frac{1}{2}}]=\sum_{1\le s\le t}(-j_s-\frac{1}{2}) h_{j_1+\frac{1}{2}}\cdots\hat{h}_{j_s+\frac{1}{2}} \cdots h_{j_t+\frac{1}{2}}{h_{i+j_{s}+\frac{1}{2}}}\\
&\hskip 20pt +\sum_{1\le s_1<s_2\le t}(-j_{s_1}-\frac{1}{2})(i+j_{s_1}+\frac{1}{2})\delta_{i+j_{s_1}+j_{s_2}+1,0}{\bf{c}}_2 h_{j_1+\frac{1}{2}}\cdots \hat{h}_{j_{s_1}+\frac{1}{2}}\cdots  \hat{h}_{j_{s_2}+\frac{1}{2}} \cdots h_{j_t+\frac{1}{2}},
\end{aligned}
\end{equation}
\begin{equation}\label{bracket3.3}
\begin{aligned}
&[h_{i-\frac{1}{2}},d_{j_1}d_{j_2}\cdots d_{j_t}]=\sum_{1\le s\le t}(i-\frac{1}{2})d_{j_1}\cdots \hat{d_{j_s}}\cdots d_{j_t}h_{i+j_s-\frac{1}{2}}\\
&\hskip 20pt +\sum_{1\le s_1<s_2\le t}a_{s_1,s_2}d_{j_1}\cdots \hat{d}_{j_{s_1}}\cdots\hat{d}_{j_{s_2}}\cdots d_{j_t}h_{i+j_{s_1}+j_{s_2}-\frac{1}{2}}+\cdots+ a_{1,2,\cdots,t}h_{i+j_1+j_2+\cdots+j_t-\frac{1}{2}},\end{aligned}
\end{equation}
\begin{equation}\label{bracket3.4}
\begin{aligned}
&[d_{i},d_{j_1}d_{j_2}\cdots d_{j_t}]=\sum_{1\le s\le t}(i-j_s)d_{j_1}\cdots \hat{d_{j_s}}\cdots d_{j_t}\tilde{d}_{i+j_s}\\
&\hskip 30pt +\sum_{1\le s_1<s_2\le t}b_{s_1,s_2}d_{j_1}\cdots \hat{d}_{j_{s_1}}\cdots\hat{d}_{j_{s_2}}\cdots d_{j_t}\tilde{d}_{i+j_{s_1}+j_{s_2}}+\cdots+ b_{1,2,\cdots,t}\tilde{d}_{i+j_1+j_2+\cdots+j_t},\end{aligned}
\end{equation}
where $\hat{h}_{j_{s}+\frac{1}{2}},\hat{d}_{j_s}$ mean that $h_{j_{s}+\frac{1}{2}}, d_{j_s}$ are deleted in the corresponding products, $a_{s_1,s_2},\cdots,$
$a_{1,2,\cdots,t},$
$b_{s_1,s_2},\cdots,b_{1,2,\cdots,t}\in \bC$, and $\tilde{d}_{i+j_1+\cdots+j_s}=d_{i+j_1+\cdots+j_s}+\frac{j_{s}^2-1}{24}\delta_{i+j_1+\cdots+j_s,0}{\bf{c}}_1,1\leq s\leq t.$

We are now in the position to present the following main result in this section.

\begin{theo}  \label{thmmain1.1}
Let $k\in\bZ_+$ and $ n\in\bZ$ with $k\ge n$. Let   $V$ be a simple
$\DD^{(0,-n)}$-module with level $\ell \not=0$ such that
 there exists $l\in\bN$ satisfying both conditions:
\begin{enumerate}[$(a)$]
\item $h_{k-\frac{1}{2}}$ acts injectively on $V$;
\item
$h_{i -\frac{1}{2}} V=d_j V=0$ for all $i>k$ and $j>l$. \end{enumerate}
Then $\Ind^{\DD}_{\DD^{(0,-n)}}(V)$ is a simple $\DD$-module if  one of the following conditions holds:
\begin{itemize}
\item[\rm(1)] $k=n$, $l\geq 2n$, and $d_l$  acts injectively on $V$ ;
\item[\rm(2)] $k>n$, $k+n\geq 2$, and $l=n+k-1$.
\end{itemize}
\end{theo}

Theorem \ref{thmmain1.1} follows from Lemmas \ref{main1'}-\ref{main4} directly.


\begin{lem}\label{main1'}Let $n\in\bZ_+$ and $V$ be a
$\DD^{(0,-n)}$-module such that
  $h_{n-\frac{1}{2}}$  acts injectively on $V$, and
$h_{i -\frac{1}{2}} V=0$ for all $i>n$.
For any $v\in \Ind(V)\setminus V$, let $\deg(v)=(\bfi,\bfj)$ .
 If $\bfi\not=\bf0$, then $\deg(h_{p+n-\frac{1}{2}}v)=(\bfi-\epsilon_p,\bfj)$ where $p=\min\{s:i_s\neq 0\}$.
\end{lem}

\begin{proof} 
Write $v$ in the form of \eqref{def2.1} and let  $(\bfk,\bfl)\in\supp(v)$.




 Noticing that $h_{p+n-\frac{1}{2}}V=0$, we have
$$h_{p+n-\frac{1}{2}}h^{\bfk}d^{\bfl}v_{\bfk,\bfl}=[h_{p+n-\frac{1}{2}},h^{\bfk}]d^{\bfl}v_{\bfk,\bfl}+h^{\bfk}[h_{p+n-\frac{1}{2}},d^{\bfl}]v_{\bfk,\bfl}.$$

First we consider the term  $[h_{p+n-\frac{1}{2}},h^{\bfk}]d^{\bfl}v_{\bfk,\bfl}$ which is zero if $k_p=0$. In the case that  $k_p>0$, since the level $\ell\ne 0$, it follows from (\ref{bracket3.1}) that $[h_{p+n-\frac{1}{2}},h^{\bfk}]=\lambda h^{\bfk-\epsilon_p}$ for
some  $\lambda\in\bC^*$. So $$\text{deg}([h_{p+n-\frac{1}{2}},h^{\bfk}]d^{\bfl}v_{\bfk,\bfl})=(\bfk-\epsilon_p,\bfl)\preceq (\bfi-\epsilon_p,\bfj),$$
where the equality holds if and only if $(\bfk,\bfl)=(\bfi,\bfj)$.

Now we consider the term  $h^{\bfk}[h_{p
+n-\frac{1}{2}},d^{\bfl}]v_{\bfk,\bfl}$ which is by  (\ref{bracket3.3})  a linear combination of some vectors  in the form $h^{\bfk}d^{\bfl_j}h_{p+n-\frac{1}{2}-j}v_{\bfk,\bfl}$ with $j\in\bZ_+$ and ${\rm{w}}
(\bfl_j)={\rm{w}}(\bfl)-j$.  If $h^{\bfk}d^{\bfl_j}h_{p+n-\frac{1}{2}-j}v_{\bfk,\bfl}\ne0$,
we denote $\deg(h^{\bfk}d^{\bfl_j}h_{p+n-\frac{1}{2}-j}v_{\bfk,\bfl})=(\bfk^*,\bfl^*)$.
We will show that
\begin{equation}\label{Eq3.5}(\bfk^*,\bfl^*)
\prec (\bfi-\epsilon_p,\bfj).\end{equation} We have four   different cases to consider.

(a) $j<p$. Then $p+n-j>n$ and $h_{p+n-\frac{1}{2}-j}v_{\bfk,\bfl}=0$. Hence $h^{\bfk}d^{\bfl_j}h_{p+n-\frac{1}{2}-j}v_{\bfk,\bfl}=0$.

(b) $j=p$. Noting that $h_{n-\frac{1}{2}}$ acts injectively on $V$, we see  $(\bfk^*,\bfl^*)=(\bfk, \bfl_p)$ and  ${\rm{w}}(\bfk^*)+{\rm{w}(\bfl^*)=(\bfk)+\rm{w}(\bfl)-}p$ with ${\rm{w}}(\bfl_p)={\rm{w}(\bfl)}-p<\mathrm{w}(\bfl).$

If ${\rm{w}(\bfk)+\rm{w}(\bfl)}<{\rm{w}(\bfi)+\rm{w}(\bfj)}$, then $(\bfk^*,\bfl^*)\prec (\bfi-\epsilon_p,\bfj).$

If ${\rm{w}(\bfk)+\rm{w}(\bfl)}={\rm{w}(\bfi)+\rm{w}(\bfj)}$,
then there is $\tau\in \M$ such that ${\rm w}(\tau)=p$ and $\bfl_p=\bfl-\tau$. Since $ (\epsilon_p, \bf0)\prec (\bf0, \tau)$ and $(\bfk,\bfl)\preceq (\bfi,\bfj)$, we see that $$(\bfk^* ,\bfl^*)=(\bfk ,\bfl)-(\bf0, \tau) \prec (\bfi ,\bfj)- (\epsilon_p,\bf0)=
 (\bfi-\epsilon_p,\bfj).$$

In both cases, (\ref{Eq3.5}) holds.

(c) $p<j<2n+p$. Then $h_{p+n-\frac{1}{2}-j}\in \mathfrak{D}^{(0,-n)}$ and $h_{p+n-\frac{1}{2}-j}v_{\bfk,\bfl}\in V.$ So $$   {\rm{w}(\bfk^*)+\rm{w}(\bfl^*)}=  {\rm{w}(\bfk)+\rm{w}(\bfl)}-j<{\rm{w}(\bfk)+\rm{w}(\bfl)}-p$$
and (\ref{Eq3.5}) holds.

(d) $j\ge 2n+p$. Then $p+n-\frac{1}{2}-j<-n+\frac{1}{2}$. Assume $p+n-\frac{1}{2}-j=-s-n+\frac{1}{2}$ for some $s\in \bZ_+$, that is, $-j+s= -2n-p+1<-p$. Clearly, the corresponding vector $h^{\bfk}d^{\bfl_j}h_{p+n-\frac{1}{2}-j}v_{\bfk,\bfl}$ can be written in the form $$h^{\bfk}h_{-s-n+\frac{1}{2}}d^{\bfl_j}v_{\bfk,\bfl}+\text{lower\, terms},$$
which means 
$$ {\rm{w}(\bfk^*)+\rm{w}(\bfl^*)}= {\rm{w}(\bfk)+\rm{w}(\bfl)}-j+s<{\rm{w}}(\bfk)+{\rm{w}}(\bfl)-p,$$ and hence (\ref{Eq3.5}) holds.


In conclusion, $\text{deg}(h_{p+n-\frac{1}{2}}h^{\bfk}d^{\bfl}v_{\bfk,\bfl})
\preceq (\bfi-\epsilon_p,\bfj),$
where the equality holds if and only if $(\bfk,\bfl)=(\bfi,\bfj)$, that is,   $\deg(h_{p+n-\frac{1}{2}}v)=(\bfi-\epsilon_p,\bfj)$.\end{proof}

\begin{lem}\label{main1} Let $n\in\bZ_+$ and $V$ be a $\DD^{(0,-n)}$-module satisfying Conditions (a), (b) and (1) in Theorem \ref{thmmain1.1}.
If $v\in \Ind(V)\setminus V$ with  $\deg(v)=(\bf0,\bfj)$ , then $\deg(d_{q+l}v)=(\bfz,\bfj-\epsilon_q)$ where  $q=\min\{s:j_s\neq 0\}$.
\end{lem}

\begin{proof} 
Write $v$ in the form of \eqref{def2.1} and let  $(\bfk,\bfl)\in\supp(v)$.




Since $d_{q+l}V=0$, we have $$d_{q+l}h^{\bfk}d^{\bfl}v_{\bfk,\bfl}=[d_{q+l},h^{\bfk}]d^{\bfl}v_{\bfk,\bfl}+h^{\bfk}[d_{q+l},d^{\bfl}]v_{\bfk,\bfl}.$$

We first consider the degree of $h^{\bfk}[d_{q+l},d^{\bfl}]v_{\bfk,\bfl}$
with $d_{q+l}h^{\bfk}d^{\bfl}v_{\bfk,\bfl}\neq 0$.  Clearly, by (\ref{bracket3.4}) we see that $h^{\bfk}[d_{q+l},d^{\bfl}]v_{\bfk,\bfl}$ is is a linear combination of some vectors of the forms $h^{\bfk}d^{\bfl_j}d_{q+l-j}v_{\bfk,\bfl}$ ,$j\in\bZ_+$ and  $h^{\bfk}d^{\bfl_{q+l}}v_{\bfk,\bfl}$ where ${\rm{w}}(\bfl_j)={\rm{w}}(\bfl)-j$.  Clearly, $\deg(h^{\bfk}d^{\bfl_{q+l}}v_{\bfk,\bfl})=(\bfk,\bfl_{q+l})$ has weight
$${\rm{w}(\bfk)+\rm{w}(\bfl)}-q-l<{\rm{w}(\bfk)+\rm{w}(\bfl)}-q\le {\rm{w}}(\bfj)-q,$$ so $\deg(h^{\bfk}d^{\bfl_{q+l}}v_{\bfk,\bfl})\prec (\bf0,\bfj-\epsilon_q)$. Then we need only to consider $h^{\bfk}d^{\bfl_j}d_{q+l-j}v_{\bfk,\bfl}$.
Denote $\deg(h^{\bfk}d^{\bfl_j}d_{q+l-j}v_{\bfk,\bfl}) $ by $({\bfk},{\bfl}^*)$.
We will show that
\begin{equation}\label{Eq3.7}({\bfk},{\bfl}^*)
\preceq (\bf0,\bfj-\epsilon_q),\end{equation} where the equality holds if and only if $(\bfk,\bfl)=(\bf0,\bfj)$.
We have four different   cases to consider.

(i) $j<q$.
Then $q+l-j>l$ and $h^{\bfk}d^{\bfl_j}d_{q+l-j}v_{\bfk,\bfl}=0$.

(ii) $j=q$. Then $q+l-j=l$. Since $d_l$ acts injectively on $V$, we see  $ ({\bfk},{\bfl}^*)=(\bfk,\bfl_q)$ and  $\mathrm{w}(\bfk)+\mathrm{w}(\bfl^*) =\mathrm{w}(\bfk)+\mathrm{w}(\bfl)-q$. If
 ${\rm{w}(\bfk)+\rm{w}(\bfl)}<{\rm{w}(\bf0)+\rm{w}(\bfj)}$, then
$ ({\bfk},{\bfl}^*)
\prec (\bf0,\bfj-\epsilon_q).$
If ${\rm{w}(\bfk)+\rm{w}(\bfl)}={\rm{w}(\bf0)+\rm{w}(\bfj)}$,
there is $\tau\in \M$ such that ${\rm w}(\tau)=q$ and $\bfl_q=\bfl-\tau$. Then $(\bf0,\epsilon_q)\preceq (\bf0, \tau)$. Since $(\bfk,\bfl)\preceq (\bf0,\bfj)$, we see that $$(\bfk ,\bfl^*)=(\bfk ,\bfl)-(\bf0, \tau) \preceq (\bf0 ,\bfj)- (\bf0,\epsilon_q )=
 (\bf0,\bfj-\epsilon_q).$$
In both cases we have that
$$({\bfk},{\bfl}^*)
\preceq (\bf0,\bfj-\epsilon_q),$$
where the equality holds if and only if $(\bfk,\bfl)=(\bf0,\bfj)$.

(iii) $q+1\le j\le q+l$. Then $0\le q+l-j\le l-1$ and $d_{q+l-j}v_{\bfk,\bfl}\in V$. So if $h^{\bfk}d^{\bfl_j}d_{q+l-j}v_{\bfk,\bfl}\neq 0$, then  $\mathrm{w}(\bfk)+\mathrm{w}(\bfl^*) =\mathrm{w}(\bfk)+\mathrm{w}(\bfl)-j<\mathrm{w}(\bfk)+\mathrm{w}(\bfl)-q$.

(iv) $j>q+l$. Then $ q+l-j<0$. Clearly, $\mathrm{w}(\bfl^*) =\mathrm{w}(\bfl_j)+(j-q-l)=\mathrm{w}(\bfl)  -q-l$, and hence
$$\mathrm{w}(\bfk)+\mathrm{w}(\bfl^*) =\mathrm{w}(\bfk)+\mathrm{w}(\bfl)-q-l<\mathrm{w}(\bfk)+\mathrm{w}(\bfl)-q.$$



  Therefore, we conclude that  (\ref{Eq3.7}) holds, i.e., $\sum_{(\bfk,\bfl)}h^{\bfk}[d_{q+l},d^{\bfl}]v_{\bfk,\bfl}$ has degree $({\bf0,\bfj}-\epsilon_q)$.

Next we consider the degree of the nonzero vector $[d_{q+l},h^{\bfk}]d^{\bfl}v_{\bfk,\bfl}$. By (\ref{bracket3.2}) we can see that
$[d_{q+l},h^{\bfk}]d^{\bfl}v_{\bfk,\bfl}$ is a linear combination of some vectors of the forms
 $h^{\bfk_s}h_{q+l-s-n+\frac{1}{2}}d^{\bfl}v_{\bfk,\bfl}, s\in\bZ_+$
and $h^{\bfk_{q+l+1-2n}}d^{\bfl}v_{\bfk,\bfl}$, where $\mathrm{w}(\bfk_s)=\mathrm{w}(\bfk)-s$. Noting that $l\geq 2n$, the degree of $h^{\bfk_{q+l+1-2n}}d^{\bfl}v_{\bfk,\bfl}$
has weight $$\mathrm{w}(\bfk)-(q+l+1-2n)+\mathrm{w}(\bfl)<\mathrm{w}(\bfk)+\mathrm{w}(\bfl)-q.$$
So
$$\deg(h^{\bfk_{q+l+1-2n}}d^{\bfl}v_{\bfk,\bfl})
\prec (\bf0,\bfj-\epsilon_q).$$

Next we will show that \begin{equation}\label{Eq3.8}\deg(h^{\bfk_s}h_{q+l-s-n+\frac{1}{2}}d^{\bfl}v_{\bfk,\bfl})
\prec (\bf0,\bfj-\epsilon_q),\end{equation}
We have two different   cases to consider.

(a) $s>q+l$. The degree of $h^{\bfk_s}h_{q+l-s-n+\frac{1}{2}}d^{\bfl}v_{\bfk,\bfl}$ has weight
 $$\mathrm{w}(\bfk_s)+(s-q-l)+\mathrm{w}(\bfl)=\mathrm{w}(\bfk)+\mathrm{w}(\bfl)-q-l<\mathrm{w}(\bfk)+\mathrm{w}(\bfl)-q.$$
 So, (\ref{Eq3.8}) holds in this case.

(b) $1\le s\le q+l$. We have $$h^{\bfk_s}h_{q+l-s-n+\frac{1}{2}}d^{\bfl}v_{\bfk,\bfl}=h^{\bfk_s}[h_{q+l-s-n+\frac{1}{2}}, d^{\bfl}]v_{\bfk,\bfl}+h^{\bfk_s}d^{\bfl}h_{q+l-s-n+\frac{1}{2}}v_{\bfk,\bfl}.$$
Noting that $h_{q+l-s-n+\frac{1}{2}}v_{\bfk,\bfl}\in V$ (in particular, $h_{q+l-s-n+\frac{1}{2}}v_{\bfk,\bfl}=0$ for $1\le s\le q+l-2n$), we see that if   $h^{\bfk_s}d^{\bfl}h_{q+l-s-n+\frac{1}{2}}v_{\bfk,\bfl}\neq 0$ for $q+l-2n+1\le s\le q+l$, its degree has weight
$$\mathrm{w}(\bfk_s)+\mathrm{w}(\bfl)=\mathrm{w}(\bfk)+\mathrm{w}(\bfl)-s<\mathrm{w}(\bfk)+\mathrm{w}(\bfl)-q.
$$
Now we consider $\deg(h^{\bfk_s}[h_{q+l-s-n+\frac{1}{2}}, d^{\bfl}]v_{\bfk,\bfl})$ which is denoted by $(\tilde{\bfk},\tilde{\bfl})$.

(b1) $1\le s\le q$, that is, $q+l-s-n\geq n$. Then $q+l-s-n+\frac{1}{2}=n+p-\frac{1}{2}$ for some $p\in\bZ_+$ and hence
$s+p=q+l-2n+1\geq q+1$. Thus, by the same arguments in proof of Lemma 3.2, we see 
$$ \mathrm{w}(\tilde{\bfk})+\mathrm{w}(\tilde{\bfl})
\le \mathrm{w}(\bfk_s)+\mathrm{w}(\bfl)-p=\mathrm{w}(\bfk)-s+\mathrm{w}(\bfl)-p
\leq \mathrm{w}(\bfk)+\mathrm{w}(\bfl)-q-1<\mathrm{w}(\bfk)+\mathrm{w}(\bfl)-q.$$
 So, (\ref{Eq3.8}) holds in this case.

(b2) $q+1\le s \le q+l$. Then by (\ref{bracket3.3}) and  the same arguments in proof of Lemma 3.2, we see 
$$\mathrm{w}(\tilde{\bfk})+\mathrm{w}(\tilde{\bfl})\le \mathrm{w}(\bfk_s)+\mathrm{w}(\bfl) =\mathrm{w}(\bfk)+\mathrm{w}(\bfl)-s\le \mathrm{w}(\bfk)+\mathrm{w}(\bfl)-q-1<\mathrm{w}(\bfk)+\mathrm{w}(\bfl)-q.$$
 So, (\ref{Eq3.8}) holds in this case as well.



 Therefore, $ \deg\big(d_{q+l}v)=(\bfz,\bfj-\epsilon_q)$, as desired.
\end{proof}

\begin{lem}\label{main3} Let $k\in\bZ_+, n\in\bZ$ with $  k\ge n$ and $k+n\ge2$, and let $V$ be a $\DD^{(0,-n)}$-module  such that  $h_{k-\frac{1}{2}}$  acts injectively on $V$, and
$h_{i -\frac{1}{2}} V=0$ for all $i>k$.
If $v\in \Ind(V)\setminus V$ with  $\deg'(v)=(\bfi,\bfj)$ and $\bfj\not=\bf0$,  then $\deg'(h_{p+k-\frac{1}{2}}v)=(\bfi,\bfj-\epsilon_p)$ where $p=\min\{s:j_s\neq 0\}$.
\end{lem}

\begin{proof}As in  (\ref{def2.1}), write  $v=\sum_{(\bfk,\bfl)}h^{\bfk}d^{\bfl}v_{\bfk,\bfl}$. Consider $\text{deg}'(h_{p+k-\frac{1}{2}}h^{\bfk}d^{\bfl}v_{\bfk,\bfl})$ if $h_{p+k-\frac{1}{2}}h^{\bfk}d^{\bfl}v_{\bfk,\bfl}\ne 0.$ Noting that $h_{p+k-\frac{1}{2}}V=0,$ we see
$$h_{p+k-\frac{1}{2}}h^{\bfk}d^{\bfl}v_{\bfk,\bfl}=[h_{p+k-\frac{1}{2}},h^{\bfk}]d^{\bfl}v_{\bfk,\bfl}+h^{\bfk}[h_{p+k-\frac{1}{2}},d^{\bfl}]v_{\bfk,\bfl}.$$

First we consider the term  $[h_{p+k-\frac{1}{2}},h^{\bfk}]d^{\bfl}v_{\bfk,\bfl}$ which is zero if   $k_{p'}=0$ for $p':=p+k-n$. In the case that  $k_{p'}>0$, since the level $\ell\ne 0$, it follows from (\ref{bracket3.1}) that $[h_{p+k-\frac{1}{2}},h^{\bfk}]=\lambda h^{\bfk-\epsilon_{p'}}$ for
some  $\lambda\in\bC^*$. Note that
$(\bfk,\bfl)\preceq' (\bfi,\bfj),
 (\bf0, \epsilon_p)\prec'(\epsilon_{p'},\bf0).$
So $$\text{deg}'([h_{p+k-\frac{1}{2}},h^{\bfk}]d^{\bfl}v_{\bfk,\bfl})=(\bfk-\epsilon_{p'},\bfl)=(\bfk,\bfl)-
(\epsilon_{p'},\bf0)\prec'   (\bfi,\bfj)-
(\bf0,\epsilon_p)=(\bfi,\bfj-\epsilon_p).$$

Now we consider the term  $h^{\bfk}[h_{p
+k-\frac{1}{2}},d^{\bfl}]v_{\bfk,\bfl}$ which is by  (\ref{bracket3.3})  a linear combination of some vectors  in the form $h^{\bfk}d^{\bfl_j}h_{p+k-\frac{1}{2}-j}v_{\bfk,\bfl}$ with $j\in\bZ_+$ and ${\rm{w}}
(\bfl_j)={\rm{w}}(\bfl)-j$.   We will show that
\begin{equation}\label{Eq3.6}\deg'(h^{\bfk}d^{\bfl_j}h_{p+k-\frac{1}{2}-j}v_{\bfk,\bfl})=(\bfk^*,\bfl^*)
\preceq' (\bfi,\bfj-\epsilon_p),\end{equation}
 where the equality holds if and only if $(\bfk,\bfl)=(\bfi,\bfj)$.
 We have four   different cases to consider.

(a) $j<p$. Then $p+k-j>n$ and $h_{p+k-\frac{1}{2}-j}v_{\bfk,\bfl}=0$. Hence $h^{\bfk}d^{\bfl_j}h_{p+k-\frac{1}{2}-j}v_{\bfk,\bfl}=0$.

(b) $j=p$. Noting that $h_{k-\frac{1}{2}}$ acts injectively on $V$, we see  $(\bfk^*,\bfl^*)=(\bfk, \bfl_p)$ and  ${\rm{w}}(\bfk^*)+{\rm{w}(\bfl^*)={\rm{w}}(\bfk)+\rm{w}(\bfl)-}p$.

If ${\rm{w}(\bfk)+\rm{w}(\bfl)}<{\rm{w}(\bfi)+\rm{w}(\bfj)}$, then $(\bfk^*,\bfl^*)\preceq' (\bfi,\bfj-\epsilon_p).$

If ${\rm{w}(\bfk)+\rm{w}(\bfl)}={\rm{w}(\bfi)+\rm{w}(\bfj)}$,
then there is $\tau\in \M$ such that ${\rm w}(\tau)=p$ and $\bfl_p=\bfl-\tau$. Since $ (\bf0, \epsilon_p)\preceq' (\bf0, \tau)$ and $(\bfk,\bfl)\preceq' (\bfi,\bfj)$, we see that $$(\bfk^* ,\bfl^*)=(\bfk ,\bfl)-({\bf}0, \tau) \preceq' (\bfi ,\bfj)- ({\bf0},\epsilon_p)=
(\bfi,\bfj-\epsilon_p),$$
 where the equality holds if and only if $(\bfk,\bfl)=(\bfi,\bfj)$.


(c) $p<j<n+k+p$. Then $h_{p+k-\frac{1}{2}-j}\in \mathfrak{D}^{(0,-n)}$ and $h_{p+k-\frac{1}{2}-j}v_{\bfk,\bfl}\in V.$ So $$   {\rm{w}(\bfk^*)+\rm{w}(\bfl^*)}=  {\rm{w}(\bfk)+\rm{w}(\bfl)}-j<{\rm{w}(\bfk)+\rm{w}(\bfl)}-p$$
and $(\bfk^*,\bfl^*)
\prec' (\bfi,\bfj-\epsilon_p)$.

(d) $j\ge n+k+p$. Then $p+k-\frac{1}{2}-j<-n+\frac{1}{2}$. Assume $p+k-\frac{1}{2}-j=-s-n+\frac{1}{2}$ for some $s\in \bZ_+$, that is, $-j+s= -n-k-p+1<-p$ since $k+n\ge2$. Since the corresponding vector
$h^{\bfk}d^{\bfl_j}h_{p+k-\frac{1}{2}-j}v_{\bfk,\bfl}=h^{\bfk}h_{-s-n+\frac{1}{2}}d^{\bfl_j}v_{\bfk,\bfl}-h^{\bfk}[h_{-s-n+\frac{1}{2}},d^{\bfl_j}]v_{\bfk,\bfl},
$  by (\ref{bracket3.3}) and simple computations we see $h^{\bfk}d^{\bfl_j}h_{p+k-\frac{1}{2}-j}v_{\bfk,\bfl}$ can written as a linear combination of vectors in the form
$h^{\bfk}h_{-s'-s-n+\frac{1}{2}}d^{\bfl_{s'+j}}v_{\bfk,\bfl}
$ where $s'\in\bN$ and $\deg'(h^{\bfk}h_{-s'-s-n+\frac{1}{2}}d^{\bfl_{s'+j}}v_{\bfk,\bfl})
$ has weight $${\rm{w}}(\bfk)+s'+s+{\rm{w}}(\bfl_{s'+j})={\rm{w}}(\bfk)+{\rm{w}}(\bfl)+s-j.
$$
So
$$ {\rm{w}(\bfk^*)+\rm{w}(\bfl^*)}= {\rm{w}(\bfk)+\rm{w}(\bfl)}-j+s<{\rm{w}}(\bfk)+{\rm{w}}(\bfl)-p\le {\rm{w}}(\bfi)+{\rm{w}}(\bfj)-p,$$ and hence $(\bfk^*,\bfl^*)
\prec' (\bfi,\bfj-\epsilon_p)$.


In conclusion, $\text{deg}'(h_{p+k-\frac{1}{2}}h^{\bfk}d^{\bfl}v_{\bfk,\bfl})
\preceq' (\bfi,\bfj-\epsilon_p),$
where the equality holds if and only if $(\bfk,\bfl)=(\bfi,\bfj)$, that is,   $\deg'(h_{p+k-\frac{1}{2}}v)=(\bfi,\bfj-\epsilon_p)$.\end{proof}

%
%

\begin{lem}\label{main4}  Let $k\in\bZ_+, n\in\bZ$ such that $k>n$ and $k+n\geq 2$, and $V$ be a $\DD^{(0,-n)}$-module such that $h_{k-\frac{1}{2}}$ acts injectively on $V$, and $h_{i -\frac{1}{2}} V=d_jV=0$ for all $i>k$, $j>k+n-1$.
Assume that $v=\sum_{(\bfk,\bfl)}h^{\bfk}d^{\bfl}v_{\bfk,\bfl}\in \Ind(V)\setminus V$ with  $\deg'(v)=(\bfi, \bf0)$. Set $q=\min\{s: i_s\neq 0\}$.
\begin{enumerate}
\item[\rm(1)] If the sum  $\sum_{(\bfk,\bfl)}h^{\bfk}d^{\bfl}v_{\bfk,\bfl}$ does not contain terms $h^{\bfk}d^{\bfl}v_{\bfk,\bfl}$ satisfying
\begin{equation}\label{weight}
{\rm{w}(\bfk)}+{\rm{w}(\bfl)}={\rm{w}(\bfi)}, {\rm{w}(\bfi)}-q\le  {\rm{w}(\bfk)}<{\rm{w}(\bfi)},
\end{equation}
then $\deg'(d_{q+k+n-1}v)=(\bfi-\epsilon_q, \bf0)$;
\item[\rm(2)] Assume that  the sum $\sum_{(\bfk,\bfl)}h^{\bfk}d^{\bfl}v_{\bfk,\bfl}$ contain terms $h^{\bfk}d^{\bfl}v_{\bfk,\bfl}$ satisfying (\ref{weight}). Let
$v'=v-\sum_{\rm{w}(\bfk)={\rm{w}(\bfi)}}h^{\bfk}v_{\bfk,\bf0}$ and
$\deg'(v')=(\bfk^*, \bfl^*)$ with $t=min\{s: l^*_s\neq 0\}$. Then  $\deg'(h_{k+t-\frac{1}{2}}v)=(\bfk^*, \bfl^*-\epsilon_t)$.
\end{enumerate}
\end{lem}

\begin{proof}
Consider $\text{deg}'(d_{q+k+n-1}h^{\bfk}d^{\bfl}v_{\bfk,\bfl})$ with $d_{q+k+n-1}h^{\bfk}d^{\bfl}v_{\bfk,\bfl}\ne 0.$ Noting that $d_{q+k+n-1}V=0,$ we see that
$$d_{q+k+n-1}h^{\bfk}d^{\bfl}v_{\bfk,\bfl}=[d_{q+k+n-1},h^{\bfk}]d^{\bfl}v_{\bfk,\bfl}+h^{\bfk}[d_{q+k+n-1},d^{\bfl}]v_{\bfk,\bfl}.$$

First we consider the term  $[d_{q+k+n-1},h^{\bfk}]d^{\bfl}v_{\bfk,\bfl}$. It follows from (\ref{bracket3.2}) that $[d_{q+k+n-1},h^{\bfk}]d^{\bfl}v_{\bfk,\bfl}$ is a linear combination of vectors in the forms $h^{\bfk_j}h_{(q-j)+k-\frac{1}{2}}d^{\bfl}v_{\bfk,\bfl}$ and $h^{\bf{s}}d^{\bfl}v_{\bfk, \bfl}$
where $\bfk_j=\bfk-\epsilon_j$, ${\rm{w}}({\bf{s}})
={\rm{w}}(\bfk)-(k+q-n)$.
If $\bfl=0$, it is not hard to  see that $\text{deg}'(d_{q+k+n-1}h^{\bfk}d^{\bfl}v_{\bfk,\bfl})\preceq' (\bfi-\epsilon_q,\bf0)$  where the equality holds if and only if $(\bfk, \bfl)=(\bfi,\bf0)$.

Next we assume that $\bfl\ne\bf0$, and continue to consider the term  $[d_{q+k+n-1},h^{\bfk}]d^{\bfl}v_{\bfk,\bfl}$.
We first consider the term $h^{\bfk_j}h_{(q-j)+k-\frac{1}{2}}d^{\bfl}v_{\bfk,\bfl}$. We break the arguments into to  four   different cases next.

(a) $j<q$. In this case, we have $h^{\bfk_j}h_{(q-j)+k-\frac{1}{2}}d^{\bfl}v_{\bfk,\bfl}=h^{\bfk_j}[h_{(q-j)+k-\frac{1}{2}}, d^{\bfl}]v_{\bfk,\bfl}$. Then it follows from (\ref{bracket3.3}) that $h^{\bfk_j}[h_{(q-j)+k-\frac{1}{2}}, d^{\bfl}]v_{\bfk,\bfl}$ is a linear combination of vectors  in the form
$h^{\bfk_j} d^{\bfl_s}h_{(q-j-s)+k-\frac{1}{2}}v_{\bfk,\bfl}$ where ${\rm{w}}(\bfl_s)={\rm{w}}(\bfl)-s$.

(a1) If $s<q-j$, then $h^{\bfk_j} d^{\bfl_s}h_{(q-j-s)+k-\frac{1}{2}}v_{\bfk,\bfl}=0.$

(a2) If $s=q-j$, then $\deg'(h^{\bfk_j} d^{\bfl_s}h_{k-\frac{1}{2}}v_{\bfk,\bfl})$ has weight
 $${\rm{w}}(\bfk_j)+{\rm{w}}(\bfl_s)={\rm{w}}(\bfk)+{\rm{w}}(\bfl)-j-s={\rm{w}}(\bfk)+{\rm{w}}(\bfl)-q.$$ If ${\rm{w}}(\bfk)+{\rm{w}}(\bfl)<{\rm{w}}(\bfi)$, or  ${\rm{w}}(\bfk)+{\rm{w}}(\bfl)={\rm{w}}(\bfi)$ and ${\rm{w}}(\bfk)<{\rm{w}}(\bfi)-q$, then $\deg'(h^{\bfk_j} d^{\bfl_s}h_{k-\frac{1}{2}}v_{\bfk,\bfl})\prec' (\bfi-\epsilon_q,\bf0).$ We will discuss the remaining cases that
  $(\bfk, \bfl)$ satisfies (3.9) in Case (2)  later.

(a3) If $q-j<s\le q+k+n-1-j$, then $ h_{(q-j-s)+k-\frac{1}{2}}v_{\bfk,\bfl}\in V$ and $\deg'(h^{\bfk_j} d^{\bfl_s}h_{k-\frac{1}{2}}v_{\bfk,\bfl})$ has weight
 $${\rm{w}}(\bfk_j)+{\rm{w}}(\bfl_s)={\rm{w}}(\bfk)+{\rm{w}}(\bfl)-j-s<{\rm{w}}(\bfk)+{\rm{w}}(\bfl)-q\le {\rm{w}}(\bfi)-q.$$ So
 $\deg'(h^{\bfk_j} d^{\bfl_s}h_{(q-j-s)+k-\frac{1}{2}}v_{\bfk,\bfl})\prec' (\bfi-\epsilon_q,\bf0).$

(a4) If $s>q+k+n-1-j$, then $q-j-s+k-\frac{1}{2}=-s'-n+\frac{1}{2} $ for some $s'\in\bZ_+$. It is easy to see
 $h^{\bfk_j} d^{\bfl_s}h_{(q-j-s)+k-\frac{1}{2}}v_{\bfk,\bfl}$ can be written as a linear combination of vectors of the form
$h^{\bfk_j} h_{-s'-s''-n+\frac{1}{2}}d^{\bfl_{s+s''}}v_{\bfk,\bfl}, 0\le s''\le {\rm{w}}(\bfl_s)$. Note that both
$\deg'(h^{\bfk_j} h_{-s'-s''-n+\frac{1}{2}}d^{\bfl_{s+s''}}v_{\bfk,\bfl})$ and $\deg'(h^{\bfk_j} h_{-s'-n+\frac{1}{2}}d^{\bfl_{s}}v_{\bfk,\bfl})$ have the same weight and   $-j-s+s'=-q-k-n+1<-q$, we see
 $\deg'(h^{\bfk_j} d^{\bfl_s}h_{(q-j-s)+k-\frac{1}{2}}v_{\bfk,\bfl})$ has weight
 $${\rm{w}}(\bfk_j)+{\rm{w}}(\bfl_s)+s'={\rm{w}}(\bfk)+{\rm{w}}(\bfl)-j-s+s'<{\rm{w}}(\bfk)+{\rm{w}}(\bfl)-q\le {\rm{w}}(\bfi)-q.$$So
 $\deg'(h^{\bfk_j} d^{\bfl_s}h_{(q-j-s)+k-\frac{1}{2}}v_{\bfk,\bfl})\prec' (\bfi-\epsilon_q,\bf0).$


(b) $j=q$. In this case, we have $h^{\bfk_q}h_{k-\frac{1}{2}}d^{\bfl}v_{\bfk,\bfl}=h^{\bfk_q}d^{\bfl}h_{k-\frac{1}{2}}v_{\bfk,\bfl}+h^{\bfk_q}[h_{k-\frac{1}{2}}, d^{\bfl}]v_{\bfk,\bfl}$.
Clearly, $\text{deg}'(h^{\bfk_q}d^{\bfl}h_{k-\frac{1}{2}}v_{\bfk,\bfl})=(\bfk_q,\bfl)\prec' (\bfi-\epsilon_q,\bf0)$ since $\bfl\ne\bf0$. By (\ref{bracket3.3}) and the similar arguments in Cases (a3) and (a4) we can deduce that
$\text{deg}'(h^{\bfk_q}[h_{k-\frac{1}{2}},d^{\bfl}]v_{\bfk,\bfl})\prec' (\bfi-\epsilon_q,\bf0)$. Hence
$\text{deg}'(h^{\bfk_q}h_{k-\frac{1}{2}}d^{\bfl}v_{\bfk,\bfl})\prec' (\bfi-\epsilon_q,\bf0)$.

(c) $q<j\leq q+k+n-1$. In this case, we have $h^{\bfk_j}h_{(q-j)+k-\frac{1}{2}}d^{\bfl}v_{\bfk,\bfl}=h^{\bfk_j}d^{\bfl}h_{(q-j)+k-\frac{1}{2}}v_{\bfk,\bfl}+h^{\bfk_j}[h_{(q-j)+k-\frac{1}{2}}, d^{\bfl}]v_{\bfk,\bfl}$. Clearly, $\deg'(h^{\bfk_j}d^{\bfl}h_{(q-j)+k-\frac{1}{2}}v_{\bfk,\bfl})= {\rm{w}}(\bfk)+{\rm{w}}(\bfl)-j<{\rm{w}}(\bfi)-q.$
Then by (\ref{bracket3.3}) and the similar arguments in Cases (a3) and (a4) we can deduce that
$\text{deg}'(h^{\bfk_q}[h_{(q-j)+k-\frac{1}{2}}, d^{\bfl}]v_{\bfk,\bfl})\prec' (\bfi-\epsilon_q,\bf0)$.
Hence, $\text{deg}'(h^{\bfk_j}h_{(q-j)+k-\frac{1}{2}}d^{\bfl}v_{\bfk,\bfl})\prec' (\bfi-\epsilon_q,\bf0)$.

(d) $j>q+k+n-1$. In this case, we have $h^{\bfk_j}h_{(q-j)+k-\frac{1}{2}}d^{\bfl}v_{\bfk,\bfl}=h^{\bfk_j}h_{-(j-(q+k+n-1))-n+\frac{1}{2}}d^{\bfl}v_{\bfk,\bfl}$.
Then  $\text{deg}'(h^{\bfk_j}h_{(q-j)+k-\frac{1}{2}}d^{\bfl}v_{\bfk,\bfl})=(\bfk^*,\bfl)$ with weight  ${\rm{w}}(\bfk^*)+{\rm{w}}(\bfl)= {\rm{w}}(\bfk)+{\rm{w}}(\bfl)
-(q+k+n-1)<{\rm{w}}(\bfi)-q$. Hence, $\text{deg}'(h^{\bfk_j}h_{(q-j)+k-\frac{1}{2}}d^{\bfl}v_{\bfk,\bfl})\prec' (\bfi-\epsilon_q,\bf0)$.

Next consider the term $h^{\bf{s}}d^{\bfl}v_{\bfk, \bfl}$. Since ${\rm{w}}(\text{deg}'(h^{\bf{s}}d^{\bfl}v_{\bfk, \bfl}))={\rm{w}}(\bf{s})+{\rm{w}}(\bfl)<
{\rm{w}}(\textbf{k})+{{\rm{w}}(\textbf{l})}$-$q\leq {\rm{w}}(\textbf{i})$-$q$, it follows that $\text{deg}'(h^{\bf{s}}d^{\bfl}v_{\bfk, \bfl})\prec' (\bfi-\epsilon_q,\bf0)$.

Thus, if  $h^{\bfk}d^{\bfl}v_{\bfk,\bfl}$ does not satisfy (\ref{weight}) we have
$$\deg'([d_{q+k+n-1},h^{\bfk}]d^{\bfl}v_{\bfk,\bfl})\preceq' (\bfi-\epsilon_q,\bf0)$$
where the equality holds if and only if $(\bfk, \bfl)=(\bfi,\bf0)$.

Now, consider the term $h^{\bfk}[d_{q+k+n-1},d^{\bfl}]v_{\bfk,\bfl}$ where we still assume that $\bfl\ne\bf0$. By (\ref{bracket3.4}) we see $h^{\bfk}[d_{q+k+n-1},d^{\bfl}]v_{\bfk,\bfl}$ is a linear combination of vectors
$h^{\bfk}d^{\bfl_j}d_{q+k+n-1-j}v_{\bfk,\bfl}$ and  $h^{\bfk}d^{\bfl_{q+k+n-1}}v_{\bfk,\bfl}$ where $\rm{w}({\bfl_j})=\rm{w}({\bfl})-j,j\in\bN$. Since $\deg'(h^{\bfk}d^{\bfl_{q+k+n-1}}v_{\bfk,\bfl})$ has weight
$${\rm{w}(\bfk)}+{\rm{w}}({\bfl})-(q+k+n-1)<{\rm{w}(\bfk)}+{\rm{w}}({\bfl})-q\le {\rm{w}}(\bfi)-q,$$ we see $\deg'(h^{\bfk}d^{\bfl_{q+k+n-1}}v_{\bfk,\bfl})\prec' (\bfi-\epsilon_q,\bf0)$. So we need only to consider the vectors $h^{\bfk}d^{\bfl_j}d_{q+k+n-1-j}v_{\bfk,\bfl}$. There are four different cases.

(i) $j<q$. Then $q+k+n-1-j>k+n-1$ and $h^{\bfk}d^{\bfl_j}d_{q+k+n-1-j}v_{\bfk,\bfl}=0$. In particular, for  ${\rm{w}(\bfl)}<q$ we have $h^{\bfk}d^{\bfl_j}d_{q+k+n-1-j}v_{\bfk,\bfl}=0$.

(ii) $j=q$. Then $q+k+n-1-q=k+n-1$ and hence  $\deg'(h^{\bfk}d^{\bfl_q}d_{k+n-1}v_{\bfk,\bfl})=(\bfk,\bfl_q)$ ( in the case $d_{k+n-1}v_{\bfk,\bfl}\ne 0$ ) with ${\rm{w}(\bfk)+\rm{w}(\bfl_q)=\rm{w}(\bfk)+\rm{w}(\bfl)}-q$.

If $\rm{w}(\bfk)+\rm{w}(\bfl)<\rm{w}(\bfi)$, or $\rm{w}(\bfk)+\rm{w}(\bfl)=\rm{w}(\bfi)$ and ${\rm{w}(\bfk)<\rm{w}(\bfi)}-q$,
then $(\bfk,\bfl_q)\prec' (\bfi-\epsilon_q,\bf0).$
We will discuss the remaining cases that
  $(\bfk, \bfl)$ satisfies (3.9) in Case (2)  later.

(iii) $q<j\le q+k+n-1$. Then $d_{q+k+n-1-j}v_{\bfk,\bfl}\in V$ and $h^{\bfk}d^{\bfl_j}d_{q+k+n-1-j}v_{\bfk,\bfl}=0$ or $\deg'(h^{\bfk}d^{\bfl_j}d_{q+k+n-1-j}v_{\bfk,\bfl})$ has weight
$${\rm{w}(\bfk)+\rm{w}(\bfl_j)=\rm{w}(\bfk)+\rm{w}(\bfl)}-j<{\rm{w}(\bfk)+\rm{w}(\bfl)}-q\le {\rm{w}}(\bfi)-q,$$
so $\deg'(h^{\bfk}d^{\bfl_j}d_{q+k+n-1-j}v_{\bfk,\bfl})\prec'  (\bfi-\epsilon_q,\bf0)$.

(iv) $j>q+k+n-1$. Then $q+k+n-1-j<0$. Assume $q+k+n-1-j=-j'$, $j'\in \bZ_+$. Then $-j+j'=-(q+k+n-1)<-q$. So $\deg'(h^{\bfk}d^{\bfl_j}d_{q+k+n-1-j}v_{\bfk,\bfl})$ has weight
$${\rm{w}(\bfk)+\rm{w}(\bfl_j)+j'=\rm{w}(\bfk)+\rm{w}(\bfl)}-j+j'={\rm{w}(\bfk)+\rm{w}(\bfl)}-(q+k+n-1)
< {\rm{w}}(\bfi)-q,$$
which means $\deg'(h^{\bfk}d^{\bfl_j}d_{q+k+n-1-j}v_{\bfk,\bfl}) \prec' (\bfi-\epsilon_q,\bf0).$

(1)  If $v=\sum_{(\bfk,\bfl)}h^{\bfk}d^{\bfl}v_{\bfk,\bfl}$ does not contain a term $h^{\bfk}d^{\bfl}v_{\bfk,\bfl}$ satisfying (\ref{weight}),
then  by the above arguments we see $\deg'(d_{q+k+n-1}v)=(\bfi-\epsilon_q, \bf0)$.

(2) If $v=\sum_{(\bfk,\bfl)}h^{\bfk}d^{\bfl}v_{\bfk,\bfl}$ contain  terms $h^{\bfk}d^{\bfl}v_{\bfk,\bfl}$ satisfying  (\ref{weight}), then we see $\deg'(v')=(\bfk^*,\bfl^*)$ with $${\rm{w}(\bfk^*)+\rm{w}(\bfl^*)=\rm{w}(\bfi), }\,
 {\rm{w}(\bfk^*)}\ge {\rm{w}(\bfi)}-q, \,1\le {\rm{w}}{(\bfl^*)}\le q.$$
 Then by Lemma \ref{main3} we see $\deg'(h_{t+k-\frac{1}{2}}v')=(\bfk^*, \bfl^*-\epsilon_t)$.

 Noticing that $k>n$, by (\ref{bracket3.1}) we see $h_{t+k-\frac{1}{2}}h^{\bfk}v_{\bfk,\bf0}=\bf0$ or $\lambda h^{\bfk_{t'}}v_{\bfk,\bf0},\lambda\in\bC^*$ with $t'=t+k-n>t$ and  ${\rm{w}(\bfk_{t'})=\rm{w}(\bfk)}-t'$, so
$\deg'(h_{t+k-\frac{1}{2}}(h^{\bfk}v_{\bfk,\bf0} ))=(\bfk_{t'},\bf0)$ has weight ${\rm{w}(\bfk_{t'})=\rm{w}(\bfk)}-t'<{\rm{w}(\bfk^*)+\rm{w}(\bfl^*)}-t={\rm{w}(\bfk^*)+\rm{w}}(\bfl^*-\epsilon_t)$. Hence
$$\deg'(h_{t+k-\frac{1}{2}}v)=\deg'\Big(h_{t+k-\frac{1}{2}}\Big(v-\sum_{\rm{w}(\bfk)={\rm{w}(\bfi)}}h^{\bfk}v_{\bfk,\bf0}\Big)\Big)=(\bfk^*, \bfl^*-\epsilon_t).$$
\end{proof}

\section{Simple restricted $\DD$-modules}\label{char}

In this section we will determine all simple restricted ${\mathfrak{D}}$-modules.
Based on Theorem \ref{MT}, we only need to determine all   simple  restricted ${\mathfrak{D}}$-modules $S$ of level $\ell\ne0$.


For a given simple restricted $\DD$-module $S$ with level $\ell \not=0$, we define the following invariants of $S$ as follows:
$$S(r)={\text{Ann}}_S(\mh^{(r)}), n_S=\min\{r\in \bZ:S(r)\ne0\}, W_0=S{(n_S)}, $$
and
$$U(r)={\text{Ann}}_{W_0}(\Vir^{(r)}), m_S=\min\{r\in \bZ:U(r)\ne0\}, U_0=U(m_S).$$

\begin{lem}\label{lem4.2} Let $S$ be a
simple restricted $\DD$-module  with level $\ell \not=0$.
\begin{enumerate}
\item[\rm(i)] $h_{n_S-\frac12}$ acts injectively on $W_0$,  $d_{m_S-1}$ acts injectively on $U_0$.
\item[\rm(ii)] $n_S, m_S\in\bN$.
\item[\rm(iii)] $W_0$ is a nonzero $\DD^{(0,-n_S)}$-module, and is invariant under the action of the operators $L_n$  defined in (\ref{rep1})-(\ref{rep3}) for $n\in\bN$.
\item[\rm(iv)] If $m_S\ge 2n_S$, then $U_0$ is a nonzero $\DD^{(0,-n_S)}$-submodule of $W_0$, and is invariant under the action of the operators $L_n$  defined in (\ref{rep1})-(\ref{rep3}) for $n\in\bN$.
\end{enumerate}
\end{lem}

\begin{proof}
(i) follows from the definitions of $n_S$ and $m_S$.

(ii) Suppose $n_S<0$, take any nonzero $v\in W_0$, we then have
$$h_{\frac12}v=0=h_{-\frac12}v.$$
This implies that $\frac12\ell v=[h_{\frac12},h_{-\frac12}]v=0$, a contradiction.  Hence, $n_S\in\bN$.

Suppose $m_S<0$. Take any nonzero $v\in U_0$, we then have $d_{-1}v=0=h_{n_S+\frac{1}{2}}v$. Then
$$-(n_S+\frac{1}{2})h_{n_S-\frac{1}{2}}v=[d_{-1}, h_{n_S+\frac{1}{2}}]v=0,$$
a contradiction with (1).  Hence, $m_S\in\bN$.

(iii) It is obvious that $W_0\neq 0$ by definition. For any $w\in W_0$, $i, j, k\in\bN$, we have
$$h_{k+n_S+\frac{1}{2}}d_iw=d_ih_{k+n_S+\frac{1}{2}}w+(k+n_S+\frac{1}{2})h_{i+k+n_S+\frac{1}{2}}w=0,$$
and
$$h_{k+n_S+\frac{1}{2}}h_{j-n_S+\frac{1}{2}}w=h_{j-n_S+\frac{1}{2}}h_{k+n_S+\frac{1}{2}}w=0.$$
Hence,  $d_iu\in W_0$ and $h_{j-n_S+\frac{1}{2}}u\in W_0$, i.e., $W_0$ is a nonzero $\DD^{(0,-n_S)}$-module.

For $n\in\bN, i\in\bN$, $w\in W_0$, by (\ref{rep4}) we have
\begin{eqnarray*}
h_{i+n_S+\frac{1}{2}}L_nw&=\Big(L_nh_{i+n_S+\frac{1}{2}}-(i+n_S+\frac{1}{2})h_{n+i+n_S+\frac{1}{2}}\Big)w=0.
\end{eqnarray*}
This implies that $L_iw\in W_0$ for $i\in\bN$, that is, $W_0$ is invariant under the action of the operators $L_i$ for $i\in\bN$.

(iv) It is obvious that $0\neq U_0\subseteq W_0$. Suppose that $m_S\geq 2n_S$. For any $u\in U_0$, $i, j, k\in\bN$, it follows from (iii) that $d_iu\in W_0$ and $h_{j-n_S+\frac{1}{2}}u\in W_0$. Furthermore,
$$d_{k+m_S}d_iu=d_id_{k+m_S}u+(k-i-m_S)d_{k+i+m_S}u=0,$$
and
$$d_{k+m_S}h_{j-n_S+\frac{1}{2}}u=h_{j-n_S+\frac{1}{2}}d_{k+m_S}u-(j-n_S+\frac{1}{2})h_{k+j+m_S-n_S+\frac{1}{2}}u=0.$$
Hence, $d_iu\in U_0$ and $h_{j-n_S+\frac{1}{2}}u\in U_0$, i.e., $U_0$ is a nonzero $\DD^{(0,-n_S)}$ submodule of $W_0$.

Furthermore, if in addition $m_S>0$, then for $n, i\in\bN$, $u\in U_0$, it follows from (iii) that $L_nu\in W_0$. Moreover, for $n\in\bN$, using (\ref{rep1}-\ref{rep4'}) we have
\begin{eqnarray*}
d_{i+m_S}L_nu=L_nd_{i+m_S} u+[d_{i+m_S},L_n]u
=[d_{i+m_S},L_n]u
=(n-i-m_S)L_{i+n+m_S}u=0.
\end{eqnarray*}
This implies that $L_iu\in U_0$ for $i\in\bN$, that is, $U_0$ is invariant under the action of the operators $L_i$ for $i\in\bN$.
\end{proof}

\begin{pro}\label{prop4.3} Let $S$ be a
simple restricted $\DD$-module  with level $\ell \not=0$.
\begin{enumerate}
\item[\rm(i)]   If $n_S=0$, then  $S\cong H^{\mathfrak{D}}\otimes U^{\mathfrak{D}}$ as $\mathfrak{D}$-modules for some  simple modules $H\in \mathcal{R}_{\mh}$ and  $U\in \mathcal{R}_{\Vir}$.
\item[\rm(ii)] If  $m_S>2n_S>0$, then $S\cong \Ind^{\DD}_{\DD^{(0,-n_S)}}(U_0)$ and $U_0$ is a simple $\DD^{(0,-n_S)}$-module.
\item[\rm(iii)] If $m_S< 2n_S$, then $U_0$ is a nonzero $\DD^{(0,-(m_S-n_S))}$-submodule of $W_0$. Moreover,
\begin{enumerate}
\item[\rm(iii-1)] If $m_S\geq 2$, then  $S\cong \Ind^{\DD}_{\DD^{(0,-(m_S-n_S))}}(U_0)$  and $U_0$ is a simple  $\DD^{(0,-(m_S-n_S))}$-module.
\item[\rm(iii-2)] If $m_S=0$ or $1$,  and $n_S>1$, then $U(2)$ is a simple $\DD^{(0,-(2-n_S))}$-module, and $S\cong \Ind^{\DD}_{\DD^{(0,-(2-n_S))}}(U(2))$.
\end{enumerate}
\end{enumerate}
\end{pro}

\begin{proof}
(i) Since  $n_S=0$, we  take any nonzero $v\in W_0$. Then $\bC v$ is a trivial $\mh^{(0)}$-module. Let $H=U(\mh)v$, the $\mh$-submodule of $S$ generated by $v$. It follows from representation theory of Heisenberg algebras  (or from the same arguments as in the proof of Lemma \ref{main1'}) that $\Ind^{\mh}_{\mh^{(0)}}(\bC v)$ is a simple $\mh$-module.  Consequently, the following surjective $\mh$-module homomorphism
\begin{eqnarray*}
\varphi:\, \Ind^{\mh}_{\mh^{(0)}}(\bC v) &\longrightarrow & H\\
\sum_{\mi\in\mathbb{M}}a_{\mi} h^{\mi}\otimes v&\mapsto &
\sum_{\mi\in\mathbb{M}} a_{\mi} h^{\mi} v
\end{eqnarray*}
is an isomorphism, that is, $H$ is a simple $\mh$-module, which is certainly restricted. Then the desired assertion follows directly from Corollary \ref{tensor}.

(ii)  By taking $V=U_0$, $k=n=n_S$ and $l=m_S-1$ in  Theorem \ref{thmmain1.1}(1) we see that  any nonzero $\DD$-submodule
of  $\Ind^{\DD}_{\DD^{(0,-n_S)}}(U_0)$ has a nonzero intersection with $U_0$.  Consequently, the surjective $\DD$-module homomorphism
\begin{eqnarray*}
\varphi:\, \Ind^{\DD}_{\DD^{(0,-n_S)}}(U_0) &\longrightarrow & S\\
\sum_{\mi, \mk\in\mathbb{M}} h^{\mi}d^{\mk}\otimes v_{\mi, \mk}&\mapsto &
\sum_{\mi, \mk\in\mathbb{M}} h^{\mi}d^{\mk} v_{\mi, \mk}
\end{eqnarray*}
is an isomorphism, i.e., $S\cong \Ind^{\DD}_{\DD^{(0,-n_S)}}(U_0)$. Since $S$ is simple, we see  $U_0$ is a simple $\DD^{(0,-n_S)}$-module.

(iii) Suppose that  $m_S< 2n_S$. For any $u\in U_0$, $i, j, k\in\bN$, it follows from Lemma \ref{lem4.2} (iii) that $d_iu\in W_0$ and $h_{j-(m_S-n_S)+\frac{1}{2}}u\in W_0$. Furthermore,
$$d_{k+m_S}d_iu=d_id_{k+m_S}u+(k-i+m_S)d_{k+i+m_S}u=0,$$
and
$$d_{k+m_S}h_{j-(m_S-n_S)+\frac{1}{2}}u=h_{j-(m_S-n_S)+\frac{1}{2}}d_{k+m_S}u-(j-(m_S-n_S)+\frac{1}{2})h_{k+j+n_S+\frac{1}{2}}u=0.$$
Hence, $d_iu\in U_0$ and $h_{j-(m_S-n_S)+\frac{1}{2}}u\in U_0$, i.e., $U_0$ is a nonzero $\DD^{(0,-(m_S-n_S))}$ submodule of $W_0$.

Now suppose $m_S\geq 2$. Then it follows from Theorem \ref{thmmain1.1}(2) that
 any nonzero $\DD$-submodule
of  $\Ind^{\DD}_{\DD^{(0,-(m_S-n_S))}}(U_0)$ has a nonzero intersection with $U_0$  by taking $k=n_S, n=m_S-n_S$ and $l=m_S-1$ therein. Consequently, $S\cong \Ind^{\DD}_{\DD^{(0,-(m_S-n_S))}}(U_0)$ by similar arguments as in (ii). Since $S$ is simple, we see  $U_0$ is a simple $\DD^{(0,-n_S)}$-module.

Suppose that $m_S=0$ or $1$,  and $n_S>1$. Then $\DD^{(0,-(2-n_S))}\subseteq \DD^{(0,-n_S)}$. Hence,  $W_0$ is a $\DD^{(0,-(2-n_S))}$-module. Moreover, for any $u\in U(2)$, $i, j\in\bN$, we have
$$d_{j+2}d_{i}u=d_{i}d_{j+2}u=0,$$
and
$$d_{j+2}h_{i-(2-n_S)+\frac{1}{2}}u=h_{i-(2-n_S)+\frac{1}{2}}d_{j+2}u+(2-n_S-i-\frac{1}{2})h_{i+j+n_S+\frac{1}{2}}u=0.$$
Therefore, $U(2)$ is a $\DD^{(0,-(2-n_S))}$-module. Then it follows from Theorem \ref{thmmain1.1}(2) that
 any nonzero $\DD$-submodule
of $\Ind^{\DD}_{\DD^{(0,-(2-n_S))}}(U(2))$ has a nonzero intersection with $U(2)$
 by taking $V=U(2)$, $k=n_S$, $n=2-n_S$ and $l=1$ therein.  Consequently, $S\cong \Ind^{\DD}_{\DD^{(0,-(2-n_S))}}(U(2))$ by similar arguments as in (ii). In particular, $U(2)$ is a simple $\DD^{(0,-(2-n_S))}$-module.
\end{proof}

From Proposition \ref{prop4.3}, what  remains to consider are the following two cases:
(1)  $m_S=2n_S>0$, (2) $m_S=0$ or $1$, and $n_S=1$.

Now we first consider Case (1): $m_S=2n_S>0$. For that, we define the operators $d_n'=d_n-L_n$ on $S$ for $n\in\bZ$.  Since $S$ is a restricted $\DD$-module, then $d_n'$ is well-defined for any $n\in\bZ$. By (\ref{rep4}) and (\ref{rep4'}), we have
\begin{equation}\label{vir-bracket}
[d_m',{\bf c}_1]=0,
[d_m',d_n']=(m-n)d_{m+n}'+\frac{m^3-m}{12}\delta_{m+n,0}(c-1), m,n\in\bZ,
\end{equation}where  ${\bf c}'_1={\bf c}_1-\text{id}_S$ and $c$ is the central charge of $S$. So the operator algebra
$$\Vir'=\bigoplus_{n\in\bZ}\bC d_n'\oplus \bC{\bf c}'_1$$
 is isomorphic to the Virasoro algebra $\Vir$.  Since $[d_n,h_{k+\frac{1}{2}}]=[L_n,h_{k+\frac{1}{2}}]=-({k+\frac{1}{2}})h_{n+k+\frac{1}{2}},$ we have $[d'_n,h_{k+\frac{1}{2}}]=0, n,k\in\bZ$ and hence  $[\Vir',\mh]=0$. Clearly, the operator algebra
 $\DD'=\Vir'\oplus  \mh$ is a direct sum, and $S=\UU(\DD)v=\UU(\DD')v, 0\ne v\in S$.
 Similar to (\ref{Natations}) we can define its subalgebras, $\DD'^{(m,n)} $ and the likes.

 Let
$$Y_n=\bigcap_{p\ge n}{\rm Ann}_{U_0}(d_p'),       r_S=\min\{n\in\bZ:Y_n\ne0\} , K_0=Y_{r_S}. $$
If $Y_n\ne0$ for any $n\in\bZ$, we define $r_S=-\infty$.
Denote by $K=U(\mh)K_0$.

\begin{lem}\label{lem4.4}
Let $S$ be a simple restricted $\DD$-module  with level $\ell \not=0$. Assume that $m_S=2n_S>0$. Then the following statements hold.
\begin{enumerate}
\item[\rm(i)] $-1\le r_S\le m_S$ or $r_S=-\infty$.
\item[\rm(ii)] $K_0$ is a $\DD^{(0,-n_S)}$-module  and $h_{n_S-\frac{1}{2}}$ acts injectively on $K_0$.
\item[\rm(iii)] $K$ is a $\DD^{(0,-\infty)}$-module and $K^{\DD}$ has a $\DD$-module structure by (\ref{rep1})-(\ref{rep3}).
\item[\rm{(iv)}] $K_0$ and $K$ are invariant under the action of $d_n'$ for $n\in\bN$.
\item[\rm(v)] If $r_S\ne -\infty$, then $d'_{r_S-1}$ acts injectively on $K_0$ and $K$.
\end{enumerate}
\end{lem}

\begin{proof} (i)  Since $m_S=2n_S>0$, the operators $d_m$ and  $L_{m}=\frac{1}{2\ell}\sum_{k\in\bZ+\frac{1}{2}}h_{m-k}h_k$ act trivially on $U_0$ for any $m\geq m_S$. This implies that $Y_{m_S}=U_0\neq 0$. Consequently, $r_S\leq m_S$ by the definition of $r_S$.

If $Y_{-2}\ne 0$, then $d'_{-2}K_0=d'_{-1}K_0=0$. We deduce that $\Vir' K_0=0$ and hence $r_S=-\infty$.

If $Y_{-2}=0$, then $r_S\ge -1$ and hence  $-1\le r_S\le m_S$.

(ii) For any $0\ne v\in K_0$ and $x\in {\DD^{(0,-n_S)}}$, it follows from Lemma \ref{lem4.2}(iv) that $xv\in U_0$. We need to show that $d'_pxv=0, p\ge r_S$. Indeed, $d_p'h_{k+\frac{1}{2}}v=h_{k+\frac{1}{2}}d_p'v=0$ by (\ref{rep4}) for any $k\geq -n_S$. Moreover, it follows from (\ref{rep4'}) and  (\ref{vir-bracket}) that
$$d_p'd_nv=d_nd_p'v+[d_p', d_n]v=(p-n)d_{p+n}'v=0.$$
Hence, $d'_pxv=0, p\ge r_S$, that is, $xv\in K_0$, as desired.

Since $0\ne K_0  \subseteq U_0\subseteq W_0$, we see that  $h_{n_S-\frac{1}{2}}$ acts injectively on $K_0$ by Lemma \ref{lem4.2}(i).

(iii) follows from (ii).

(iv) It follows from Lemma \ref{lem4.2}(iv)   that $U_0$ is invariant under the action of $d_n'$ for $n\in\bN$, so is $K_0$ by (\ref{vir-bracket}). Moreover, since $[\Vir',\mh]=0$, $K$ is also is invariant under the action of $d_n'$ for $n\in\bN$.

(v) follows directly from the definition of $r_S$ and $K$.
\end{proof}

\begin{pro}\label{prop for -inf}Let $S$ be a simple restricted $\DD$-module with central charge $c$ and level $\ell \not=0$. Assume that $m_S=2n_S>0$. If $r_S=-\infty$, then $c=1$. Moreover,
$S= K^{\DD}$ and $K$ is a simple $\mh$-module.
\end{pro}
\begin{proof}
Since $r_S=-\infty$, we see that $\Vir' K_0=0$. This together with (\ref{vir-bracket}) implies that $c=1$.  Noting that $[\Vir',\mh]=0$, we further obtain that $\Vir' K=0$, that is,  $d_nv=L_nv\in K$ for any $v\in K$ and $n\in\bZ$. Hence $K^{\DD}$ is a $\DD$-submodule of $S$, yielding that $S= K^{\DD}$. In particular, $K$ is a simple $\mh$-module.
\end{proof}

\begin{pro}\label{prop4.6}
Let $S$ be a simple restricted $\DD$-module  with level $\ell \not=0$. If $r_S\ge 2$, then
$K_0$ is a simple $\DD^{(0,-n_S)}$-module and
$S\cong \Ind_{\DD^{(0,-n_S)}}^{\DD}K_0$.
\end{pro}

\begin{proof}We first show that  $\Ind_{\DD^{(0,-n_s)}}^{\DD^{(0,-\infty)}}K_0 \cong K$ as  $\DD^{(0,-\infty)}$ modules. For that, let
\begin{eqnarray*}
\phi:\, \Ind_{\DD^{(0,-n_s)}}^{\DD^{(0,-\infty)}}K_0  &\longrightarrow &K\\
\sum_{\bfk\in\mathbb{M}} h^{\bfk}\otimes v_{\bfk}&\mapsto &
\sum_{\bfk\in\mathbb{M}} h^{\bfk} v_{\bfk},
\end{eqnarray*} where $h^{\bfk}=\cdots h^{k_2}_{-2-n_S+\frac{1}{2}}h_{-1-n_S+\frac{1}{2}}^{k_1}$. Then $\phi$ is a  $\DD^{(0,-\infty)}$-module epimorphism and $\phi|_{K_0}$ is one-to-one. By similar arguments in the proof of Lemma \ref{main1'} we see that any nonzero submodule of $\Ind_{\DD^{(0,-n_s)}}^{\DD^{(0,-\infty)}}K_0$ contains nonzero vectors of $K_0$, which forces that the kernel of $\phi$ must be zero and  hence $\phi$ is an isomorphism.

By Lemma \ref{lem4.4}(v), we see that
$d_{r_S-1}'$ acts injectively on   $K$.


As $\DD$-modules,
$$\Ind_{\DD^{(0,-n_S)}}^{\DD}K_0\cong\Ind_{\DD^{(0,-\infty)}}^{\DD}(\Ind_{\DD^{(0,-n_S)}}^{(0,-\infty)}K_0)\cong\Ind_{\DD^{(0,-\infty)}}^{\DD}K.$$
And we further have $\Ind_{\DD^{(0,-\infty)}}^{\DD}K\cong \Ind_{\Vir'^{(0)}}^{\Vir'}K$ as vector spaces. Moreover, we have the following $\DD$-module epimorphism
\begin{eqnarray*}
\pi: \Ind_{\DD^{(0,-\infty)}}^{\DD}K=\Ind_{\Vir'^{(0)}}^{\Vir'}K&\rightarrow& S,\cr
\sum_{\bfl\in\mathbb{M}}d'^{\bfl}\otimes v_{\bfl}&\mapsto& \sum_{\bfl\in\mathbb{M}}d'^{\bfl} v_{\bfl},
\end{eqnarray*}
where $d'^{\bfl}=\cdots (d'_{-2})^{l_2}(d'_{-1})^{l_1}$.
We see that $\pi$ is also a $\Vir'$-module epimorphism. By the proof of Theorem 2.1 in \cite{MZ2} we know that any nonzero $\Vir'$-submodule of $\Ind_{\Vir'^{(0)}}^{\Vir'}K$ contain nonzero vectors of $K$. Note that $\pi|_K$ is one-to-one, we see  that the image of any nonzero $\DD$-submodule ( and hence $\Vir'$-submodule ) of $\Ind_{\DD^{(0,-\infty)}}^{\DD}K$  must be a nonzero $\DD$-submodule of $S$ and hence be the  whole module $S$, which forces that the kernel of $\pi$ must be  $ 0$. Therefore, $\pi$ is an isomorphism. Since $S$ is simple, we see  $K_0$ is a simple $\DD^{(0,-n_S)}$-module.\end{proof}

As a direct consequence of Proposition \ref{prop4.6}, we have
\begin{cor}\label{case 2}
Let $S$ be a simple restricted $\DD$-module  with level $\ell \not=0$. If $m_S\leq 1$ and $n_S=1$,
then $K_0$ is a simple $\DD^{(0,-1)}$-module and $S\cong \Ind_{\DD^{(0,-1)}}^{\DD}K_0$.
\end{cor}

\begin{proof}
For any nonzero $u\in U_0$, since  $m_S\leq 1$ and $n_S=1$, it follows from the definitions of $m_S, n_S$ and Lemma \ref{lem4.2}(i) that
$$d_1u=0,\, L_1u=\frac{1}{2\ell}\sum_{k\in\bZ+\frac{1}{2}}h_{1-k}h_ku=\frac{1}{2\ell}(h_{\frac{1}{2}})^2u\neq 0.$$
This implies that $d_1^{\prime}u\neq 0$, i.e., $d_1^{\prime}$ acts injectively on $U_0$. Hence $r_S\geq 2$. More precisely, since
$$d_{2+i}v=L_{2+i}v=0,\,\,\forall\,i\in\bN, v\in U_0,$$
we see that $r_S=2$. Now the desired assertion follows directly from Proposition \ref{prop4.6}.
\end{proof}

\begin{rem}
From Corollary \ref{case 2}, we have dealt with the Case (2).
\end{rem}

What remains to consider for Case (1) is  that $ m_S=2n_S\ge2$ and $r_S\le 1$.
In this case we will show that $K$ is a simple $\mh$-module.

For the Verma module $M_{\Vir}(c,h)$ over $\Vir$,
it is well-known from \cite{A, FF} that there exist two homogeneous elements $P_1, P_2\in \UU(\Vir^-)\Vir^-$
such that $ \UU(\Vir^-)P_1w_1+ \UU(\Vir^-)P_2w_1$ is the unique maximal proper $\Vir$-submodule of $M_{\Vir}(c,h)$,
where $P_1, P_2$ are allowed to  be zero and $w_1$ is the
highest weight vector in $M_{\Vir}(c,h)$.

\begin{lem}\label{lem4.5'}
Let $d=0,-1$. Suppose $M$ is a $\Vir^{(d)}$-module on which $d_0$ acts as multiplication by a given scalar $\lambda$. Then there exists a unique maximal submodule $N$ of ${\rm Ind}^{\Vir}_{\Vir^{(d)}}M$ with $N\cap M=0$. More precisely, $N$ is generated by $P_1M$ and $P_2M$, i.e., $N= \UU(\Vir^-)(P_1M+P_2M)$.
\end{lem}

\begin{proof}
Note that $d_0$ acts diagonalizably on ${\rm Ind}^{\Vir}_{\Vir^{(d)}}M$ and its submodules, and  $$M=\{u\in{\rm Ind}^{\Vir}_{\Vir^{(d)}}M\mid d_0u=\lambda u\},$$
i.e., $M$ is the highest weight space of ${\rm Ind}^{\Vir}_{\Vir^{(d)}}M$.
 Let $N$ be the sum of all $\Vir$-submodules of ${\rm Ind}^{\Vir}_{\Vir^{(d)}}M$ which intersect with $M$ trivially. Then $N$ is the desired unique maximal $\Vir$-submodule of ${\rm Ind}^{\Vir}_{\Vir^{(d)}}M$ with $N\cap M=0$.

Let $N^{\prime}$ be the $\Vir$-submodule generated by $P_1M$ and $P_2M$, i.e., $N^{\prime}= \UU(\Vir^-)(P_1M+P_2M)$. Then $N^{\prime}\cap M=0$. Hence, $N^{\prime}\subseteq N$. Suppose there is a proper submodule $U$ of ${\rm Ind}^{\Vir}_{\Vir^{(d)}}M$ that is not contained in $N^{\prime}$. There is a nonzero homogeneous $v=\sum _{i=1}^ru_iv_i\in U\setminus N^{\prime}$ where $u_i\in  \UU(\Vir^-)$ and $v_1,...v_r\in M$ are linearly independent. Note that all $u_i$ have the same weight. Then some $u_iv_i\notin N^{\prime}$, say $u_1v_1\notin N^{\prime}$.
There is a  homogeneous $u\in  \UU(\Vir)$ such that $uu_1v_1=v_1$.
Noting that all $uu_i$ has weight $0$, so $uu_iv_i\in \bC v_i$. Thus $uv\in M\setminus\{0\}.$ This implies that $N\subseteq N^{\prime}$. Hence, $N=N^{\prime}$, as desired.
\end{proof}

\begin{pro}\label{pro4.10}
Let $S$ be a simple restricted $\DD$-module  with level $\ell \not=0$.   If  $   m_S=2n_s\ge2$,  and $r_S=0$ or $ -1$, then
$K$ is a simple $\mh$-module and
$S\cong U^{\DD}\otimes K^{\DD}$ for some  simple $U\in \mathcal{R}_{\Vir}$.
\end{pro}

\begin{proof} By Lemma \ref{lem4.4} (iii), we see that $K^{\DD}$ is a $\DD$-module, and hence $K^{\DD'}$ is a $\DD'$-module with $d_n' K=0$ for any $n\in\bZ$. Let $\bC v_0$ be a one-dimensional $\DD'^{(r_S,-\infty)}$-module with module structure defining by
$d'_nv_0=h_{k+\frac{1}{2}}v_0={\bf c}_2v_0=0, n\ge r_S, k\in\bZ, {\bf c}'_1v_0=(c-2)v_0.$ Then $\bC v_0\otimes K^{\DD'} $ is a $\DD'^{(r_S,-\infty)}$-module with central charge $c-1$ and level $\ell$. It is easy to see that we have the following $\DD'^{(r_S,-\infty)}$-module homomorphism
\begin{eqnarray*}
\tau_{K}: \bC v_0\otimes K^{\DD'}& \longrightarrow& S,\cr
 v_0\otimes u&\mapsto& u, \forall u\in K.
\end{eqnarray*}
Clearly, $\tau_K$ is an injective map and can be extended to a $\DD'$-module epimorphism
\begin{eqnarray*}
\tau:\Ind_{{\DD'}^{(r_S,-\infty)}}^{\DD'}(\bC v_0\otimes K^{\DD'})&\longrightarrow& S,\cr
 x(v_0\otimes u)&\mapsto& xu, x\in  \UU(\DD'), u\in K.
\end{eqnarray*}
By Lemma 8 in \cite{LZ3} we know that
$$\Ind_{\DD'^{(r_S,-\infty)}}^{\DD'}(\bC v_0\otimes K^{\DD'})\cong (\Ind_{\DD'^{(r_S,-\infty)}}^{\DD'}\bC v_0)\otimes K^{\DD'}
=(\Ind_{\Vir'^{(r_S)}}^{\Vir'}\bC v_0)^{\DD'}\otimes K^{\DD'}.
$$Then we have the following $\DD'$-module epimorphism
\begin{eqnarray*}
\tau':(\Ind_{\Vir'^{(r_S)}}^{\Vir'}\bC v_0)^{\DD'}\otimes K^{\DD'}&\longrightarrow& S,\cr
  xv_0\otimes u&\mapsto& xu, x\in  \UU(\Vir'), u\in K.
\end{eqnarray*}
Note that $(\Ind_{\Vir'^{(r_S)}}^{\Vir'}\bC v_0)^{\DD'}\otimes K^{\DD'}\cong\Ind_{\Vir'^{(r_S)}}^{\Vir'}(\bC v_0\otimes K^{\DD'})$ as $\Vir'$-modules, and $\tau'$ is also a $\Vir'$-module epimorphism, $\tau'|_{\bC v_0\otimes K^{\DD'}}$ is one-to-one, and $(\Ind_{\Vir'^{(r_S)}}^{\Vir'}\bC v_0)^{\DD'}\otimes K^{\DD'}$ is a highest weight $\Vir'$-module.

Let $V=\Ind_{\Vir'^{(r_S)}}^{\Vir'}\bC v_0$ and  $\mathfrak{K}=\text{Ker}(\tau')$. It should be noted that
$$\bC v_0\otimes K^{\DD'}=\{u\in V^{\DD'}\otimes K^{\DD'}\mid d_0^{\prime}u=0\}.$$
We see that $(\bC v_0\otimes K^{\DD'})\cap \mathfrak{K}=0$. Let $\mathfrak{K}^{\prime}$  be the sum of all $\Vir'$-submodules $W$ of $V^{\DD'}\otimes K^{\DD'}$ with $(\bC v_0\otimes K^{\DD'})\cap W=0$, that is, the unique maximal $\Vir'$-submodule of $V^{\DD'}\otimes K^{\DD'}$ with trivial intersection with  $(\bC v_0\otimes K^{\DD'})$. It is obvious that $\mathfrak{K}\subseteq\mathfrak{K}'$. Next we further show that $\mathfrak{K}=\mathfrak{K}'$. For that, take any $\Vir'$-submodule $W$ of $V^{\DD'}\otimes K^{\DD'}$ such that $(\bC v_0\otimes K^{\DD'})\cap W=0$. Then for any weight vector $w=\sum_{\bfl\in \mathbb{M}}d'^{\bfl}v_0\otimes u_{\bfl}\in W$, where $u_{\bfl}\in K^{\DD'}, d'^{\bfl}=\cdots (d'_{-2})^{l_2}(d'_{-1})^{l_1}$ if $r_S=0$, or $d'^{\bfl}=\cdots (d'_{-2})^{l_2}$ if $r_S=-1$, and all ${\rm{w}}(\bfl)\ge 1$  are equal. Note that
$h_{k+\frac{1}{2}}w=\sum_{\bfl\in \mathbb{M}}d'^{\bfl}v_0\otimes h_{k+\frac{1}{2}}u_{\bfl}$  either equals to $0$ or has the same weight as $w$ under the action of $d_0^{\prime}$. So $U(\DD')\mathfrak{K}'\cap  (\bC v_0\otimes K^{\DD'})=0$. The maximality of $\mathfrak{K}'$ forces that $\mathfrak{K}'=U(\DD')\mathfrak{K}'$ is a  proper $\DD'$-submodule of $V^{\DD'}\otimes K^{\DD'}$. Since $\mathfrak{K}$ is a maximal proper $\DD'$-submodule of $V^{\DD'}\otimes K^{\DD'}$, it follows that $\mathfrak{K}=\mathfrak{K}'$.

By Lemma \ref{lem4.5'} we know that $\mathfrak{K}$ is generated by
$P_1( \bC v_0\otimes K^{\DD'})=\bC P_1 v_0\otimes K^{\DD'}$ and $P_2 (\bC v_0\otimes K^{\DD'})=\bC P_2 v_0\otimes K^{\DD'}$. Let  $V'$ be the  maximal  submodule of $V$ generated by $P_1v_0$ and $P_2v_0$, then $\mathfrak{K}=V'^{\DD'}\otimes K^{\DD'}$. Therefore,
$$S\cong (V^{\DD'}\otimes K^{\DD'})/(V'^{\DD'}\otimes K^{\DD'})\cong (V/V')^{\DD'}\otimes K^{\DD'},$$
which forces  that $K^{\DD'}$ is a simple $\DD'$-module and hence a simple $\mh$-module. So $S$ contains a simple $\mh$-module $K$. By Corollary \ref{tensor} we know there exists a simple $\Vir$-module $U\in {\mathcal{R}}_{\Vir}$ such that $S\cong U^{\DD}\otimes K^{\DD}$, as desired.
\end{proof}

\begin{lem}\label{lem4.11}
Let $M$ be a $\Vir^{(0)}$-module on which  $\Vir^{(1)}$ acts  trivially. If any finitely generated $\bC[d_0]$-submodule of $M$ is a free $\bC[d_0]$-module, then any nonzero submodule  of ${\rm Ind}^{\Vir}_{\Vir^{(0)}}M$  intersects with $M$ non-trivially.
\end{lem}
\begin{proof} Let $V$ be a nonzero submodule of ${\rm Ind}^{\Vir}_{\Vir^{(0)}}M$.
Take a nonzero $u\in V$. If $u\in M$, there is nothing to do. Now assume $u\in V\backslash M$. Write
$u=\sum_{i=1}^n a_iu_i$ where $a_i\in  \UU(\Vir_{\le 0})$, $u_i\in M$. Since $M_1=\sum_{1\le i\le n}\bC[d_0]u_i$ ( a $\Vir^{(0)}$-submodule of $M$ ) is a finitely generated $\bC[d_0]$-module, we see $M_1$ is a free module over $\bC[d_0]$ by the assumption. Without loss of generality, we may assume that $M_1=\oplus_{1\le i\le n}\bC [d_0]u_i$ with basis  $u_1,\cdots,u_n$ over $\bC[d_0]$. Note that each $a_i$ can be expressed as a sum of eigenvalue subspaces of $\ad\, d_0$ for $1\leq i\leq n$. Assume that $a_1$ has a maximal eigenvalue among all $a_i$ for $1\leq i\leq n$. Then $a_1u_1\notin M$. For any  $\lambda\in\bC$, let $M_1(\lambda)$ be  the $\bC[d_0]$-submodule  of $M_1$  generated by  $u_2, u_3,\cdots, u_n, d_0u_1-\lambda u_1$. Then $M_1/M_1(\lambda)$ is a one-dimensional
$\Vir^{(0)}$-module with $d_0(u_1+M_1(\lambda))=\lambda u_1+M_1(\lambda)$. By the Verma module theory for Virasoro algebra we know that there exists some $0\ne \lambda_0\in \bC$ such that the corresponding Verma module $\mathfrak{V}={\rm Ind}_{\Vir^{(0)}}^{\Vir}(M_1/M_1(\lambda_0))$
is irreducible. We know that $u=a_1u_1\ne0$ in $\mathfrak{V}$.
Hence we can find a  homogeneous $w\in  \UU(\Vir^+)$ such that $wa_1u_1=f_1(d_0)u_1$ in ${\rm Ind}^{\Vir}_{\Vir^{(0)}}M$, where $0\neq f_1(d_0)\in\bC[d_0]$.
So $wu=\sum_{i=1}^n wa_iu_i=\sum_{i=1}^n f_i(d_0)u_i $ for $f_i(d_0)\in\bC[d_0]$, $1\leq i\leq n$. Therefore, $0\ne wu\in V\cap M_1\subset V\cap M,$ as desired.
\end{proof}

\begin{pro}\label{eigenvalue prop} Let $S$ be a
simple restricted $\DD$-module  with level $\ell \not=0$.   If $m_S=2n_s\ge2$, $r_S=1$, then $d_0'$ has an  eigenvector in $K$.
\end{pro}

\begin{proof}
Suppose first that any finitely generated $\bC[d'_0]$-submodule of $K={\rm Ind}_{\mh^{(-n_S)}}^{\mh}K_0$ is a free $\bC[d'_0]$-module. By Lemma \ref{lem4.11} we see that
the following $\DD'$-module homomorphism
\begin{eqnarray*}
\tau:\Ind_{\DD'^{(0,-\infty)}}^{\DD'}K=\Ind_{\Vir'^{(0)}}^{\Vir'}K&\longrightarrow & S, \cr
 x\otimes u&\mapsto &xu, x\in  \UU(\Vir'), u\in K.
\end{eqnarray*}
is an isomorphism. So
$S={\rm Ind}_{\Vir'^{(0)}}^{\Vir'}K$, and consequently, $K$ is an irreducible
$\DD'^{(0,-\infty)}$-module. Since $\Vir'^{(1)}K=0$, we consider $K$ as an irreducible module over the Lie algebra $\mh\oplus \bC d'_0$. Since $d'_0$ is the center of the Lie algebra $\mh\oplus \bC d'_0$, we see that  the action of $d'_0$ on $K$ is a scalar,  a contradiction. So this case does not occur.

Now   there exists some  finitely generated $\bC[d'_0]$-submodule $M$ of $K$ that is not a free $\bC[d'_0]$-module. Since $\bC[d'_0]$ is a principal ideal domain,  by the structure theorem of finitely generated modules over a principal ideal domain,   there exists a monic polynomial $f(d'_0)\in\bC[d'_0]$  with positive degree  and nonzero element $u\in M$ such that $f(d'_0)u=0$. Furthermore, we can write $f(d'_0)=(d_0'-\lambda_1) (d_0'-\lambda_2)\cdots (d_0'-\lambda_p)$ for some $\lambda_1,\cdots,\lambda_p\in\bC$. Then there exists some $s\leq p$ such that $w:=\prod_{i=s+1}^p(d_0'-\lambda_{j})u\neq 0$ and $d_0'w=\lambda_sw$, where we make convention that $w=u$ if $s=p$. Then $w$ is a desired eigenvector of $d_0'$.
\end{proof}

\begin{pro} \label{pro4.13} Let $S$ be a
simple restricted $\DD$-module  with level $\ell \not=0$. If $m_S=2n_s\ge2$, $r_S=1$, then $K$ is
 a simple $\mh$-module  and
$S\cong U^{\DD}\otimes K^{\DD}$ for some  simple $U\in \mathcal{R}_{\Vir}$.
\end{pro}

\begin{proof}
We see that  $S$ is a weight $\DD'$-module since  $S$ is a
simple  $\DD'$-module and $d'_0$  has an eigenvector. From Lemma \ref{lem4.4}(iii), $K$ and $K_0$ are weight $\DD'$-modules as well.
We can take  some $0\ne u_0\in K$ such that $d_0'u_0=\lambda u_0$ for some $\lambda \ne 0$ by Proposition \ref{eigenvalue prop}. Set $K'=U(\mh)u_0$, which is an $\mh$ submodule of $K$. Then we have the $\DD'$-module $K'^{\DD'}$, on which $\Vir'$ acts trivially by definition for any $n\in\bZ$. Let $\bC v_0$ be the one-dimensional $\DD'^{(0,-\infty)}$-module defined by $d'_0v_0=\lambda v_0, d_n'v_0=h_{k+\frac{1}{2}}v_0={\bf c}_2v_0=0, n\in\bZ_+,k\in\bZ$, ${\bf c}_1'v_0=(c-2)v_0$. Then $\bC v_0\otimes K'^{\DD'}$ is a $\DD'^{(0,-\infty)}$-module with central charge $c-1$ and level $\ell$.  There is a $\DD'^{(0,-\infty)}$-module homomorphism
\begin{eqnarray*}
\tau_{K'}: \bC v_0\otimes K'^{\DD'}& \longrightarrow & S,\cr
 v_0\otimes u &\mapsto & u, \forall u\in K',
\end{eqnarray*}
which is injective and can be extended to be the following $\DD'$-module homomorphism
\begin{eqnarray*}
\tau:\Ind_{{\DD'}^{(0,-\infty)}}^{\DD'}(\bC v_0\otimes K'^{\DD'})&\longrightarrow & S, \cr
 x(v_0\otimes u)&\mapsto &xu, x\in  \UU(\DD'), u\in K'.
\end{eqnarray*}
Since $S$ is a simple $\DD'$ module and $\tau\ne 0$, we see that $\tau$ is surjective.
By similar arguments in the proof of Proposition \ref{pro4.10}, we can obtain that $K'$ is a simple $\mh$-module.
By Corollary \ref{tensor} we know there exists a simple $\Vir$-module $U\in {\mathcal{R}}_{\Vir}$ such that $S\cong U^{\DD}\otimes K'^{\DD}$, as desired. Now it is clear that $K=K'$.
\end{proof}

We are now in a position to present the following main result on a classification of simple restricted $\DD$-modules with nonzero level.

\begin{theo}\label{mainthm}
Let $S$ be a simple restricted $\DD$-module  with level $\ell \not=0$. The invariants $m_S, n_S, r_S$ of $S$, $U_0, U(2), K_0, K$ are defined as before. Then one of the following cases occurs.

{\rm Case 1}: $n_S=0$.

In this case,  $S\cong H^{\mathfrak{D}}\otimes U^{\mathfrak{D}}$ as $\mathfrak{D}$-modules for some  simple modules $H\in \mathcal{R}_{\mh}$ and  $U\in \mathcal{R}_{\Vir}$.

{\rm Case 2}: $n_S>0$.

In this case, we further have the following three subcases.

{\rm Subcase 2.1}: $m_S>2n_S$.

In this subcase, $S\cong \Ind^{\DD}_{\DD^{(0,-n_S)}}(U_0)$.

{\rm Subcase 2.2}: $m_S=2n_S$.

In this subcase, we have
\begin{equation*}
S\cong\begin{cases} K^{\DD}, &{\text{ if }}r_S=-\infty,\cr
U^{\DD}\otimes K^{\DD}, &{\text{ if }}-1\leq r_S\leq 1, \cr
\Ind_{\DD^{(0,-n_S)}}^{\DD}K_0, &{\text { otherwise,
}}\end{cases}
\end{equation*}
where $U\in \mathcal{R}_{\Vir}$.

{\rm Subcase 2.3}: $m_S<2n_S$.

In this subcase, we have
\begin{equation*}
S\cong\begin{cases}  \Ind^{\DD}_{\DD^{(0,-(m_S-n_S))}}(U_0),\vspace{1.5mm} &{\text{ if }}m_S\geq 2,\cr
\Ind^{\DD}_{\DD^{(0,-(2-n_S))}}(U(2)),\vspace{1.5mm} &{\text{ if }} m_S<2, n_S>1, \cr
\Ind_{\DD^{(0,-1)}}^{\DD}K_0, &{\text { otherwise.
}}\end{cases}
\end{equation*}
\end{theo}

\begin{proof}
The assertion follows directly from Proposition \ref{prop4.3}, Proposition \ref{prop for -inf}, Proposition \ref{prop4.6}, Corollary \ref{case 2}, Proposition \ref{pro4.10} and Proposition \ref{pro4.13}.
\end{proof}

\begin{rem}By Theorems \ref{MT} and \ref{mainthm}, we know that  any simple restricted module $S$  is a highest weight $\Vir$-module with trivial action of $\mh$, or a tensor product of a simple restricted $\Vir$-module and a simple restricted $\mh$-module, or an induced module from some simple module $M$ over certain subalgebra of $\DD$.
Moreover, $M$ can be viewed as a simple module over some finite-dimensional solvable
Lie algebra. This reduces the study of such $\DD$-modules to the study of simple modules over the corresponding
finite-dimensional solvable Lie algebras.
\end{rem}

\section{Simple restricted $\bar {\DD}$-modules with nonzero level}\label{char}

In this section we will determine all simple  restricted ${\bar {\DD}}$-modules $M$ of level $\ell\ne0$.  The main method we will use is similar to the one used in  Section 4.

For a given simple restricted $\bar {\DD}$-module $M$ with level $\ell \not=0$, we define the following invariants of $M$ as follows:
$$M(r)={\text{Ann}}_M(\bar {\mh}^{(r)}), n_M=\min\{r\in \bZ:M(r)\ne0\}, M_0=M{(n_M)}.$$

\begin{lem}\label{lem4.2'} Let $M$ be an
irreducible restricted $\bar \DD$-module  with level $\ell \not=0$.
\begin{enumerate}
\item[\rm(i)] $n_M\in\bN$, and $h_{n_M-1}$ acts injectively on $M_0$.
\item[\rm(ii)] $M_0$ is a nonzero $\bar \DD^{(0,-(n_M-1))}$-module, and is invariant under the action of the operators $\bar{L}_n$  defined in (\ref{rep1}) for $n\in\bN$.
\end{enumerate}
\end{lem}

\begin{proof}

(i) Assume that $n_M<0$. Take any nonzero $v\in M_0$, we then have
$$h_{1}v=0=h_{-1}v.$$
This implies that $v=\frac{1}{\ell}[h_{1},h_{-1}]v=0$, a contradiction.  Hence, $n_M\in\bN$.

The definition of $n_M$ means that $h_{n_M-1}$ acts injectively on $M_0$.

(ii) It is obvious that $M_0\neq 0$ by definition. For any $w\in M_0$, $i, j, k\in\bN$, we have
$$h_{k+n_M}d_iw=d_ih_{k+n_M}w+(k+n_M)h_{i+k+n_M}w=0,$$
and
$$h_{k+n_M}h_{j-n_M+1}w=h_{j-n_M+1}h_{k+n_M}w=0.$$
Hence,  $d_iw, h_{j-n_M+1}w\in M_0$, i.e., $M_0$ is a nonzero $\DD^{(0,-(n_M-1))}$-module.

For $i, n\in\bN$, $w\in M_0$, noticing $n_M\ge 0$ by (i), it follows from (\ref{rep3}) that
\begin{eqnarray*}
h_{i+n_M}\bar L_nw&=\Big(\bar L_nh_{i+n_M}+(i+n_M)h_{n+i+n_M}\Big)w=0.
\end{eqnarray*}
This implies that $\bar L_nw\in M_0$ for $n\in\bN$, that is, $M_0$ is invariant under the action of the operators $\bar  L_n$ for $n\in\bN$.
\end{proof}

\begin{pro}\label{prop4.3'} Let $M$ be a simple restricted $\bar \DD$-module  with level $\ell \not=0$.
 If $n_M=0, 1$, then  $M\cong H^{\bar \DD}\otimes U^{\bar \DD}$ as $\bar \DD$-modules for some  simple modules $H\in \mathcal{R}_{\bar \mh}$ and  $U\in \mathcal{R}_{\Vir}$.
\end{pro}
\begin{proof}
Since  $n_M=0, 1$, we  take any nonzero $v\in M_0$. Then $\bC v$ is a   $\bar \mh^{(0)}$-module. Let $H=\UU(\bar \mh)v$, the $\bar \mh$-submodule of $M$ generated by $v$. It follows from representation theory of Heisenberg algebras  that $\Ind^{\bar \mh}_{\bar \mh^{(0)}}(\bC v)$ is a simple $\bar \mh$-module.  Consequently, the following surjective $\bar \mh$-module homomorphism
\begin{eqnarray*}
\varphi:\, \Ind^{\bar \mh}_{\bar \mh^{(0)}}(\bC v) &\longrightarrow & H\\
\sum_{\mi\in\mathbb{M}}a_{\mi} h^{\mi}\otimes v&\mapsto &
\sum_{\mi\in\mathbb{M}} a_{\mi} h^{\mi} v
\end{eqnarray*}
is an isomorphism, that is, $H$ is a simple $\bar \mh$-module, which is certainly restricted. Then the desired assertion follows directly from  \cite[Theorem 12]{LZ3}.
\end{proof}

Next we assume that $n_M\ge 2$.

We define the operators $d_n'=d_n-\bar L_n$ on $M$ for $n\in\bZ$.  Since $M$ is a restricted $\bar  \DD$-module, then $d_n'$ is well-defined for any $n\in\bZ$. By (\ref{rep3}) and (\ref{rep3'}), we have
\begin{equation}\label{vir-bracket'}
[d_m',{\bar \c}'_1]=0,
[d_m',d_n']=(m-n)d_{m+n}'+\frac{m^3-m}{12}\delta_{m+n,0}{\bar \c}'_1, m,n\in\bZ,
\end{equation}where  ${\bar \c}'_1=c-(1-\frac{12z^2}{\ell})\text{id}_M$ and $c$ is the central charge of $M$. So the operator algebra
$$\Vir'=\bigoplus_{n\in\bZ}\bC d_n'\oplus \bC{\bar \c}'_1$$
 is isomorphic to the Virasoro algebra $\Vir$.  Since $[d_n,h_{k}]=[\bar L_n,h_{k}]=-kh_{n+k}+\delta_{n+k,0}(n^2+n){\bar \c}_2,$ we have
 \begin{equation}\label{d'bracket}
 [d'_n,h_{k}]=0, n, k\in\bZ\end{equation}
  and hence  $[\Vir',\bar \mh+\bC {\bar \c}_2]=0$. Clearly, the operator algebra
 $\bar \DD'=\Vir'\oplus  (\bar \mh+\bC {\bar \c}_2)$ is a direct sum, and $M=\UU(\bar \DD)v=\UU(\bar \DD')v$ for any $ v\in M\setminus\{0\}$. Let
$$Y_n=\bigcap_{p\ge n}{\rm Ann}_{M_0}(d_p'),       r_M=\min\{n\in\bZ:Y_n\ne0\}, K_0=Y_{r_M}. $$
Noting that $M$ is a restricted $\bar \DD$-module, we know that $r_M<+\infty$. If $Y_n\ne0$ for any $n\in\bZ$, we define $r_M=-\infty$.
Denote by $K=\UU(\bar \mh)K_0$.

\begin{lem}\label{lem4.4'}
Let $M$ be a simple restricted $\bar \DD$-module  with level $\ell \not=0$. Then the following statements hold.
\begin{enumerate}
\item[\rm(i)] $ r_M\ge -1$ or $r_M=-\infty$.
\item[\rm(ii)] If $r_M\ge -1$, then $K_0$ is a $\bar \DD^{(0,-(n_M-1))}$-module  and $h_{n_M-1}$ acts injectively on $K_0$.
\item[\rm(iii)] $K$ is a $\bar \DD^{(0,-\infty)}$-module and $K(z)^{\bar \DD}$ has a $\bar \DD$-module structure by (\ref{rep1})-(\ref{rep2}).
\item[\rm{(iv)}] $K_0$ and $K$ are invariant under the actions of $L_n$ and $d_n'$ for $n\in\bN$.
\item[\rm(v)] If $r_M\ne -\infty$, then $d'_{r_M-1}$ acts injectively on $K_0$ and $K$.
\end{enumerate}
\end{lem}

\begin{proof} (i) If $Y_{-2}\ne 0$, then $d'_{p}K_0=0, p\ge -2$. We deduce that $\Vir' K_0=0$ and hence $r_M=-\infty$.

If $Y_{-2}=0$, then $r_M\ge -1$.

(ii) For any $0\ne v\in K_0$ and $x\in {\bar \DD^{(0,-(n_M-1))}}$, it follows from Lemma \ref{lem4.2'}(ii) that $xv\in M_0$. We need to show that $d'_pxv=0, p\ge r_M$. Indeed, $d_p'h_{k}v=h_{k}d_p'v=0$ by (\ref{d'bracket}) for any $k\geq -(n_M-1)$. Moreover, it follows from (\ref{rep3'}) and  (\ref{vir-bracket'}) that
$$d_p'd_nv=d_nd_p'v+[d_p', d_n]v=(n-p)d_{p+n}'v=0, \forall n\in\bN.$$
Hence, $d'_pxv=0, p\ge r_M$, that is, $xv\in K_0$, as desired.

Since $0\ne K_0  \subseteq M_0$, we see that  $h_{n_M-1}$ acts injectively on $K_0$ by Lemma \ref{lem4.2'}(i).

(iii) follows from (ii).

(iv) Note that if $n_M=0$, then $\bar L_nK_0=0$ for any $n\in\bN$. For $n_M>0$ we compute  that  $$\bar L_n=\frac{1}{2\ell}\sum_{k\in\bZ}:h_{n-k}h_k:+\frac{(n+1)z}{\ell}h_n=\frac{1}{2\ell}\sum_{-(n_M-1)\le k\le n_M-1}:h_{n-k}h_k:+\frac{(n+1)z}{\ell}h_n, n\in\bN.$$
We see $\bar L_nK_0\subset K_0$ and $\bar L_nK\subset K$ by (ii), and hence $d'_nK_0\subset K_0$ and $d'_nK\subset K$.


(v) follows directly from the definitions of $r_M$ and $K$.
\end{proof}

We first consider the case $r_M=-\infty$.

\begin{pro}\label{pro3.5'}
Let $M$ be a
simple restricted $\bar \DD$-module with central charge $c$ and level $\ell \not=0$. If $r_M=-\infty$, then $M= K(z)^{\bar \DD}$ for some $z\in\bC$. Hence $c=1-\frac{12z^2}{\ell}$
 and $K$ is a simple $\bar \mh$-module.
\end{pro}
\begin{proof}
Since $r_M=-\infty$, we see that $\Vir' K_0=0$. This together with (\ref{vir-bracket'}) implies that $c=1-\frac{12z^2}{\ell}$.  Noting that $[\Vir',\bar  \mh+\bC{ \bar \c}_2]=0$, we further obtain that $\Vir' K=0$, that is,  $d_nv=\bar L_nv\in K$ for any $v\in K$ and $n\in\bZ$. Hence $K(z)^{\bar \DD}$ is a $\bar \DD$-submodule of $M$, yielding that $M= K(z)^{\bar \DD}$. In particular, $K$ is a simple $\bar \mh$-module.
\end{proof}

\begin{pro}\label{prop4.6'}
Let $M$ be a simple restricted $\bar \DD$-module  with level $\ell \not=0$. If $r_M\ge 2$ and $n_M\ge 2$, then
$K_0$ is a simple $\bar \DD^{(0,-(n_M-1))}$-module and
$M\cong \Ind_{\bar \DD^{(0,-(n_M-1))}}^{\bar \DD}K_0$.
\end{pro}

\begin{proof}We first show that  $\Ind_{\bar \DD^{(0,-(n_M-1))}}^{\bar \DD^{(0,-\infty)}}K_0 \cong K$ as  $\bar \DD^{(0,-\infty)}$ modules. For that, let
\begin{eqnarray*}
\phi:\, \Ind_{\bar \DD^{(0,-(n_M-1))}}^{\bar \DD^{(0,-\infty)}}K_0  &\longrightarrow &K\\
\sum_{\bfk\in\mathbb{M}} h^{\bfk}\otimes v_{\bfk}&\mapsto &
\sum_{\bfk\in\mathbb{M}} h^{\bfk} v_{\bfk},
\end{eqnarray*} where $h^{\bfk}=\cdots h^{k_2}_{-2-(n_M-1)}h_{-1-(n_M-1)}^{k_1}\in\UU(\bar \mh)$. Then $\phi$ is a  $\bar \DD^{(0,-\infty)}$-module epimorphism and $\phi|_{K_0}$ is one-to-one.

{\bf Claim}. Any nonzero submodule $V$ of $\Ind_{\bar \DD^{(0,-(n_M-1))}}^{\bar \DD^{(0,-\infty)}}K_0$ does not  intersect with $K_0$ trivially.

Assume $V\cap K_0=0$. Let $v=\sum_{\bfk\in\mathbb{M}} h^{\bfk}\otimes v_{\bfk}\in V\backslash K_0 $ with  minimal degree $\bfi$. Then $\bf 0\prec \bfi$.

Let $p=\text{min}\{s:i_s\ne 0\}$. Since  $h_{p+n_M-1}v_{\bfk}=0$, we have $h_{p+n_M-1}h^{\bfk}v_{\bfk}=[h_{p+n_M-1},h^{\bfk}]v_{\bfk}$.
The following equality
\begin{equation*}
[h_{i},h_{j_1}h_{j_2}\cdots h_{j_t}]=\sum_{1\le s\le t}\delta_{i+j_s,0}i {\bar \c}_3 h_{j_1}\cdots \hat{h}_{j_{s}}\cdots h_{j_t}, \,i, j_1\le j_2\le\cdots\le j_t\in\bZ
\end{equation*} implies that if  $k_p=0$ then $h_{p+n_M-1}h^{\bfk}v_{\bfk}=0$; and if $k_p\ne 0$, noticing
the level $\ell\ne 0$,
then $[h_{p+n},h^{\bfk}]=\lambda h^{\bfk-\epsilon_p}$ for
some  $\lambda\in\bC^*$ and hence $$\text{deg}([h_{p+n_M-1},h^{\bfk}]v_{\bfk})=\bfk-\epsilon_p\preceq \bfi-\epsilon_p,$$
where the equality holds if and only if $\bfk=\bfi$. Hence $\deg(h_{p+n_M-1}v)=\bfi-\epsilon_p\prec \bfi$ and $h_{p+n_M-1}v\in V$,   contrary to the choice of $v$. Thus, the claim holds.

From the Claim we know that
the kernel of $\phi$ must be zero and  hence $\phi$ is an isomorphism.

By Lemma \ref{lem4.4'}(v), we see that
$d_{r_M-1}'$ acts injectively on   $K$.

As $\bar \DD$-modules,
$$\Ind_{\bar \DD^{(0,-(n_M-1))}}^{\bar \DD}K_0\cong\Ind_{\bar \DD^{(0,-\infty)}}^{\bar \DD}(\Ind_{\bar \DD^{(0,-(n_M-1))}}^{\bar \DD^{(0,-\infty)}}K_0)\cong\Ind_{\bar \DD^{(0,-\infty)}}^{\bar \DD}K.$$
And we further have $\Ind_{\bar \DD^{(0,-\infty)}}^{\bar \DD}K\cong \Ind_{\Vir'^{(0)}}^{\Vir'}K$ as vector spaces. Moreover, we have the following $\bar \DD$-module epimorphism
\begin{eqnarray*}
\pi: \Ind_{\bar \DD^{(0,-\infty)}}^{\bar \DD}K=\Ind_{\Vir'^{(0)}}^{\Vir'}K&\rightarrow& M,\cr
\sum_{\bfl\in\mathbb{M}}d'^{\bfl}\otimes v_{\bfl}&\mapsto& \sum_{\bfl\in\mathbb{M}}d'^{\bfl} v_{\bfl},
\end{eqnarray*}
where $d'^{\bfl}=\cdots (d'_{-2})^{l_2}(d'_{-1})^{l_1}$.
We see that $\pi$ is also a $\Vir'$-module epimorphism. By the proof of Theorem 2.1 in \cite{MZ2} we know that any nonzero $\Vir'$-submodule of $\Ind_{\Vir'^{(0)}}^{\Vir'}K$ contain nonzero vectors of $K$. Note that $\pi|_K$ is one-to-one, we see  that the image of any nonzero $\bar \DD$-submodule (and hence $\Vir'$-submodule) of $\Ind_{\bar \DD^{(0,-\infty)}}^{\bar \DD}K$  must be a nonzero $\bar \DD$-submodule of $M$ and hence be the  whole module $M$, which forces that the kernel of $\pi$ must be  $ 0$. Therefore, $\pi$ is an isomorphism. Since $M$ is simple, we see  $K_0$ is a simple $\bar \DD^{(0,-(n_M-1))}$-module.\end{proof}

\begin{pro}\label{eigenvalue prop'} Let $M$ be a simple restricted $\bar \DD$-module  with level $\ell \not=0$.   If  $r_M=1$, then $d_0'$ has an  eigenvector in $K$.
\end{pro}

\begin{proof}Lemma \ref{lem4.4'} (iv) means that $K$ is a $\bar \DD'^{(0,-\infty)}$-module.
Assume that any finitely generated $\bC[d'_0]$-submodule of $K$ is a free $\bC[d'_0]$-module. By Lemma \ref{lem4.11} we see that
the following $\bar \DD'$-module homomorphism
\begin{eqnarray*}
\tau:\Ind_{\bar \DD'^{(0,-\infty)}}^{{\bar \DD}'}K=\Ind_{\Vir'^{(0)}}^{\Vir'}K&\longrightarrow & M, \cr
 x\otimes u&\mapsto &xu, x\in \UU(\Vir'), u\in K.
\end{eqnarray*}
is an isomorphism. So
$M={\rm Ind}_{\Vir'^{(0)}}^{\Vir'}K$, and consequently, $K$ is a simple
$\bar \DD'^{(0,-\infty)}$-module. Since $r_M=1$ and $\Vir'^{(1)}K=0$,  $K$ can be seen as a simple module over the Lie algebra $\mh\oplus \bC\c_2\oplus\bC d'_0$ where $\bC d'_0$ lies in the center of the Lie algebra. Schur's lemma tells us that $d'_0$ acts as a scalar  on $K$,  a contradiction. So this case will not occur.

Therefore,   there exists some  finitely generated $\bC[d'_0]$-submodule $W$ of $K$ that is not a free $\bC[d'_0]$-module. Since $\bC[d'_0]$ is a principal ideal domain,  by the structure theorem of finitely generated modules over a principal ideal domain,   there exists a monic polynomial $f(d'_0)\in\bC[d'_0]$  with minimal positive degree  and nonzero element $u\in W$ such that $f(d'_0)u=0$. Write $f(d'_0)=\Pi_{1\le i\le s}(d_0'-\lambda_i)$, $\lambda_1,\cdots,\lambda_s\in\bC$. Denote $w:=\prod_{i=1}^{s-1}(d_0'-\lambda_{i})u\neq 0$, we see $(d_0'-\lambda_s)w=0$  where we make convention that $w=u$ if $s=1$. Then $w$ is a desired eigenvector of $d_0'$.
\end{proof}

\begin{pro} \label{pro3.8} Let $M$ be a simple restricted $\bar \DD$-module  with level $\ell \not=0$.   If $r_M=0,\pm 1$, then $K$ is
 a simple $\mh$-module  and
$M \cong  K(z)^{\bar \DD}\otimes U^{\bar \DD}$ for some simple module $U\in \mathcal{R}_{\Vir}$ and    some $z\in\bC$.
\end{pro}

\begin{proof} If $r_M=1$, then by Proposition \ref{eigenvalue prop'} we know that there exists  $0\ne u\in K$ such that $d_0'u=\lambda u$ for some $\lambda \ne 0$; if $r_M=0,-1$, then $d_0'K=0$. In summary, for all  the three cases,  $d_0'$ has an eigenvector in $ K$.
Since $M$ is a simple $\bar \DD'$-module,  Schur's lemma implies that $h_0, {\bar \c}'_1, {\bar \c}_2, {\bar \c}_3$ act as scalars on $M$. So $M$ is a weight $\bar  \DD'$-module, and $K$ is a  weight module for $\bar \DD'^{(r_M-\delta_{r_{_M}, 1},-\infty)}$. Take a weight vector $u_0\in K$ with $d'_0u_0=\lambda_0u_0$ for some $\lambda_0\in\bC$.

Set $K'=\UU(\bar \mh)u_0$, which is an $\bar \mh$ submodule of $K$. Now we define  the $\bar \DD'$-module $K'^{{\bar \DD}'}$ with trivial action of  $\Vir'$. Let $\bC v_0$ be the one-dimensional $\bar \DD'^{(r_M-\delta_{r_{_M}, 1},-\infty)}$-module defined by $${\bar \c}_1'v_0=(c-1+\frac{12z^2}{\ell})v_0,\,\,\,
d'_0 v_0=\lambda_0v_0, \,\,\,d_n'v_0=h_{k}v_0={\bar\c}_2v_0={\bar\c}_3v_0=0, \,\,\, 0\ne n\ge r_M,k\in\bZ.$$ Then $\bC v_0\otimes K'^{{\bar \DD}'}$ is a $\bar \DD'^{(r_M-\delta_{r_{_M}, 1},-\infty)}$-module with central charge $c-1+\frac{12z^2}{\ell}$ and level $\ell$.  There is a $\bar \DD'^{(r_M-\delta_{r_{_M}, 1},-\infty)}$-module homomorphism
\begin{eqnarray*}
\tau_{K'}: \bC v_0\otimes K'^{{\bar \DD}'}& \longrightarrow & M,\cr
 v_0\otimes u &\mapsto & u, \forall u\in K',
\end{eqnarray*}
which is injective and can be extended to be the following $\bar \DD'$-module epimomorphism
\begin{eqnarray*}
\tau:\Ind_{{\bar \DD}'^{(r_M-\delta_{r_{_M}, 1},-\infty)}}^{\bar \DD'}(\bC v_0\otimes K'^{{\bar \DD}'})&\longrightarrow & M, \cr
 x(v_0\otimes u)&\mapsto &xu, x\in \UU(\bar \DD'), u\in K'.
\end{eqnarray*}
By Lemma 8 in \cite{LZ3} we know that
$$\Ind_{\bar \DD'^{(r_M-\delta_{r_{_M}, 1},-\infty)}}^{{\bar \DD}'}(\bC v_0\otimes K'^{{\bar \DD}'})\cong (\Ind_{\bar \DD'^{(r_M-\delta_{r_{_M}, 1},-\infty)}}^{{\bar \DD}'}\bC v_0)\otimes K'^{{\bar \DD}'}
=(\Ind_{\Vir'^{(r_M-\delta_{r_{_M}, 1})}}^{\Vir'}\bC v_0)^{{\bar \DD}'}\otimes K'^{{\bar \DD}'}.
$$Then we have the following $\bar \DD'$-module epimorphism
\begin{eqnarray*}
\tau':(\Ind_{\Vir'^{(r_M-\delta_{r_{_M}, 1})}}^{\Vir'}\bC v_0)^{{\bar \DD}'}\otimes K'^{{\bar \DD}'}&\longrightarrow& M,\cr
  xv_0\otimes u&\mapsto& xu, x\in \UU(\Vir'), u\in K'.
\end{eqnarray*}

Note that $(\Ind_{\Vir'^{(r_M-\delta_{r_{_M}, 1})}}^{\Vir'}\bC v_0)^{{\bar \DD}'}\otimes K'^{{\bar \DD}'}\cong\Ind_{\Vir'^{(r_M-\delta_{r_{_M}, 1})}}^{\Vir'}(\bC v_0\otimes K'^{{\bar \DD}'})$ as $\Vir'$-modules, and $\tau'$ is also a $\Vir'$-module epimorphism, $\tau'|_{\bC v_0\otimes K'^{{\bar \DD}'}}$ is one-to-one, and $(\Ind_{\Vir'^{(r_M-\delta_{r_{_M}, 1})}}^{\Vir'}\bC v_0)^{{\bar \DD}'}\otimes K'^{{\bar \DD}'}$ is a highest weight $\Vir'$-module.
Let $V=\Ind_{\Vir'^{(r_M-\delta_{r_{_M}, 1})}}^{\Vir'}\bC v_0$ and  $\mathfrak{K}=\text{Ker}(\tau')$. It should be noted that
$$\bC v_0\otimes K'^{{\bar \DD}'}=\{u\in V^{{\bar \DD}'}\otimes K'^{{\bar \DD}'}\mid d_0^{\prime}u=\lambda_0u\}.$$
We see that $(\bC v_0\otimes K'^{{\bar \DD}'})\cap \mathfrak{K}=0$. Let $\mathfrak{K}^{\prime}$  be the sum of all $\Vir'$-submodules $W$ of $V^{{\bar \DD}'}\otimes K'^{{\bar \DD}'}$
with $(\bC v_0\otimes K'^{{\bar \DD}'})\cap W=0$, that is, the unique maximal (weight) $\Vir'$-submodule of $V^{{\bar \DD}'}\otimes K'^{{\bar \DD}'}$ with trivial intersection with  $(\bC v_0\otimes K'^{{\bar \DD}'})$. It is obvious that $\mathfrak{K}\subseteq\mathfrak{K}'$.
Next we further show that $\mathfrak{K}=\mathfrak{K}'$. For that, take any $\Vir'$-
submodule $W$ of $V^{{\bar \DD}'}\otimes K'^{{\bar \DD}'}$ such that $(\bC v_0\otimes K'^{{\bar \DD}'})\cap W=0$. Then for any weight vector $w=\sum_{\bfl\in \mathbb{M}}
d'^{\bfl}v_0\otimes u_{\bfl}\in W$, where $u_{\bfl}\in {K'}^{{\bar \DD}'}, d'^{\bfl}=\cdots (d'_{-2})^{l_2}(d'_{-1})^{l_1}$ if $r_M=1,0$, or $d'^{\bfl}=\cdots (d'_{-2})^{l_2}$ if $r_M=-1$, and all ${\rm{w}}(\bfl)\ge 1$  are equal. Note that
$h_{k}w=\sum_{\bfl\in \mathbb{M}}d'^{\bfl}v_0\otimes h_{k}u_{\bfl}$  either
equals to $0$ or has the same weight as $w$ under the action of $d_0^{\prime}$.
So $\UU(\bar \DD')W\cap  (\bC v_0\otimes K'^{{\bar \DD}'})=0$, i.e., $\UU(\bar \DD')W\subset \mathfrak{K}'$. Hence $\UU(\bar \DD')\mathfrak{K}'\cap  (\bC v_0\otimes K'^{{\bar \DD}'})=0$. The maximality of $\mathfrak{K}'$ forces that $\mathfrak{K}'=\UU(\bar \DD')\mathfrak{K}'$ is a  proper $\bar \DD'$-submodule of $V^{{\bar \DD}'}\otimes K'^{{\bar \DD}'}$. Since $\mathfrak{K}$ is a maximal proper $\bar \DD'$-submodule of $V^{{\bar \DD}'}\otimes K'^{{\bar \DD}'}$, it follows that $\mathfrak{K}=\mathfrak{K}'$.

By Lemma \ref{lem4.5'} we know that $\mathfrak{K}$ is generated by
$P_1(\bC v_0\otimes K^{{\bar \DD}'})=\bC P_1 v_0\otimes K'^{{\bar \DD}'}$ and $P_2 (\bC v_0\otimes K'^{{\bar \DD}'})=\bC P_2 v_0\otimes K'^{{\bar \DD}'}$. Let  $V'$ be the  maximal  submodule of $V$ generated by $P_1v_0$ and $P_2v_0$, then $\mathfrak{K}=V'^{{\bar \DD}'}\otimes K'^{{\bar \DD}'}$. Therefore,
\begin{equation}\label{cong}
M\cong (V^{{\bar \DD}'}\otimes K'^{{\bar \DD}'})/(V'^{{\bar \DD}'}\otimes K'^{{\bar \DD}'})\cong (V/V')^{{\bar \DD}'}\otimes K'^{{\bar \DD}'},\end{equation}
which forces  that $K'^{{\bar \DD}'}$ is a simple $\bar \DD'$-module and hence a simple $\bar \mh$-module. So   $K'$ is  a simple $\bar \mh$-module. By  \cite[Theorem 12]{LZ3} we know there exists a simple $\Vir$-module $U\in {\mathcal{R}}_{\Vir}$ such that $M\cong K'^{\bar \DD} \otimes U^{\bar \DD} $. From this isomorphism and some computations we see that $K_0\subseteq K'^{\bar \DD}\otimes v_0$ where $v_0$ is a highest weight vector. So $K=K'$.\end{proof}

We are now in a position to present the following main result on characterization of simple restricted $\bar \DD$-modules with nonzero level.

\begin{theo}\label{mainthm'}
Let $M$ be a simple restricted $\bar \DD$-module  with level $\ell \not=0$. The invariants $ n_M, r_M$ of $M$, $ K_0, K$ are defined as before. Then
\begin{equation*}
M\cong\begin{cases} K(z)^{\bar \DD}, &{\text{ if }}r_M=-\infty,\cr
K(z)^{\bar \DD}\otimes U^{\bar \DD}, &{\text{ if }}-1\leq r_M\leq 1 \text{ or } n_M=0, 1, \cr
\Ind_{\bar \DD^{(0,-(n_M-1))}}^{\bar \DD}K_0, &{\text { otherwise,
}}\end{cases}
\end{equation*}
for  some $U\in \mathcal{R}_{\Vir}$ and    some $z\in\bC$.
\end{theo}

\begin{proof}
The assertion follows directly from Proposition \ref{prop4.3'}, Proposition \ref{pro3.5'}, Proposition \ref{prop4.6'}, Proposition \ref{pro3.8}.
\end{proof}


The following result characterizes simple Whittaker modules over the twisted Heisenberg-Virasoro algebra $\bar \DD$.

\begin{theo}\label{thmmain3.10}
Let $M$ be a $\bar \DD$-module
(not necessarily weight) on which the algebra ${\bar \DD}^{+}$ acts
locally finitely. Then the following statements hold.
\begin{enumerate}
\item[\rm(i)] The module $M$ contains a nonzero vector $v$ such that
${\bar \DD}^{+}\, v\subseteq\bC v$.
\item[\rm(ii)] If $M$ is simple, then $M$ is a Whittaker module or a highest weight module.
\end{enumerate}
\end{theo}

\begin{proof}
(i) Let $(M_1,\rho)$ be a finite dimensional ${\bar \DD}^{+}$-submodule of $M$. Then $M_1$ is also a finite dimensional $\Vir_{\geq 1}$-module. Let $\mathfrak{a}:=\ker(\rho|_{\Vir_{\geq 1}})$ be the kernel of the representation map of $\Vir_{\geq 1}$ on $M_1$. Then $\mathfrak{a}$ is an ideal of $\Vir_{\geq 1}$ of finite codimension. We claim that $d_n\in\mathfrak{a}$ for some $n\in\bZ_+$. If this is not true, then there exists a minimal $m\in\bZ_+$ such that $\mathfrak{a}$ contains an element of the form $a_{i_1}d_{i_1}+a_{i_2}d_{i_2}+\cdots+a_{i_{m+1}}d_{i_{m+1}}$ for positive integers $i_1<i_2<\cdots<i_{m+1}$ and nonzero complex numbers $a_{i_1}, a_{i_2},\cdots, a_{i_{m+1}}$. We further see that $\mathfrak{a}$ contains
$$[d_{i_1},a_{i_1}d_{i_1}+a_{i_2}d_{i_2}+\cdots+a_{i_{m+1}}d_{i_{m+1}}]=a_{i_2}(i_2-i_1)d_{i_1+i_2}+a_{i_3}(i_3-i_1)d_{i_1+i_3}+
\cdots+a_{i_{m+1}}(i_{m+1}-i_1)d_{i_1+i_{m+1}},$$
which contradicts with the minimality of $m$. Hence the claim follows. Consequently,
$$\widetilde{\Vir}_{\geq n}:=\sum_{i\geq n,\, i\neq 2n}\bC d_i=\bC d_n+ [d_n, \Vir_{\geq 1}]\subseteq\mathfrak{a}.$$
Then
$$\widetilde{\Vir}_{\geq n}+ {\bar \mh_{\geq n+1}}=\widetilde{\Vir}_{\geq n}+[ {\bar \mh_{\geq 1}},\widetilde{\Vir}_{\geq n}]\subseteq\ker(\rho).$$
This implies that $M_1$ is a finite dimensional module over a finite dimensional solvable Lie algebra $\bar {\mathfrak{D}}^{+}/(\widetilde{\Vir}_{\geq n}+ {\bar \mh_{\geq n+1}})$. The desired assertion follows directly from Lie Theorem.

(ii) follows directly from (i) and \cite{MZ2}.
\end{proof}

\begin{rem}From Theorem \ref{thmmain3.10} we know that if $M$ is a simple Whittaker module over $\bar \DD$ with nonzero level, and $\bar \DD^+ v\subset \bC v$ for some nonzero vector $v\in M$, then
$K=\UU(\bar \mh)v=\UU(\oplus_{r\in-\bZ_+}\bC h_r) v$ is a simple Whittaker module over $\bar \mh$. Therefore, \cite[Theorem 12]{LZ3} implies that $M\cong U^{\bar \DD}\otimes K(z)^{\bar \DD}$ for some  $U\in\R_{\Vir}$. Clearly, $U$ is a simple Whittaker module  or a simple highest weight module over $\Vir$.
\end{rem}

\section{Application one: characterization of simple highest weight modules and Whittaker modules over the mirror Heisenberg-Virasoro algebra}

Based on the results on structure of simple restricted modules over the mirror Heisenberg-Virasoro algebra $\DD$ given in Theorem \ref{MT} and Theorem \ref{mainthm}, we give characterization of simple highest weight $\DD$-modules and  simple Whittaker $\DD$-modules in this section.

We first have the following result characterizing simple highest weight modules over the mirror Heisenberg-Virasoro algebra.

\begin{theo}\label{thmmain}
Let $\DD$ be the mirror Heisenberg-Virasoro algebra  with the triangular decomposition
${{\mathfrak{D}}}={\mathfrak{D}}^{+}\oplus {\mathfrak{D}}^{0}\oplus {\mathfrak{D}}^{-}$.
Let $S$ be a $\DD$-module
(not necessarily weight) on which every
element in the algebra ${\mathfrak{D}}^{+}$ acts
locally nilpotently. Then the following statements hold.
\begin{enumerate}
\item[\rm(i)] The module $S$ contains a nonzero vector $v$ such that
${\mathfrak{D}}^{+}\, v=0$.
\item[\rm(ii)] If $S$ is simple, then $S$ is a highest weight module.
\end{enumerate}
\end{theo}

\begin{proof}
(i) It follows from \cite[Theorem 1]{MZ1} that there exists a nonzero vector $v\in S$ such that $d_iv=0$ for any $i\in\bZ_+$. If $h_{\frac{1}{2}}v=0$, then ${\mathfrak{D}}^{+}\, v=0$ as $d_1, d_2$ and $h_{\frac{1}{2}}$ generate ${\mathfrak{D}}^{+}$. Assume that $w:=h_{\frac{1}{2}}v\neq 0$. Then
$$d_1w=d_1h_{\frac{1}{2}}v=h_{\frac{1}{2}}d_1v+[d_1, h_{\frac{1}{2}}]v=-\frac{1}{2}h_{\frac{3}{2}}v.$$
Similar arguments yield that the element $d_1^jw=\lambda h_{j+\frac{1}{2}}v$ for some $\lambda\in\bC^*$ and $j\in\bZ_+$. As $d_1$ acts locally nilpotently on $S$, it follow that there exists some $n\in\bZ_+$ such that $h_{j+\frac{1}{2}}v=0$ for $j\geq n$.

We now show that for every $m\in\bN$ there exists some nonzero element $u\in S$ such that $d_iu=h_{k+\frac{1}{2}}u=0$ for $i\in\bZ_+$ and $k\geq m$ by a backward induction on $m$. The above arguments imply that the assertion is true for $m\geq n$. Assume that $0\neq u\in S$ satisfies that $d_iu=h_{k+\frac{1}{2}}u=0$ for $i\in\bZ_+$ and $k\geq m>0$. If $h_{m-\frac{1}{2}}u=0$, then the induction step is proved. Otherwise, $h_{m-\frac{1}{2}}u\neq 0$, and there exists some $l\in\bN$ such that $u^{\prime}:=h_{m-\frac{1}{2}}^lu\neq 0$ and $h_{m-\frac{1}{2}}u^{\prime}=h_{m-\frac{1}{2}}^{l+1}u=0$. Moreover,
$d_iu^{\prime}=h_{k+\frac{1}{2}}u^{\prime}=0$ for $i\in\bZ_+$ and $k\geq m-1$. The induction step follows.

(ii) By (i), we know that $S$ is a simple restricted $\DD$-module with $n_S=0$ and $ m_S\le 1$. From Theorem \ref{MT} and Case 1 of Theorem \ref{mainthm} we know that $S\cong H^{\mathfrak{D}}\otimes U^{\mathfrak{D}}$ as $\mathfrak{D}$-modules for some  simple modules $H\in \mathcal{R}_{\mh}$ and  $U\in \mathcal{R}_{\Vir}$.  Moreover,  $H=\Ind^{\mh}_{\mh^{(0)}}(\bC v)$ is a simple highest weight module over $\DD$. Note that every
element in the algebra $\Vir^{(1)}$  acts
locally nilpotently on $\bC v\otimes U$ by the assumption. This implies that the same property also holds on $U$. From \cite[Theorem 1]{MZ1} we know that $U$ is a simple highest weight $\Vir$-module. This completes the proof.
\end{proof}

As a direct consequence of Theorem \ref{thmmain}, we have
\begin{cor}
Let $S$ be an simple restricted $\DD$-module with $m_S\leq 1$ and $n_S=0$. Then $S$ is a highest weight module.
\end{cor}

\begin{proof}
The assumption that $m_S\leq 1$ and $n_S=0$ implies that there exists a nonzero vector $v\in M$ such that ${\mathfrak{D}}^{+}v=0$. Then $M=\UU({\mathfrak{D}}^{-}+{\mathfrak{D}}^{0})v$. It follows that each element in ${\mathfrak{D}}^{+}$ acts
locally nilpotently on $M$. Consequently, the desired assertion follows directly from Theorem \ref{thmmain}.
\end{proof}


The following result characterizes simple Whittaker modules over the mirror Heisenberg-Virasoro algebra.

\begin{theo}\label{thmmain17}
Let $M$ be a $\DD$-module
(not necessarily weight) on which the algebra ${\mathfrak{D}}^{+}$ acts
locally finitely. Then the following statements hold.
\begin{enumerate}
\item[\rm(i)] The module $M$ contains a nonzero vector $v$ such that
${\mathfrak{D}}^{+}\, v\subseteq\bC v$.
\item[\rm(ii)] If $M$ is simple, then $M$ is a Whittaker module or a highest weight module.
\end{enumerate}
\end{theo}

\begin{proof}
(i) Let $(M_1,\rho)$ be a finite dimensional ${\mathfrak{D}}^{+}$-submodule of $M$. Then $M_1$ is also a finite dimensional $\Vir_{\geq 1}$-module. Let $\mathfrak{a}:=\ker(\rho|_{\Vir_{\geq 1}})$ be the kernel of the representation map of $\Vir_{\geq 1}$ on $M_1$. Then $\mathfrak{a}$ is an ideal of $\Vir_{\geq 1}$ of finite codimension. We claim that $d_n\in\mathfrak{a}$ for some $n\in\bZ_+$. If this is not true, then there exists a minimal $m\in\bZ_+$ such that $\mathfrak{a}$ contains an element of the form $a_{i_1}d_{i_1}+a_{i_2}d_{i_2}+\cdots+a_{i_{m+1}}d_{i_{m+1}}$ for positive integers $i_1<i_2<\cdots<i_{m+1}$ and nonzero complex numbers $a_{i_1}, a_{i_2},\cdots, a_{i_{m+1}}$. We further see that $\mathfrak{a}$ contains
$$[d_{i_1},a_{i_1}d_{i_1}+a_{i_2}d_{i_2}+\cdots+a_{i_{m+1}}d_{i_{m+1}}]=a_{i_2}(i_1-i_2)d_{i_1+i_2}+a_{i_3}(i_1-i_3)d_{i_1+i_3}+
\cdots+a_{i_{m+1}}(i_1-i_{m+1})d_{i_1+i_{m+1}},$$
which contradicts with the minimality of $m$. Hence the claim follows. Consequently,
$$\widetilde{\Vir}_{\geq n}:=\sum_{i\geq n,\, i\neq 2n}\bC d_i=\bC d_n+ [d_n, \Vir_{\geq 1}]\subseteq\mathfrak{a}.$$
Then
$$\widetilde{\Vir}_{\geq n}+\mh_{\geq n}=\widetilde{\Vir}_{\geq n}+[\bC h_{\frac{1}{2}}+\bC h_{\frac{3}{2}},\widetilde{\Vir}_{\geq n}]\subseteq\ker(\rho).$$
This implies that $M_1$ is a finite dimensional module over a finite dimensional solvable Lie algebra ${\mathfrak{D}}^{+}/(\widetilde{\Vir}_{\geq n}+\mh_{\geq n})$. The desired assertion follows directly from Lie Theorem.

(ii) follows directly from (i).
\end{proof}

\section{Examples}
\label{examples}

In this section, we will give a few  examples of simple restricted $\bar  \DD$- and $\DD$-modules, which are also weak (simple) untwisted and twisted $\mathcal{V}^{c}$-modules.

\begin{exa} For any  $n\in\bZ_+$, let $\mathcal{W}_0=\bC[x_1,\cdots,x_n]$ be the polynomial algebra in indeterminates $x_1,\cdots,x_n$. Define the $\mh^{(-n)}$-module structure on $\mathcal{W}_0$ by
\begin{equation}\label{1stexample}
\begin{aligned}
&h_{i-\frac{1}{2}}\cdot f(x_1,\cdots,x_i,\cdots,x_n)=\lambda_if(x_1,\cdots,x_i-1,\cdots,x_n),\\
&h_{-i+\frac{1}{2}}\cdot f(x_1,\cdots,x_i,\cdots,x_n)=-\frac{\ell(i-\frac{1}{2})}{\lambda_i}(x_i+a_i)f(x_1,\cdots,x_i+1,\cdots,x_n),\\
&h_{n+j+\frac{1}{2}}\cdot f(x_1,\cdots,x_i,\cdots,x_n)=0, \\
&\c_2\cdot f(x_1,\cdots,x_n)=\ell f(x_1,\cdots,x_n)\end{aligned}
\end{equation}
where $\ell, \lambda_i\in\bC^*, a_i\in\bC, j\in\bN, 1\le i\le n$. It is not hard to  check that $\mathcal{W}_0$ is a simple $\mh^{(-n)}$-module. Then the induced $\mh$-module $K=\Ind_{\mh^{(-n)}}^{\mh}\mathcal{W}_0$ is a simple restricted $\mh$-module. So $K^{\DD}$ is a simple restricted $\DD$-module with central charge $1$ and level $\ell$. We may denote  $K^{\DD}=K^{\DD}(\ell, \Lambda_n,\mathfrak{a}_n)$ for any $\ell\in\bC^*,$
$\Lambda_n=(\lambda_1,\cdots,\lambda_n)\in (\bC^*)^n,$
$\mathfrak{a}_n=(a_1,\cdots,a_n)\in\bC^n$.

Let $U$ be a simple restricted $\Vir$-module (Theorem \ref{Vir-modules} classified all simple restricted $\Vir$-modules). From Corollary \ref{tensor}, then  $S=U^{\DD}\otimes K^{\DD}(\ell, \Lambda_n,\mathfrak{a}_n)$ is a simple restricted $\DD$-module.

If we replace (\ref{1stexample}) by
\begin{equation*}
\begin{aligned}
&h_{i}\cdot f(x_1,\cdots,x_i,\cdots,x_n)=\lambda_if(x_1,\cdots,x_i-1,\cdots,x_n),\\
&h_{-i}\cdot f(x_1,\cdots,x_i,\cdots,x_n)=-\frac{\ell i}{\lambda_i}(x_i+a_i)f(x_1,\cdots,x_i+1,\cdots,x_n),\\
&h_{n+j+1}\cdot f(x_1,\cdots,x_i,\cdots,x_n)=0, \\
&\bar\c_3\cdot f(x_1,\cdots,x_n)=\ell f(x_1,\cdots,x_n)\end{aligned}
\end{equation*} for $\ell, \lambda_i\in\bC^*, a_i\in\bC, j\in\bN, 1\le i\le n$, then $\mathcal{W}_0$ is a simple $\bar \mh^{(-n)}$-module, and the induced $\bar \mh$-module $\bar  K=\Ind_{\bar \mh^{(-n)}}^{\bar \mh}\mathcal{W}_0$ is a simple restricted $\bar \mh$-module. Hence, for any $z\in\bC$, we have the  simple $\bar \DD$-module $\bar  K(z)^{\bar \DD}=\bar  K(z)^{\bar \DD}(\ell, \Lambda_n,\mathfrak{a}_n)$ for any $\ell\in\bC^*,$
$\Lambda_n=(\lambda_1,\cdots,\lambda_n)\in (\bC^*)^n,$
$\mathfrak{a}_n=(a_1,\cdots,a_n)\in\bC^n$. For any simple $\Vir$-module $U\in\RR_{\Vir}$, the tensor product  $M=U^{\bar \DD}\otimes \bar  K(z)^{\bar \DD}(\ell, \Lambda_n,\mathfrak{a}_n)$ is a simple restricted $\bar \DD$-module.
\end{exa}

For characterizing simple induced restricted $\DD$-and $\bar {\DD}$-module which are not tensor product modules, we need the following

\begin{lem}\label{restricted submodule}
Let $S=U^{\DD}\otimes V^{\DD}$ be a simple restricted $\DD$-module with $n_S>0$ and nonzero level, where $U\in\R_{\Vir}$ and $V\in\R_{\mh}$. Let  $V_0=\text{Ann}_V(\mh^{(n_S)})$ and $W_0=\text{Ann}_S(\mh^{(n_S)})$. Then    $V_0$ is a simple $\DD^{(0,-n_S)}$-module, and
$W_0=U\otimes V_0$. Hence $W_0$ contains a simple $  \mh^{(-n_S)}$ submodule.
\end{lem}
\begin{proof} This is clear.
\end{proof}
We also have the $\bar \DD$-module version of Lemma \ref{restricted submodule}:
\begin{lem}\label{restricted submodule'}
Let $M=H(z)^{\bar \DD}\otimes U^{\bar \DD}$ be a simple restricted $\bar \DD$-module with $n_M>1$ and nonzero level, where   $z\in\bC$, $H\in\R_{\bar \mh}$ and $U\in\R_{\Vir}$.
Let  $H_0=\text{Ann}_H(\bar \mh^{(n_M)})$ and $M_0=\text{Ann}_M(\bar \mh^{(n_M)})$. Then    $H_0$ is a simple $\bar \DD^{(0,-n_M+1)}$-module, and
$M_0=H_0\otimes U$. Hence $M_0$ contains a simple $\bar \mh^{(-n_M+1)}$ submodule.
\end{lem}

Lemma \ref{restricted submodule} (resp. Lemma \ref{restricted submodule'})  means that if $S\in\R_{\DD}$ (resp. $M\in\R_{\bar \DD}$) is not a tensor product module, then $W_0$ (resp. $M_0$) contains no simple $\mh^{(-n_S)}$-submodule (resp. $\bar \mh^{(-n_M+1)}$-submodules).

Here we will first consider the case $n_S=1$ (resp. $n_M=2$).
Let $\mathfrak{b}=\bC h+\bC e$ be the 2-dimensional solvable Lie algebra with basis $h,e$ and subject to Lie bracket $[h,e]=e$.
%
The following concrete example using \cite[Example 13]{LMZ} tells us
 how to construct induced restricted $\DD$-module (resp. $\bar \DD$-module) from a $\bC[e]$-torsion-free simple $\mathfrak{b}$-module.

\begin{exa}[Simple induced restricted module, $n_S=1/n_M=2$] \label{ex-7.4}
Let $c_1,c_2\in\bC$ with $c_2\ne0$. Let $W'=(t-1)^{-1}\bC[t,t^{-1}]$. From \cite[Example 13]{LMZ} we know that  $W'$ is a simple $\mathfrak{b}$-module  whose structure  is given by
\begin{equation*}
h\cdot f(t)=t\frac{d}{dt}(f(t))+\frac{f(t)}{t^2(t-1)},\,\,
e\cdot f(t)=tf(t),\forall f(t)\in W'.
\end{equation*}
We can make $W'$ into a $\DD^{(0,0)}$-module by
$$\aligned  {\bf c}_1\cdot f(t)&=c_1  f(t),\,\,\, {\bf c}_2\cdot f(t)=c_2  f(t),\\
d_0\cdot f(t)&=-\frac{1}{2}h\cdot f(t), \,h_{\frac{1}{2}}\cdot f(t)=e\cdot f(t), \,d_i\cdot f(t)=h_{\frac{1}{2}+i}\cdot f(t)=0,\,\,\,i\in\bZ_+.\endaligned
$$
Then $W'$ is a simple $\DD^{(0,0)}$-module. Clearly, the action of $h_{\frac{1}{2}}$ on $W'$ implies that  $W'$ contains no simple $\mh^{(0)}$-module. Then $W_0=\Ind_{\DD^{(0,0)}}^{\DD^{(0,-1)}}W'$ is a simple $\DD^{(0,-1)}$-module and contains no simple $\mh^{(-1)}$-module. So $W_0$ is not a tensor product $\DD^{(0,-1)}$-module. Let $S=\Ind_{\DD^{(0,-1)}}^{\DD}W_0$. It is easy to see $n_S=1,m_S=2=r_S$, and $W_0=U_0=K_0$. The proof of  Proposition \ref{prop4.6} implies that $S$ is a simple restricted $\DD$-module. And Lemma \ref{restricted submodule}  means that  $S$ is not a tensor product $\DD$-module.

For $c,z,z'\in\bC, \ell\in\bC^*$, we also can make $W'$ into a $\bar \DD^{(0,0)}$-module by
$$\begin{aligned}&d_0\cdot f(t)=h\cdot f(t), \,h_{1}\cdot f(t)=e\cdot f(t), \\
&h_0\cdot f(t)=z'f(t), h_{1+i}\cdot f(t)=d_i\cdot f(t)=0,\,\,\,i\in\bZ_+,\\
& \bar\c_1\cdot f(t)=cf(t),\, \bar\c_2\cdot f(t)=zf(t),\, \bar\c_3\cdot f(t)=\ell f(t),\\
\end{aligned}
$$where $f(t)\in W'$.
Then $W'$ is a simple $\bar \DD^{(0,0)}$-module. Clearly, the action of $h_{1}$ on $W'$ implies that  $W'$ contains no simple $\bar \mh^{(0)}$-module. Then $M_0=\Ind_{\bar \DD^{(0,0)}}^{\bar \DD^{(0,-1)}}W'$ is a simple $\bar \DD^{(0,-1)}$-module and contains no simple $\bar \mh^{(-1)}$-module.  Let $M=\Ind_{\bar \DD^{(0,-1)}}^{\bar \DD}M_0$. It is easy to see $n_M=2,$$r_M= 3$. The proof of  Proposition \ref{prop4.6'} implies that $M$ is a simple restricted $\bar \DD$-module. And Lemma \ref{restricted submodule'}  means that  $M$ is not a tensor product $\bar \DD$-module.

\end{exa}

\begin{exa}[Simple induced modules of semi-Whittaker type, $n_S\ge 2, n_M\ge 3$]\footnote{This example is a modified version of the one provided by Drazen Adamovic.} \label{ex-7.6}   Take $p,q\in\bZ_+, {\bf a} =(a_1, \dots, a_{q}) \in  ({\bC}^*)^{q}$,  ${\bf b} = (b_1, \dots, b_p) \in  ({\bC}^*)^p$, $c,\ell\in\bC$ with $\ell\ne0$.  Define the $1$-dimensional
$\DD^{(p,q)}$-module ${\bC}_{{\bf a}, {\bf b}} = {\bC} v_0$ with
\begin{equation}\label{induced-modules} \aligned  {\bf c}_1\cdot v_0&=c  v_0,\,\,\, {\bf c}_2\cdot v_0=\ell  v_0,\\
 d_{p} v_0 &= a_1 v_0, \cdots, d_{p+q-1} v_0 = a_q v_0, \ d_i v_0 = 0 \ \mbox{for} \ i >  p+q-1, \\
 h_{q+\tfrac{1}{2}} v_0 &= b_1 v_0, \cdots   h_{p+q- \tfrac{1}{2}} v_0 = b_p v_0, \   h_{i-\tfrac{1}{2}} v_0 = 0 \ \mbox{for} \  i > p+q.\endaligned\end{equation}
It is not hard to show that  $U({\bf a},{\bf b}) :=  \Ind_{\DD^{(p,q)}}^{\DD^{(0,-1)}}{\bC}_{{\bf a},{\bf b}}$ is a  simple $\DD^{(0,-1)}$-module.
Then in  Theorem \ref{thmmain1.1} (2) we have $V=U({\bf a},{\bf b}), n=1, k=p+q=l$, and so $S=  \widehat  U({\bf a},{\bf b}) :=  \Ind_{\DD^{(0,-1)}}^{\DD } U({\bf a},{\bf b})$
is a simple restricted  $\DD$-module. In  Lemma \ref{restricted submodule},  $n_S=p+q$, and $W_0=\Ind_{\mh^{(q)}}^{\mh^{(-(p+q))}}\left(\Ind_{\DD^{(p,q)}}^{\DD^{(0,q)}} \bC_{{\bf a},{\bf b}}\right)$  does not contain any simple $\mh^{(-(p+q))}$-module ( for $h_{\pm 1/2}$ acts freely on $W_0$). Hence, by  Lemma \ref{restricted submodule},  $\widehat  U({\bf a},{\bf b})$ is not a  tensor product $\DD$-module.


If we, in the above example,  replace (\ref{induced-modules}) by
\begin{equation*}\label{induced module} \aligned & {\bar \c}_1\cdot v_0=c  v_0,\,{\bar \c}_2\cdot v_0=z  v_0,\,\bar\c_3 v_0=\ell v_0,\\
&d_{p} v_0 = a_1 v_0, \cdots, d_{p+q-1} v_0 = a_q v_0, \ d_i v_0 = 0 \ \mbox{for} \ i >  p+q-1, \\
& h_{q+1} v_0 = b_1 v_0, \cdots   h_{p+q} v_0 = b_p v_0, \   h_{i} v_0 = 0 \ \mbox{for} \  i > p+q,
\endaligned\end{equation*}where $z\in\bC$
 and leave other parts invariant, then for any $z'\in\bC$, the induced
 $\bar \DD^{(0,-(p+q))}$-module
$$\bar  V=\Ind_{\bar \DD^{(p,q+1)}}^{\bar \DD^{(0,-(p+q))}}  {\bC}_{\bf a,\bf b}/\Big(\UU( \bar \DD^{(0,-(p+q))})(h_0-z')(1\otimes v_0)\Big) $$
 is a  simple $\bar \DD^{(0,-(p+q))}$-module. Let $M=\Ind_{\bar \DD^{(0,-(p+q))}}^{\bar \DD}\bar  V$. The proof of Theorem \ref{prop4.6'} implies  that $M$ is a simple restricted  $\bar \DD$-module where
$n_M=p+q+1,r_M=2(p+q)+1$, and $K_0=\bar  V=M_0$. Since $\bar  V$ contains no simple $\bar \mh^{(-n_M+1)}$-module, we see, by Lemma \ref{restricted submodule'}, that $M$ is not a tensor product $\bar \DD$-module.
\end{exa}

 \begin{rem}From Theorem \ref{mainthm} (resp. Theorem \ref{prop4.3'}) we know that if $n_S=0$ (resp. $n_M=0,  1$), then  simple restricted $\DD$-modules(resp. $\bar \DD$-modules) must be tensor product modules. And Examples \ref{ex-7.4}-\ref{ex-7.6} mean that for any $n_S>0$ (resp. $n_M>1$), there do exist simple restricted $\DD$-modules (resp. $\bar \DD$-modules) which are not tensor product modules. Clearly, the $\bar \DD$-modules here are simple restricted $\widetilde\DD$-modules for $z=0$. \end{rem}

 \appendix

 \section{Application two: simple modules for Heisenberg-Virasoro vertex operator algebras $\mathcal V^{c}$ (by Drazen Adamovic)}\label{voasection}

\label{voa-interp}

A connection between restricted modules over the Heisenberg-Virasoro algebra and VOA modules in untwisted cases was considered by Guo and Wang in \cite{GW}. In this appendix we extend this correspondence for restricted modules for the mirror Heisenberg-Virasoro algebra. Restricted modules of nonzero level for the mirror Heisenberg-Virasoro   algebra can be treated as  weak twisted modules for the Heisenberg-Virasoro vertex algebras, and restricted modules  of nonzero level  for the twisted  Heisenberg-Virasoro   algebra can be treated as  weak modules for the Heisenberg-Virasoro vertex algebras. For convenience,  in the Heisenberg-Virasoro  algebra
 $\widetilde \DD$, we  denote $\c_1=\bar\c_1, \c_2=\bar \c_3$.

	Let $\mathcal P$ be the subalgebra of $\widetilde {\mathfrak{D}}$ spanned by
	$$ \c_1, \c_2, d_m, h_m, \quad (m \in {\bZ}_{\ge 0}). $$
	Let $(\ell_1, \ell_2) \in {\bC} ^2$. Consider the $1$-dimensional $\mathcal P$--module ${\bC} v_{\ell_1,\ell_2}$ such that
	$$ \c_1  v_{\ell_1,\ell_2} = \ell_1 v_{\ell_1,\ell_2},   \  \c_2  v_{\ell_1,\ell_2} = \ell_2 v_{\ell_1,\ell_2}, \  d_m v_{\ell_1,\ell_2}  = h_m v_{\ell_1,\ell_2} = 0 \quad (m \in {\bZ}_{\ge 0}). $$
	Let $\mathcal V^{\ell_1, \ell_2}$ be the following induced $\widetilde {\mathfrak{D}}$-module:
$$\mathcal V^{\ell_1, \ell_2} = U( \widetilde {\mathfrak{D}}) \otimes_{ U(\mathcal P)} {\bC} v_{\ell_1,\ell_2}. $$	
Then $\mathcal V^{\ell_1, \ell_2}$ is a highest weight $\widetilde {\mathfrak{D}}$-module, with the highest weight vector ${\bf 1}_{\ell_1, \ell_2}  =1 \otimes v_{\ell_1, \ell_2}$.
	
	Define the following fields acting on $\mathcal V^{\ell_1, \ell_2}$:
	$$ h(z) = \sum_{m \in {\bZ}} h_m z^{-m-1}, \quad d(z) = \sum_{m \in {\bZ}} d_m z^{-n-2}.$$
	\begin{itemize}
	\item Let $V_{Vir} ^c$ denotes the universal Virasoro vertex algebra of central charge $c$ generated by the Virasoro  field $L^{(1)} (z) = \sum_{m \in \bZ} L^{(1)}_n  z^{-n-2}$ (cf. \cite{LL}, \cite{FZ}).
	\item Let $ M(\ell) $ denotes the Heisenberg vertex algebra of level $\ell$, with the vertex operator $Y_{M(\ell)}( \cdot, z)$ which is uniquely  generated with the Heisenberg field $$ Y_{M(\ell)}(h(-1) {\bf 1}, z) =h(z) = \sum_{n \in {\Z}} h_n z^{-n-1}. $$ If $\ell \ne 0$, then $M(\ell)$ is simple and always isomorphic to $M(1)$ (cf. \cite{LL}).
	\item $M(\ell)$  contains a Virasoro vertex subalgebra $V_{Vir} ^{c=1}$ generated by the Virasoro field
	$$ L^{Heis} (z) := \frac{1}{2\ell} :h(z) h(z): = \sum_{n \in {\bZ}} L^{Heis} _n z^{-n-2}. $$
	Moreover, we have
	$$ L^{Heis}_n  = \frac{1}{2\ell} \sum_{k \in \bZ} :h_k h_{n-k}:$$
	For details see \cite{FLM}, \cite{LL}.
	
	\item $M(\ell)$ contains the Virasoro vector
	$$\omega_{\lambda }  = \frac{1}{2\ell} (h_{-1} ^2){\bf 1} - \frac{\lambda}{\ell} h_{-2} {\bf 1}$$
	of the central charge $c = 1 - 12 {\lambda} ^2 / \ell$. The components of the vertex operator
	\begin{align} Y_{M(\ell)} (\omega_{\lambda}) = \bar L(z) = L^{Heis}(z) - \frac{\lambda}{\ell} \partial_z  h(z) =  \sum _{n \in {\bZ}} \bar L_n z^{-n-2} \label{rep1-voa}\end{align}
	represent the operators $\bar L_n$ defined directly by (\ref{rep1}).
	\item Note that the field $\bar L(z)$ and operators $\bar L_n$ acts on any weak $M(\ell)$-module, and in particular on any restricted module for the Heisenberg algebra.
	\end{itemize}
	
	Note that $M(\ell_2)$ is naturally a vertex subalgebra of $\mathcal V^{\ell_1, \ell_2}$.
	
	\begin{pro} Assume that $\ell_2 \ne 0$. We have the following isomorphism of  vertex algebras
	$$  V_{Vir}^{ c } \otimes M(\ell_2) \cong   \mathcal V^{\ell_1, \ell_2},$$
	such that $c= \ell_1 -1$ and
	$$ L^{(1)}   (z) \mapsto d(z) - L^{Heis}(z), \quad h(z) \mapsto h(z). $$
	In particular, $\mathcal V^{\ell_1, \ell_2} \cong V_{Vir}^{ c } \otimes M(1)$.
	\end{pro}
	Since $M(\ell) \cong M(1)$, without loss of  generality we can assume that $\ell =1$.
	In what follows we set $\mathcal V^{c} := V_{Vir}^{ c } \otimes M(1)$.

	The following theorem relates restricted  $\bar \DD$-modules as (untwisted) modules for the vertex operator algebra  $\mathcal V^{c}$.
	
\begin{theo}\label{5.3-untw}(cf. \cite{GW})
The following statements holds.
\item[\rm(1)] Assume that $W$ is a (simple) restricted $\bar \DD$-module of central charge $\ell_1$ and level one. Then $W$ has the unique structure of a weak (simple)  $\mathcal V^{c=\ell_1-1}$-module generated by the fields:
$$ h (z) = \sum_{r \in   {\bZ}} h_r z^{-r-1},   \quad L^{(1)}(z) =   d(z) - L^{Heis}(z)  = \sum_{n \in {\bZ}} L^{(1)} _n z^{-n-2}. $$
 \item[\rm(2)]  Assume that $W$ is a  weak (simple)  $\mathcal V^{c}$-module. Then $W$  has the structure of a (simple) restricted $\bar \DD$-module of level one such that
 $$ h_n \mapsto h_n, \quad d_n = L^{(1)}_n + L^{Heis} _n   $$
 \end{theo}
	
	The vertex--algebraic interpretation of the restricted $\DD$--modules is via the twisted $\mathcal V^{c}$-modules.
	There is an automorphism $\theta_1$  of order two of $M(1)$ such that
$ \theta_{1} (h) = -h$ (cf. \cite{FLM}).  We extend this automorphism to the automorphism $\theta = \mbox{Id} \otimes \theta_1$ of  $\mathcal V^{c}$.

We have the following theorem.

\begin{theo}\label{5.3}
The following statements holds.
\item[\rm(1)] Assume that $W$ is a (simple) restricted $\DD$-module of central charge $\ell_1$ and level one. Then $W$ has the unique structure of a weak (simple) $\theta$-twisted $\mathcal V^{c=\ell_1-1}$-module generated by the fields:
$$ h^{tw}(z) = \sum_{r \in \tfrac{1}{2} + {\bZ}} h_r z^{-r-1},   L^{(1)} (z) =  d(z) -  L_{tw} ^{Heis} (z), $$
where
 \begin{align} L_{tw} ^{Heis} (z):=\tfrac{1}{2} : h^{tw} (z) ^2 : + \frac{1}{16} z^{-2} = \sum _{n\in {\bZ}} L _n z^{-n-2}.  \label{heis-tw-vir} \end{align}
 \item[\rm(2)]  Assume that $W$ is a  weak (simple) $\theta$-twisted $\mathcal V^{c}$-module. Then $W$ has the structure of    a (simple) restricted $\DD$-module of level one such that
 $$ h_r \mapsto h_r, \quad d_n = L^{(1)}_n + L_n. $$
 \end{theo}
 \begin{proof}
 As in \cite{FLM} (see also \cite{Tanabe}, \cite{HY}) we see that the field $ h^{tw}(z)$ defines on $W$ the unique structure of a $\theta_1$-twisted $M(1)$-module with the Virasoro field (\ref{heis-tw-vir}).
 Then we define $L^{(1)} (z) = d(z) - L _{tw}^{Heis}(z) = \sum_{n \in {\bZ}} L^{(1)}_n z^{-n-2}$.   The field $L^{(1)}(z)$ defines on $W$ the structure of a restricted  module for the Virasoro algebra of central charge $c= \ell_1 -1$. Since
 $$[L^{(1)} _n, h_r] = 0, \quad \forall n \in {\bZ}, \ r \in \tfrac{1}{2} + {\bZ}, $$
 we have that the action of $L^{(1)}(z)$ commutes with $h^{tw}(z)$. Therefore
 $W$ is a $\theta$-twisted $\mathcal V_{Vir}^{c}$-module.
 Since all components of the vertex operators are obtained from the action of $d_n, h_r$, $n \in {\bZ}, r \in \tfrac{1}{2} + {\bZ}$, we see that $W$ is an irreducible $\DD$-module if and only if $W$ is an irreducible module for the vertex algebra $\mathcal V^{c}$. This proves the assertion (1).

 Assume that $W$ is $\theta$-twisted $\mathcal V^{c}$-module with the vertex operator $Y^{tw} _W(v, z)$, $v \in \mathcal V^{c}$. Then the twisted Jacobi identity (cf. \cite{FLM}) shows that:
 \begin{itemize}
\item  The components of the field $$ L^{(1)} (z) = Y_{W} ^{tw} (L^{(1)} _{-2}{\bf 1}, z)= \sum_{n \in \bZ } L^{(1)} _n z^{-n-2}$$ define on $W$ structure of a restricted module for the Virasoro algebra with central charge $c$;
\item  The components of the field  $$h^{tw}(z) = Y_{W} ^{tw} (h_{-1} {\bf 1}, z)= \sum_{r  \in \tfrac{1}{2} + \bZ }  h_{r} z^{-r-1}$$ define on $W$ the structure of a restricted module for the Heisenberg algebra of level one.
 \item The fields $h^{tw}(z)$ and $L^{(1)}(z)$ commute.
 \item The field $$L_{tw}^{Heis}   (z) =Y(\tfrac{1}{2} h_{-1} ^2 {\bf 1}, z) = \sum_{n \in \bZ } L_n z^{-n-2}$$ is a Virasoro field commuting with $L^{(1)}(z)$ such that
 $$[L_n, h_{r}] = - r h_{n+r}, \quad n \in \bZ , r \in \tfrac{1}{2} + \bZ. $$
 Define $d(z) = L^{(1)}(z) + L_{tw}^{Heis}(z) = \sum_{n \in \bZ}  d_n z^{-n-2}$. Then the components of the fields $d(z)$ and $h^{tw}(z)$ define on $W$ the structure of a restricted $\DD$-module of central charge $\ell_1 = c+1$ and level one.
\end{itemize}

Arguments for the irreducibility are the same as in (1). This proves the assertion (2).
\end{proof}

\begin{rem} The simple modules in Examples \ref{ex-7.4}-\ref{ex-7.6} show that the vertex operator algebra $V_{Vir} ^{c} \otimes M(1)$ has   simple weak (untwisted and twisted) modules which are not isomorphic to any tensor product $S_1 \otimes S_2$ for any     simple weak  $V_{Vir} ^{c}$-module $S_1$ and any simple weak (untwisted and twisted) $M(1)$-module  $S_2$.  It would be interesting to find analogs of these non-tensor product  modules for other  types of  vertex operator algebras.\end{rem}

\begin{center}{\bf Acknowledgments}
\end{center}

The authors would like to thank Drazen Adamovic to provide Appendix A and Example 7.5, and help modifying the introduction. The last two authors  would  like to thank H. Chen and D. Gao for helpful discussions at the early stage  on the restricted modules over the twisted Heisenberg-Virasor algebra.

H.T. is partially supported by the Fundamental Research Funds for the Central Universities (135120008).
Y.Y. is partially supported by the National Natural Science Foundation of China (11771279, 12071136, 11671138).
K.Z. is partially supported by the  National Natural Science Foundation of China (11871190) and NSERC (311907-2020).

\end{document}